\documentclass[11pt]{amsart}
\usepackage{amscd,amssymb,epsfig, epsf,pinlabel,graphicx,yhmath,latexsym,graphicx}
\usepackage[dvipsnames]{xcolor}

\usepackage{hyperref}
\hypersetup{colorlinks=true,citecolor=brown,linkcolor=brown,urlcolor=black,filecolor=black}

\usepackage{scalerel,stackengine}

\usepackage[all]{xy}
\CompileMatrices

\begin{document}
%
%
%

%
%

\theoremstyle{plain}
\newtheorem{theorem}{Theorem}[section]
\newtheorem{lemma}[theorem]{Lemma}
\newtheorem{proposition}[theorem]{Proposition}
\newtheorem{corollary}[theorem]{Corollary}
\newtheorem{definition}[theorem]{Definition}

\theoremstyle{definition}
\newtheorem{example}[theorem]{Example}
\newtheorem{remark}[theorem]{Remark}
\newtheorem{summary}[theorem]{Summary}
\newtheorem{notation}[theorem]{Notation}
\newtheorem{problem}[theorem]{Problem}

\theoremstyle{remark}
\newtheorem{claim}[theorem]{Claim}
\newtheorem{sublemma}[theorem]{Sub-lemma}
\newtheorem{innerremark}[theorem]{Remark}

%
%

\numberwithin{equation}{section}
\numberwithin{figure}{section}
\allowdisplaybreaks
\renewcommand{\labelitemi}{$\centerdot$}

%
%

\def \N{\mathbb{N}_0}
\def \Z{\mathbb{Z}}
\def \Q{\mathbb{Q}}
\def \R{\mathbb{R}}
\def \K{\mathbb{K}}

\def \A{\mathcal{A}}

\def \ad{{\rm ad}}
\def \aug{{\rm aug}}
\def \Aut{{\rm Aut}}
\def \cl{{\rm cl}}
\def \Coker{{\rm Coker}}
\def \deg{{\rm deg}}
\def \ideg{{\rm i\hbox{-}deg}}
\def \edeg{{\rm e\hbox{-}deg}}
\def \End{{\rm End}}
\def \Gr{{\rm Gr}}
\def \Hom{{\rm Hom}}
\def \incl{{\rm incl}}
\def \id{{\rm id}}
\def \Img{{\rm Im}}
\def \Ker{{\rm Ker}}
\def \Lk{{\rm Lk}}
\def \mod{{\rm mod}}
\def \ord{{\rm ord}}
\def \pr{{\rm pr}}
\def \QDer{{\rm QDer}}
\def \rk{{\rm rk}}
\def \sgn{{\rm sgn}}
\def \sign{\pm}
\def \Tors{{\rm Tors}}
\def \ud{{\rm d}}
\def \Mat{\operatorname{Mat}}

\newcommand{\set}[1]{\lfloor #1\rceil}
\newcommand{\horiz}[1]{\displaystyle {#1}^{\rightleftharpoons}}
\newcommand{\ecrossing}{\begin{array}{c}\includegraphics[scale=0.06]{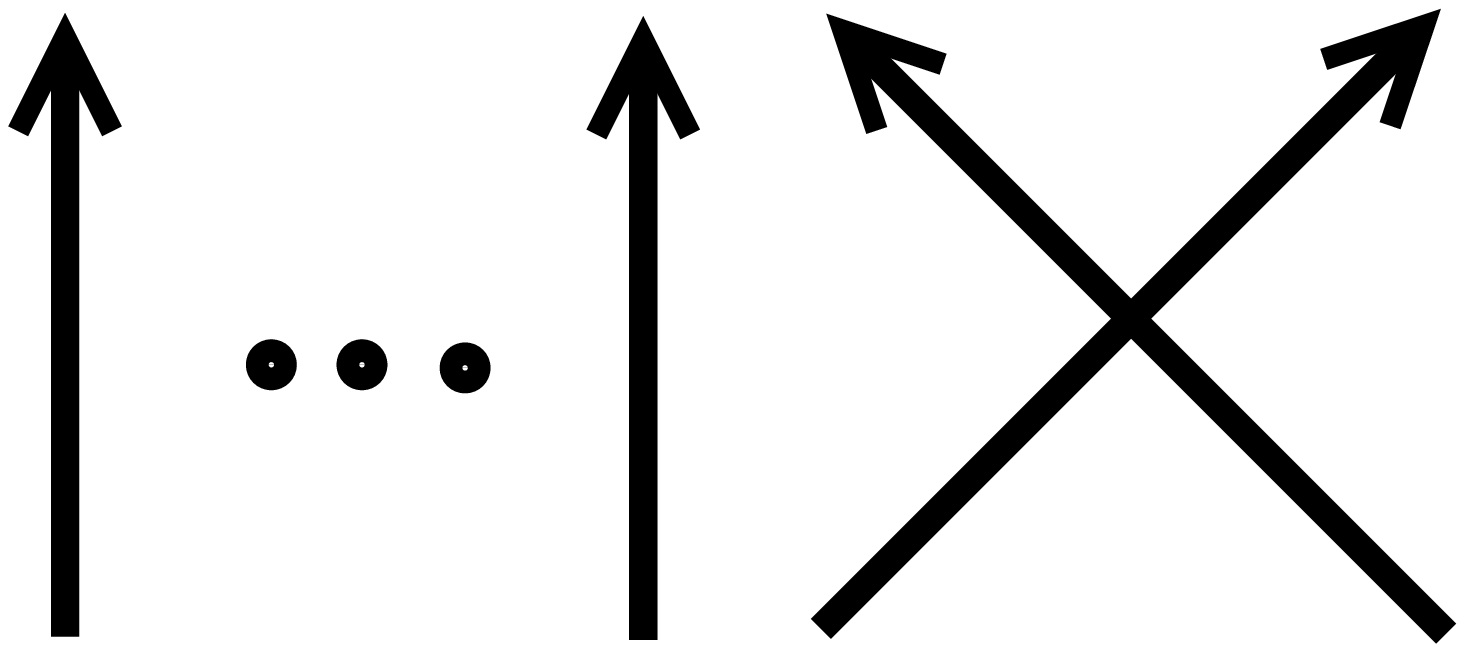}\end{array}}
\newcommand{\ichord}{\begin{array}{c}\includegraphics[scale=0.06]{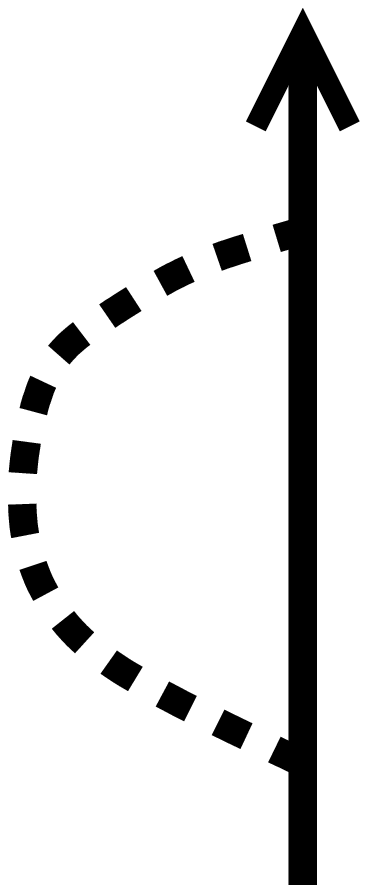}\end{array}}
\newcommand{\tstrut}[1]{\begin{array}{c}\labellist \scriptsize \hair 2pt  \pinlabel {$#1$}  [r] at 1 120
					\endlabellist  \includegraphics[scale=0.06]{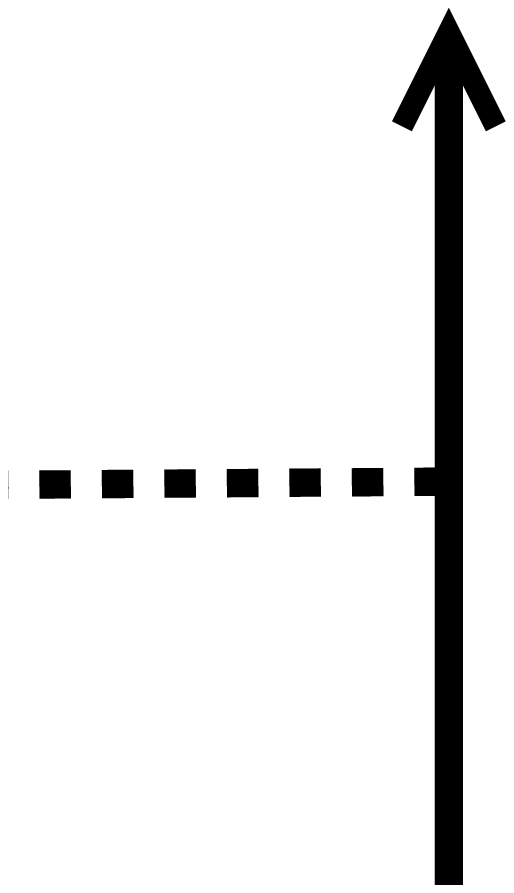} \end{array}}
\newcommand{\strutgraph}[2]{ \! \! \begin{array}{c}  \labellist \scriptsize \hair 2pt 
					\pinlabel {\scriptsize $#1$} [l] at 37 18
					\pinlabel {\scriptsize $#2$} [l] at 37 170
					\endlabellist \includegraphics[scale=0.1]{one-chord} \end{array}}
\newcommand{\figtotext}[3]{\begin{array}{c}\includegraphics[width=#1pt,height=#2pt]{#3}\end{array}}

\newcommand{\by}[1]{\stackrel{\eqref{#1}}{=}}
\newcommand{\up}{\vspace{-0.5cm}}

\stackMath
\newcommand\reallywidehat[1]{%
\savestack{\tmpbox}{\stretchto{%
  \scaleto{%
    \scalerel*[\widthof{\ensuremath{#1}}]{\kern-.6pt\bigwedge\kern-.6pt}%
    {\rule[-\textheight/2]{1ex}{\textheight}}
  }{\textheight}%
}{0.5ex}}%
\stackon[1pt]{#1}{\tmpbox}%
}

\title[]{\large  Formal descriptions of Turaev's loop operations}
\author[Gw\'ena\"el Massuyeau]{Gw\'ena\"el Massuyeau}
\address{Gw\'ena\"el Massuyeau \newline
\indent IRMA,    Universit\'e de Strasbourg \& CNRS \newline
\indent 7 rue Ren\'e Descartes \newline
\indent 67084 Strasbourg, France \newline
\indent $\mathtt{massuyeau@math.unistra.fr}$
}

\date{{July 13, 2016 (first version: November 12, 2015)}}

\begin{abstract}
Using intersection  and self-intersection of loops,
Turaev introduced in the seventies  two fundamental operations  on the algebra $\Q[\pi]$ of the  fundamental group  $\pi$ of a surface with boundary.
The first operation is binary and  measures the intersection of two oriented based curves on the surface, 
while the second operation is unary and computes the self-intersection of an oriented based curve.
It is already known that Turaev's intersection pairing has an algebraic description
when  the group algebra $\Q[\pi]$ is completed with respect to powers of its augmentation ideal 
and is appropriately identified to the degree-completion of the tensor algebra $T(H)$ of $H:=H_1(\pi;\Q)$.

In this paper, we obtain a similar algebraic description for  Turaev's self-intersection map in the case of a disk with $p$ punctures.
Here we consider the  identification between the completions of $\Q[\pi]$ and $T(H)$ 
that arises from a Drinfeld associator  by embedding $\pi$ into the pure braid group on $(p+1)$ strands;
 our algebraic description  involves a formal power series which is explicitly determined by the associator. 
The proof is based on some  three-dimensional formulas for  Turaev's loop operations,
which  involve $2$-strand pure braids  and are shown for any  surface with boundary.
\end{abstract}


\maketitle

\setcounter{tocdepth}{1} 
\tableofcontents

\section{Introduction}

Let $\Sigma$ be a connected oriented surface with non-empty boundary.
Let $\Z[\pi]$ be the group ring of the fundamental group $\pi:= \pi_1(\Sigma,\ast)$ based at a point $\ast \in \partial \Sigma$.
Turaev introduced in \cite{Tu_loops} two fundamental operations in $\Z[\pi]$.
On the one hand, there is the \emph{homotopy intersection pairing}
$$
\eta: \Z[\pi] \times \Z[\pi] \longrightarrow  \Z[\pi],
$$
which is  a version of  Reidemeister's equivariant intersection pairing in $\Sigma$. 
It appeared later that the pairing $\eta$ and the Hopf algebra structure of $\Z[\pi] $
 determine all  the loop operations that have been considered by Goldman
 in his study of the Poisson structures on the representation varieties of surface groups \cite{Go};
 these operations include the so-called ``Goldman bracket.''
The pairing $\eta$ has also been rediscovered in~\cite{Pe}, 
where it  is used  to study the unitarity property for the Magnus representation of mapping class groups 
 and the Burau/Gassner representations of braid groups.
 
On the other hand, there is the \emph{homotopy self-intersection map}, 
which is more naturally defined on the group ring of the fundamental group 
$\overrightarrow{\pi}:= \pi_1(U\Sigma,\vec \ast\,)$ of the unit tangent bundle of~$\Sigma$:
$$
\vec{\mu}: \Z[\overrightarrow{\pi}] \longrightarrow  \Z[\pi].
$$
Using the  Hopf algebra structures of $\Z[\overrightarrow{\pi}]$ and $\Z[\pi]$,
this map determines the so-called ``Turaev cobracket'' introduced in \cite{Tu_skein}, which constitutes a Lie bialgebra with the Goldman bracket.
The map $\vec{\mu}$ has also  been used in \cite{BP} to generalize Whitney's index formula to arbitrary surfaces.
It is legitimate to regard the homotopy self-intersection map   as a refinement of the homotopy intersection pairing, since we have
$$
\vec{\mu}\big(\vec{a}\, \vec{b}\big) = a\, \vec{\mu}( \vec{b}) + \vec{\mu}(\vec{a})\, b + \eta(a,b)
$$
for any  $\vec{a},\vec{b}\in \overrightarrow{\pi}$ mapping to $a,b\in \pi$ under the canonical projection $\overrightarrow{\pi} \to \pi$.

We now take rational coefficients and consider  the completion $\widehat{\Q[\pi]}$ of the group algebra $\Q[\pi]$ 
with respect to the filtration defined by powers of the augmentation ideal $I$. Since $\pi$ is a free group,
there exists an isomorphism of complete Hopf algebras 
\begin{equation} \label{eq:cha_iso}
 \widehat{\Q[\pi]} \longrightarrow T(\!(H)\!)
\end{equation} 
onto the degree-completion $T(\!(H)\!)$ of the tensor algebra $T(H)$ generated by $H:= H_1(\Sigma;\Q)$.
For instance, given a basis $(\zeta_i)_i$ of $\pi$, we can always consider the multiplicative map $\pi \to T(\!(H)\!)$ defined by
\begin{equation} \label{eq:basis_expansion}
\zeta_i \longmapsto \exp([\zeta_i]) \quad \hbox{where $[\zeta_i]\in H$ is the homology class of $\zeta_i$},
\end{equation}
and extend this map  by linearity/continuity to $\widehat{\Q[\pi]}$.
Of course, there is (generally speaking) no canonical isomorphism between the complete Hopf algebras $\widehat{\Q[\pi]}$ and $T(\!(H)\!)$, 
but we can always require that \eqref{eq:cha_iso} induces the canonical isomorphism at the graded level
 \begin{equation} \label{eq:canonical_alg_iso}
 \Gr\, \widehat{\Q[\pi]} \simeq \Gr\, {\Q[\pi]}  = \bigoplus_{k=0}^\infty I^k/I^{k+1}
 \stackrel{\simeq}{\longrightarrow} T(H)= \Gr\, T(H)\simeq    \Gr\, T(\!(H)\!)  
\end{equation}
defined in degree $1$ by $(x-1)\mapsto [x]$ for any $x\in \pi$.
Kawazumi and Kuno  proved in \cite{KK_twists} that, if the surface $\Sigma$ is compact with a single boundary component,
and if the isomorphism \eqref{eq:cha_iso} is ``symplectic'' in some sense \cite{Ma},
then the Lie algebra defined by the Goldman bracket corresponds  to the Lie algebra associated by Kontsevich to the symplectic vector space $(H,\omega)$ and to  the cyclic operad of associative algebras \cite{Kn}.
Note that Kontsevich's Lie algebra is also the ``necklace Lie algebra''  \cite{BLB,Gi} associated to a quiver with a single vertex and 
as many edges as the genus of $\Sigma$.
This ``formal description'' of the Goldman bracket has  been generalized in \cite{MT_twists} where it is shown that, under the same assumptions, 
the loop operation $\eta$   corresponds {through} the isomorphism \eqref{eq:cha_iso} to the sum of two simple operations in $T(\!(H)\!)$: 
the first tensor operation is defined by contraction with the homology intersection form $\omega:H \times H \to \Q$,
while the second tensor operation involves some multiplications ruled by Bernoulli numbers.

It is now natural to ask for similar ``formal descriptions'' of the homotopy self-intersection map and the Turaev cobracket.
This problem has been considered in the past few years by Kawazumi and Kuno and, independently, by Turaev and the author.
Indeed Schedler upgraded  the necklace Lie algebra to a Lie bialgebra by defining a quiver analogue of the Turaev cobracket~\cite{Sc_necklace}.
It turns out that, if $\Sigma$ is compact with a single boundary component  and if the isomorphism \eqref{eq:cha_iso} is {``symplectic'',}
then the Turaev cobracket corresponds to the Schedler cobracket \emph{at the graded level}.
This cobracket correspondence between surfaces and quivers was announced in \cite{MT_dim_2} and has been proved in \cite{MT_draft}.
Independently, Kawazumi and Kuno obtained a similar algebraic description of the graded level of the Turaev cobracket, which has been proved in \cite{KK_intersection}.
Their motivation in the study of mapping class groups 
 was to obtain new obstructions for the image of the Johnson homomorphisms 
by relating Morita's traces~\cite{Mo} to the Turaev cobracket: see the survey paper \cite{KK_survey}.
But some computations done by Kuno \cite{Ku} reveal that there exist some isomorphisms~\eqref{eq:cha_iso} satisfying the symplectic condition,
and for which the cobracket correspondence between surfaces and quivers  does not hold \emph{beyond} the graded level.

In this paper, we consider the simpler case where $\Sigma$ is a disk with finitely many punctures numbered from $1$ to $p$,
and where the isomorphism \eqref{eq:cha_iso} satisfies a ``special'' condition replacing the above-mentioned symplectic condition.
In this case too, there is a formal description of the homotopy intersection pairing 
which can be deduced from the case of a compact surface of genus $p$ with a single boundary component.
We  show that, if the isomorphism \eqref{eq:cha_iso} arises from the Kontsevich integral $Z$ by embedding $\pi$ into
 the group of $(p+1)$-strand pure braids  \cite{HgM,AET},
then the self-intersection map translates  into the sum of two simple operations in $T(\!(H)\!)$: 
the first tensor operation is canonical,
while the second tensor operation depends explicitly on the Drinfeld associator $\Phi$ underlying the construction of $Z$ (see Theorem \ref{th:vec_mu}). 
To~be more specific, this second tensor operation consists of  multiplications with a formal power series $\phi$
defined by some coefficients of $\Phi$, and it turns out that $\phi$ is essentially the $\Gamma$-function of  $\Phi$  defined by Enriquez \cite{En_Gamma}.
Besides, we prove that the second tensor operation does not depend anymore on $\Phi$  if the latter is assumed to be even
and, again, it reduces to some multiplications ruled by the Bernoulli numbers (see Corollary \ref{cor:even_case}).
Finally we obtain that, for any Drinfeld associator $\Phi$, the Turaev cobracket translates into the Schedler cobracket
associated to a star-shaped quiver consisting of one ``central'' vertex  connected by $p$ edges to $p$ ``peripheral'' vertices (see Corollary \ref{cor:delta}).

We mention that another formal description of the Turaev cobracket for a punctured disk 
has  been  obtained recently by Kawazumi \cite{Ka_genus_0}.
In his work, the isomorphism \eqref{eq:cha_iso} is  simply defined by \eqref{eq:basis_expansion}
using the basis $(\zeta_1,\dots,\zeta_p)$ of $\pi$ given by some small loops around the punctures.
But this isomorphism does not have the above-mentioned special property, 
so that the resulting formulas for the Turaev cobracket seem to be much more complicated than ours.
{Besides, around the same time when the first version of our paper was released, Alekseev, Kawazumi, Kuno and Naef  announced some results similar to ours: 
in their work  \cite{AKKN}, the use of Drinfeld associators is replaced by  solutions to the Kashiwara--Vergne problem.
It is expected that the results of \cite{AET}  make the link between the two approaches.}

To get our results for a punctured disk, we start by revisiting Turaev's loop operations. 
Specifically we prove some three-dimensional formulas for the operations $\eta$ and $\vec\mu$
which involve some embeddings of the groups $\pi$ and $\overrightarrow{\pi}$ into the group of pure braids in $\Sigma$.
Roughly speaking, for any $a,b\in \pi$, our formula for $\eta(a,b)$ is a certain ``free differential operator'' applied to the $2$-strand pure braid
$$
\labellist
\scriptsize\hair 2pt
 \pinlabel {$a^{-1}$} at 38 437
 \pinlabel {$b$} at 144 329
 \pinlabel {$a$}  at 36 221
 \pinlabel {$b^{-1}$} at 150 110
\endlabellist
\centering
\includegraphics[height=2.7cm,width=2.3cm]{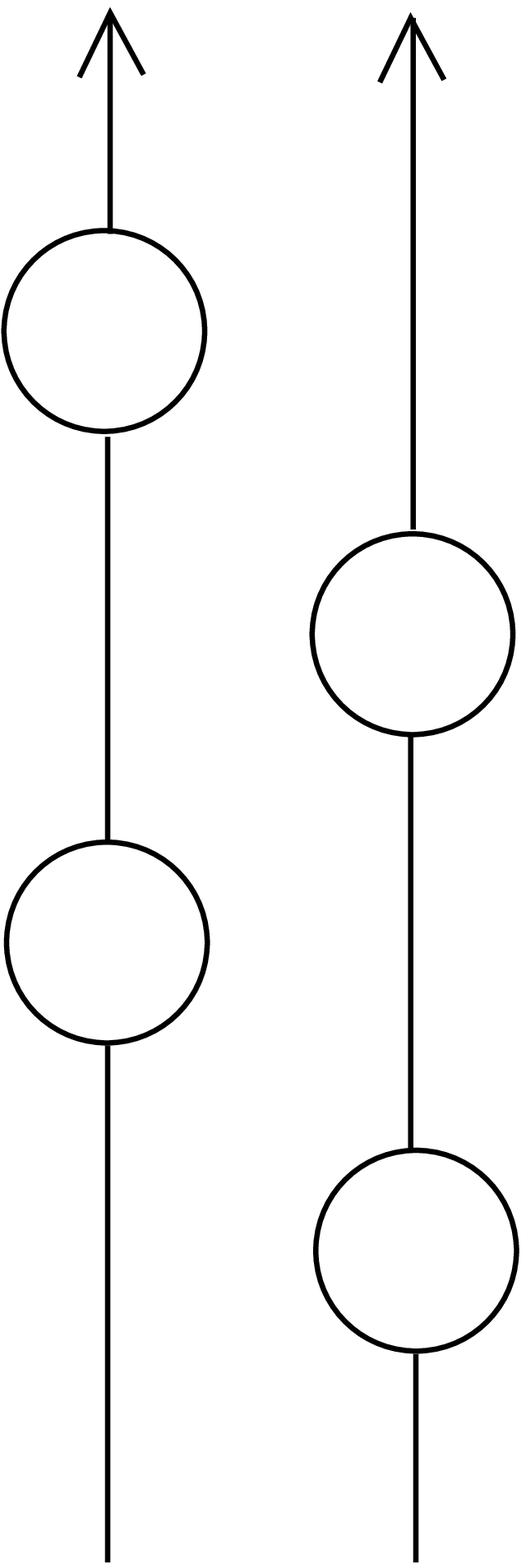}
$$
(see Theorem \ref{th:3d-eta}).
Similarly, for any $\vec{a} \in \overrightarrow{\pi}$ projecting to $a\in \pi$, our formula for $\vec\mu(\vec a)$
 is the same ``free differential operator'' applied to the $2$-strand pure braid
$$
\labellist
\scriptsize\hair 2pt
 \pinlabel {$\hbox{``doubling'' of {$\vec{a}$}}$}  at 109 232
 \pinlabel  {$a^{-1}$} at 170 118
 \pinlabel    {$a^{-1}$} at 61 367
\endlabellist
\centering
\includegraphics[height=2.6cm,width=2.6cm]{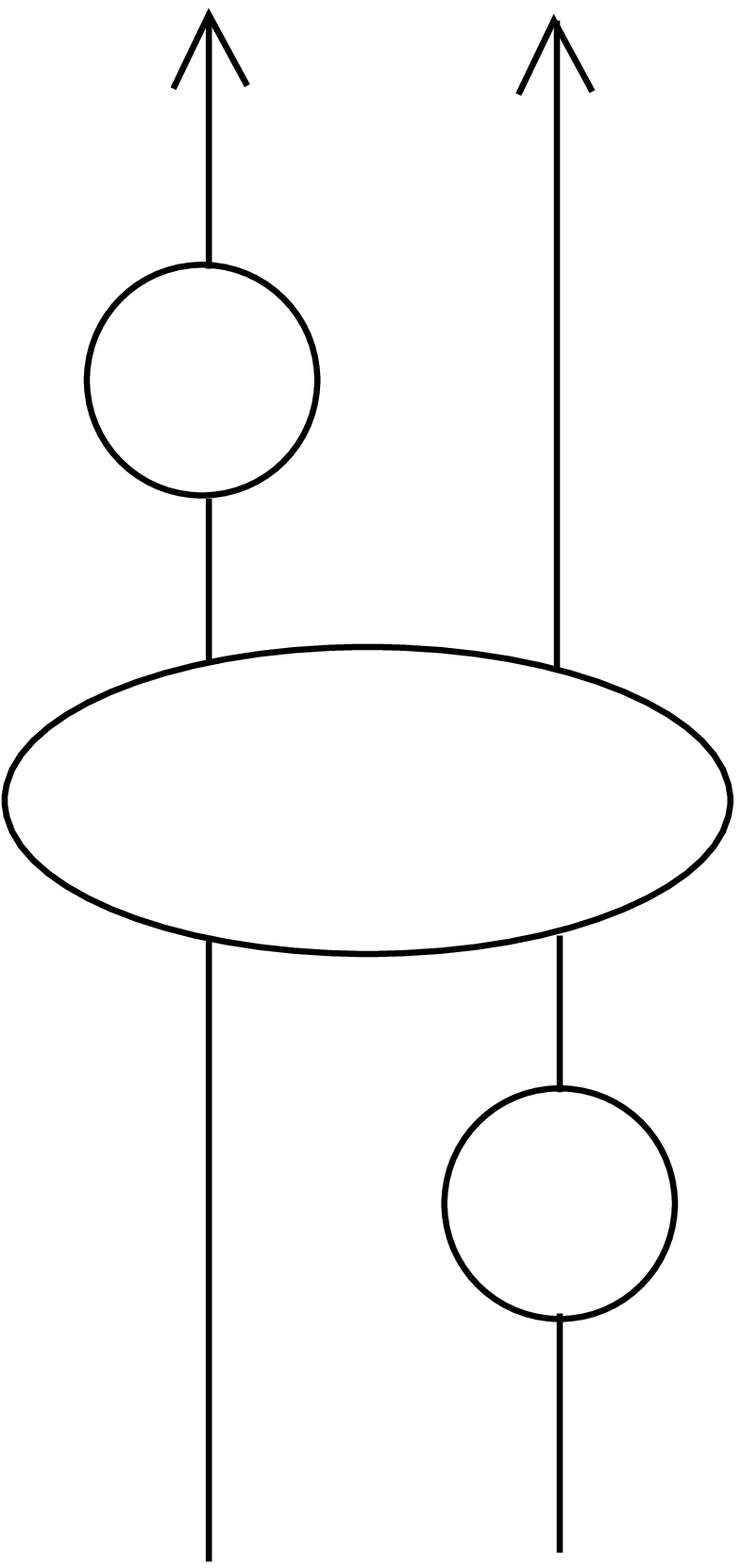}
$$
(see Theorem \ref{th:3d-mu}).
Despite their simplicity, these three-dimensional formulas for Turaev's loop operations seem to be new and we believe that they can be  of independent interest.
Therefore we prove those formulas in the  general case of a connected oriented surface $\Sigma$ with non-empty boundary.
In particular, we hope them to be  useful to obtain 
some similar, formal descriptions of the map $\vec \mu$ and the Turaev cobracket
when $\Sigma$ is a compact surface with a single boundary component.
The proof of this more difficult case would follow the same strategy and should involve a higher-genus analogue of the Kontsevich integral:
this could be  the invariant of surface tangles derived either from  Enriquez's higher-genus associators  \cite{En_elliptic,Hu}
or, alternatively, from the LMO functor \cite{CHM,No}. (See \cite{Katz} for a connection between these two alternatives.)

The paper is organized as follows. 
We review in Section~\ref{sec:Fox} the notion of ``Fox pairing'' and present the related notion of  ``quasi-derivation'':
these notions abstract the properties of the pairing $\eta$ and the map $\vec \mu$, respectively.
In Section~\ref{sec:Turaev_operations} we review the original {two-dimensional} definitions of the  operations  $\eta$ and  $\vec \mu$,
while the {three-dimensional} formulas for these operations are stated and proved in Section~\ref{sec:Turaev-dim3}.
This first part of the paper holds for any connected oriented surface  $\Sigma$ with non-empty boundary.
Next, we assume that $\Sigma$ is a punctured disk.
Section~\ref{sec:special} introduces ``special expansions'' of $\pi=\pi_1(\Sigma,\ast)$, 
which produce the special isomorphisms  \eqref{eq:cha_iso} that we need.
The construction of a special expansion from the Kontsevich integral is explained in Section \ref{sec:special_Z}.
Section~\ref{sec:statement} states our main results:
the formal description of the self-intersection map $\vec{\mu}$ for an arbitrary (respectively, even) associator $\Phi$,
and the formal description of the Turaev cobracket.
Section~\ref{sec:proofs} contains the proofs of these results.
The paper  ends with an appendix where the formal description of $\eta$ for a punctured disk is derived from the analogous result for a compact surface 
with one boundary component.\\

In the sequel, we fix the following miscellaneous conventions:
\begin{itemize}
\item {The set of non-negative integers is denoted by $\N$.}
\item We shall specify a commutative ring $\K$ at the beginning of each section
and by a ``module'',  an ``algebra'',  a ``linear map'', etc$.$
we shall mean a ``$\K$-module'',  an ``associative $\K$-algebra'', a ``$\K$-linear map'', etc. 
\item  The \emph{commutator} of two elements $x,y$ of a group $G$ is  defined by $[x,y]:={xy x^{-1} y^{-1} \in G}$.
The \emph{commutator} of two elements $a,b$ of an associative $\K$-algebra $A$ is   the element $[a,b]:= ab -ba \in A$.
\item  Given a based loop $\alpha$ in a pointed topological space $(X,\ast)$, we still denote by $\alpha$ its class 
in the fundamental group $\pi_1(X,\ast)$.
\item  We  use the ``blackboard framing'' convention to present framed tangles by their projection diagrams.
\end{itemize}

\noindent
\textbf{Acknowledgements.}
{We are grateful to Benjamin Enriquez and Vladimir Turaev for  helpful  discussions.
We would like to thank Nariya Kawazumi  for his comments on a first version of this paper.
We are also indebted to the referee for  their careful reading and for having corrected/simplified  the proof of Lemma \ref{lem:special_to_symplectic}.}\\

\section{Fox pairings and quasi-derivations} \label{sec:Fox}

 We briefly review from \cite{MT_dim_2} the notion  of ``Fox pairing'',
 we present the  notion of ``quasi-derivation'' and we introduce several constructions related to these notions.
In this section,   $\K$ is an arbitrary commutative ring,
and by a ``Hopf algebra'' we mean an involutive Hopf algebra over $\K$.

\subsection{Fox derivatives}

Let $A$ be an augmented algebra with counit $\varepsilon:A \to \K$.
A map ${\partial}:A \to A$ is a \emph{left} (respectively, a \emph{right})   \emph{Fox derivative} if
$\partial(ab)=a \partial(b)+ \partial(a) \varepsilon(b)$ (respectively, $\partial(ab)= \partial(a) b+ \varepsilon(a) \partial(b)$)  for any $a,b \in A$.

As an example, consider  the free product  $G\ast F(z)$ of a  group $G$ and the group $F(z)$ freely generated by an element $z$. 
Let $A:=\K[G\ast F(z)]$ be the group algebra of $G\ast F(z)$. Then there is a unique left  Fox derivative
$$
\frac{\partial \ }{\partial z}: A \longrightarrow A
$$
such that 
$$
\frac{\partial z}{\partial z} = 1 \quad  \hbox{and} \quad \forall g \in G, \
\frac{\partial g}{\partial z} =0.
$$
The following lemma, which develops this example, will be very useful in the next sections.

\begin{lemma} \label{lem:Fox}
Let $p: G\ast F(z) \to G$ be the group homomorphism defined by $p(z):=1$ and $p(g):=g$ for all $g \in G$. 
Let $J$ be a finite ordered set and, for every $j\in J$, let $\varepsilon_j\in \{-1,+1\}$ and $r_j\in G\ast F(z)$. Then
$$
p \frac{\partial \ }{\partial z}\left(\, \prod_{j\in J} r_j\, z^{\varepsilon_j}\, r_j^{-1}\right) = \sum_{j\in J} \varepsilon_j\, p(r_j) \ \in \K[G].
$$
\end{lemma}

\begin{proof}
Let $\varepsilon \in \{-1,+1\}$, and  $r,s  \in G \ast F(z)$. Then
\begin{eqnarray*}
p \frac{\partial \ }{\partial z} (r z^\varepsilon r^{-1} s) &=& p \left( \frac{\partial r}{\partial z}\right)+ p \left( r \frac{\partial  z^\varepsilon r^{-1} s}{\partial z} \right) \\
&=&p \left( \frac{\partial r}{\partial z}\right) + p(r)\, p\left(  \frac{\partial z^\varepsilon  }{\partial z} +  z^\varepsilon \frac{\partial r^{-1} s}{\partial z}\right) \\
&=& p \left( \frac{\partial r}{\partial z}\right)  + \varepsilon p(r) + p(r)\, p \left(\frac{\partial r^{-1} s}{\partial z}\right) \\
& = & p \left( \frac{\partial r}{\partial z}\right)  + \varepsilon p(r)  + p(r)\,  p\left( \frac{\partial r^{-1}}{\partial z}+ r^{-1} \frac{\partial s }{\partial z}\right) \\
& = & p \left( \frac{\partial r}{\partial z}\right)  + \varepsilon p(r)  + p(r)\,  p\left( -r^{-1} \frac{\partial r}{\partial z}\right)+ p\left( \frac{\partial s }{\partial z}\right)
\  =  \  \varepsilon p(r) + p\left(\frac{\partial s}{\partial z}\right).
\end{eqnarray*} 
Hence the lemma follows from  an induction on the cardinality of $J$.
\end{proof}

\subsection{Fox pairings} \label{subsec:Fox_pairings}

We recall some definitions and results from \cite{MT_dim_2}.
Let $A$ be a Hopf algebra with coproduct $\Delta: A \to A \otimes A$, counit $\varepsilon:A \to \K$ and antipode $S: A\to A$.
We will  use Sweedler's convention
$$
\Delta(a) = a' \otimes a'', \quad (\Delta \otimes \id) \Delta(a)= a' \otimes a'' \otimes a''', \quad \hbox{etc.}
$$
to denote the successive iterations of the coproduct of an element $a$ of   $A$.

A \emph{Fox pairing} in $A$ is a bilinear map $\rho: A \times A \to A$ 
which is a left Fox derivative in its first variable and a right Fox derivative in its second variable:
$$
\forall a,b,c \in A, \quad \rho(ab,c)= a \rho(b,c) + {\rho(a,c)\, \varepsilon(b) }
\quad \hbox{and} \quad
\rho(a,bc) = \rho(a,b)c + {\varepsilon(b)\, \rho(a,c) }
$$
For instance, any element $e\in A$ defines an \emph{inner} Fox pairing defined by 
$$
\forall a,b \in A, \quad \rho_e(a,b) := (a-\varepsilon(a))\, e\, (b-\varepsilon(b)).
$$

The \emph{transpose} of a Fox pairing $\rho$ in $A$
is the bilinear map $\rho^t: A \times A \to A$ defined by $\rho^t(a,b):=S\rho(S(b),S(a))$ for any $a,b\in A$. 
Note that $\rho^t$ is a Fox pairing and $(\rho^t)^t =\rho$ (here we use the fact that $S^2 =\id$, see \cite[Lemma 5.1]{MT_dim_2}).
Furthermore, it can be proved that
\begin{equation}
\label{eq:re-involution}
\forall a,b\in A, \quad \rho^t(a,b) = a'\, S\big(\rho(b'',a'')\big)\, b'.
\end{equation}
(see  \cite[Lemma 5.2]{MT_dim_2}).
The Fox pairing $\rho$ is said to be  \emph{skew-symmetric} if $\rho^t=-\rho$.

We now explain an auxiliary construction for Fox pairings.
Consider the bilinear map  $\langle-,-\rangle_\rho: A \times A \to A$  defined by
\begin{equation} \label{eq:bracket}
\forall a,b \in A, \ \langle a, b \rangle_\rho := b'\, S(\rho(a'',b'')')\, a' \, \rho(a'',b'')''   {.}
\end{equation}
 The following facts are verified by direct computations:
\begin{itemize}
\item[(i)] $\langle-,-\rangle_\rho$ is a derivation in its second variable (see \cite[Lemma 6.1]{MT_dim_2}); 
\item[(ii)] for any $e\in A$, the map $\langle-,-\rangle_{\rho_e}$ induced by the inner Fox pairing $\rho_e$ is trivial (see \cite[Lemma 6.3]{MT_dim_2});
\item[(iii)] if $\rho$ is skew-symmetric, $\langle-,-\rangle_\rho$ induces a skew-symmetric  bilinear map
$
\langle-,-\rangle_\rho: {\check{A} \times \check{A}}  \to \check{A} 
$
on the quotient module  $\check{A} := A/[A,A]$, where $[A,A]$ denotes the submodule spanned by commutators in the algebra $A$
(see \cite[Lemma 6.2]{MT_dim_2}).
\end{itemize}
The operation in $\check{A}$ produced in (iii) is called the \emph{bracket} induced by the  Fox pairing $\rho$. 
In the sequel, with a slight abuse of notation, the class modulo $[A,A]$ of an element $a\in A$ is still denoted by $a\in \check{A}$.

\subsection{Quasi-derivations} \label{subsec:quasi-der}

Let $\widetilde{A}$ and $A$ be  Hopf algebras.
We shall denote their coproducts, counits and antipodes by the same letters $\Delta$, $\varepsilon$ and $S$,
except if this could lead to some confusion
 (in which case we shall write $\Delta_{\widetilde A}, \varepsilon_{\widetilde A}, S_{\widetilde A}$ and $\Delta_{A}, \varepsilon_{A}, S_{A}$).
We also assume that a surjective homomorphism of Hopf algebras $p:\widetilde{A} \to A$ is given: 
 the image $p(\widetilde{a})$ of an element $\widetilde{a}\in \widetilde{A}$ will often be simply denoted by $a\in A$.

A linear map $q: \widetilde{A} \to A$  is called a \emph{quasi-derivation ruled by} the  Fox pairing  $\rho:A \times A \to A$ if it satisfies
$$
\forall \widetilde{a},\widetilde{b} \in \widetilde{A}, \quad q\big(\widetilde{a}\, \widetilde{b}\big) =  q(\widetilde{a})\, b + a\, q(\widetilde{b}) + \rho(a,b). 
$$
Denote by $\QDer(\rho)$ the set of quasi-derivations $\widetilde{A} \to A$ ruled by $\rho$. 
(Clearly, this set depends on the projection $p:\widetilde{A}\to A$ too, but we omit it from our notations.)
The module of derivations $\widetilde{A} \to A$ (where $A$ is regarded as an $\widetilde{A}$-bimodule via $p$)
acts freely and transitively  by addition on the set $\QDer(\rho)$.

\begin{lemma}
Let $e_1,e_2 \in A$  and set $e:=e_1+e_2$. Then the linear map defined by
\begin{eqnarray*}
\forall \widetilde{a}\in \widetilde{A}, \quad q_{e_1,e_2}(\widetilde{a}) 
& :=  & (\varepsilon(a)-a)\, e_1 + e_2\, (\varepsilon(a)-a)  \\
&=& \varepsilon(a)\, e -ae_1 -e_2a
\end{eqnarray*}
is a quasi-derivation ruled by the inner Fox pairing $\rho_e$.
\end{lemma}

\begin{proof}
For any $\widetilde{a}, \widetilde{b} \in \widetilde{A}$, 
\begin{eqnarray*}
&& q_{e_1,e_2}\big(\widetilde{a}\, \widetilde{b}\big) -q_{e_1,e_2}(\widetilde{a})\, b - a\, q_{e_1,e_2}(\widetilde{b})\\
&=&  \big( \varepsilon(ab)\, e -ab e_1 - e_2 ab \big)
 - \big(\varepsilon(a)\, e -ae_1 -e_2a\big)b  - a \big(\varepsilon(b)\, e -be_1 -e_2b\big)\\
&=&  \varepsilon(a) \varepsilon(b) e  -a b e_1   - e_2 ab  - \varepsilon(a)\, eb + ae_1b + e_2ab  -  \varepsilon(b)\, ae + abe_1 +ae_2b  \\
&=&    \varepsilon(a) \varepsilon(b)\, e  + a eb -   \varepsilon(a)\, e b  -   \varepsilon(b)\, ae \ = \ \rho_e(a,b).
\end{eqnarray*}

\up
\end{proof}

The \emph{transpose} of a $q\in \QDer(\rho)$ is the linear map $q^{t}: \widetilde{A} \to A$ defined by  $q^t(\widetilde{a}) :=  S_A\, q\, S_{\widetilde A}(\widetilde{a})$ for all $\widetilde{a} \in \widetilde{A}$.
We say that $q $ is \emph{skew-symmetric} if $q^t = -q$.

\begin{lemma}
For any $q\in \QDer(\rho)$, we have $q^t \in \QDer(\rho^t)$.
\end{lemma}

\begin{proof}
For any $\widetilde{a}, \widetilde{b} \in \widetilde{A}$, we have
\begin{eqnarray*}
q^t\big(\widetilde{a}\, \widetilde{b}\big) &=&  S q\big( S(\widetilde{b})\, S(\widetilde{a})\big) \\
&=& S \left( pS(\widetilde{b})\, qS(\widetilde{a}) + qS(\widetilde{b})\, pS(\widetilde{a}) + \rho(pS(\widetilde{b}),pS(\widetilde{a})) \right)  \\
&=& S \left( S(b)\, qS(\widetilde a) + qS(\widetilde{b})\, S(a) + \rho(S(b),S(a)) \right) \\
&=&  SqS(\widetilde a)\, b + a\, SqS(\widetilde b) + S \rho(S(b),S(a))\\
&=& q^t(\widetilde{a})\, b + a\, q^t(\widetilde{b})+ \rho^t(a,b).
\end{eqnarray*}

\up
\end{proof}

We now explain some auxiliary constructions for quasi-derivations.
First, any $q \in \QDer(\rho)$ can be transformed to a  map $d_q: \widetilde{A} \to A \otimes A$ by the formula
\begin{equation} \label{eq:d_q}
\forall \widetilde{a} \in  \widetilde{A}, \quad d_q (\widetilde{a}) := p(\widetilde{a}')\, S_{{A}}( q(\widetilde{a}'')') \otimes q(\widetilde{a}'')''
\end{equation}
where the coproducts $\Delta_{\widetilde{A}}(\widetilde{a})= \widetilde{a}' \otimes \widetilde{a}''$  and $\Delta_A(q(\widetilde{a}''))= q(\widetilde{a}'')' \otimes q(\widetilde{a}'')''$ 
are written with Sweedler's convention.
Note that $q$ can be recovered from $d_q$ by the formula $q = (\varepsilon_A \otimes \id_A) d_q$.

\begin{lemma}
Let $e_1, e_2 \in A$ and set $e := e_1+e_2$. Then the map $d_{e_1,e_2} := d_{q_{e_1,e_2}}$ is given by 
$$
\forall \widetilde{a} \in \widetilde{A}, \quad d_{e_1,e_2}(\widetilde{a}) =  aS( e') \otimes e'' -  a'  S(e_1') S( a'') \otimes  a'''  e_1''  -  S(e'_2) \otimes e_2'' a.
$$
\end{lemma}

\begin{proof}
Let $\widetilde{a} \in \widetilde{A}$. We have
\begin{eqnarray*}
&& d_{e_1,e_2}(\widetilde{a}) \\
 & = &  \varepsilon( p(\widetilde{a}''))  p(\widetilde{a}') S( e') \otimes e''  -  p(\widetilde{a}') S( p(\widetilde{a}'')'\, e_1') \otimes  p(\widetilde{a}'')''\, e_1''  
   -  p(\widetilde{a}') S(e_2'\, p(\widetilde{a}'')' ) \otimes e_2''\, p(\widetilde{a}'')'' \\
& = &    \varepsilon(\widetilde{a}'')  p(\widetilde{a}') S( e') \otimes e''  -  p(\widetilde{a})' S( (p(\widetilde{a})'')' e_1') \otimes  (p(\widetilde{a})'')''  e_1''    
-  p(\widetilde{a})' S(e_2' (p(\widetilde{a})'')' ) \otimes e_2'' (p(\widetilde{a})'')'' \\
& = &   aS( e') \otimes e'' -  a' S( a'' e_1') \otimes  a'''  e_1''  -  a' S(e_2' a'' ) \otimes e_2'' a''' \\
& = &   aS( e') \otimes e'' -  a'  S(e_1') S( a'') \otimes  a'''  e_1''  -  a' S(a'' ) S(e'_2) \otimes e_2'' a''' \\
& = &   aS( e') \otimes e'' -  a'  S(e_1') S( a'') \otimes  a'''  e_1''  -  S(e'_2) \otimes e_2'' a.
\end{eqnarray*}

\up
\end{proof}

Next, any $q\in \QDer(\rho)$ induces a map $\delta_q: \widetilde{A} \to \check{A} \otimes \check{A}$ defined by
\begin{equation}  \label{eq:delta_q}
\forall \widetilde{a} \in \widetilde{A}, \quad \delta_q(\widetilde {a}) := d_q(\widetilde {a}) - P_{21} d_q(\widetilde {a}) \in \check{A} \otimes \check{A}
\end{equation}
where  $P_{21}: \check{A} \otimes \check{A} \to \check{A} \otimes \check{A}$ is the permutation map defined by $P_{21}(u\otimes v):= v \otimes u$.

\begin{lemma} \label{lem:delta_e1_e2}
Let $e_1, e_2 \in A$, set $e:= e_1+e_2$. 
Then the map $\delta_{e_1,e_2} := \delta_{q_{e_1,e_2}}$ only depends on~$e$ and is given by
$$
\forall \widetilde{a} \in \widetilde{A}, \quad  \delta_{e_1,e_2} (\widetilde{a})= aS( e') \otimes e'' +a e'' \otimes S(e') -  S(e')  \otimes   e''  a     - e'' \otimes  S(e') a \ \in \check{A} \otimes \check{A}.
$$
\end{lemma}

\begin{proof}
Let $\widetilde{a} \in \widetilde{A}$. We have
\begin{eqnarray*}
\delta_{e_1,e_2}(\widetilde{a}) &=&  aS( e') \otimes e'' -  a'  S(e_1') S( a'') \otimes  a'''  e_1''  -  S(e'_2) \otimes e_2'' a \\
&&  - e'' \otimes a S(e') + a'''  e_1'' \otimes a'  S(e_1') S( a'')  +   e_2'' a \otimes S(e'_2) \\
 &=&  aS( e') \otimes e'' -  S(e_1') S( a'') a'  \otimes  a'''  e_1''  -  S(e'_2) \otimes e_2'' a \\
&&  - e'' \otimes a S(e') +  a'''  e_1'' \otimes S(e_1') S( a'')a'  +   e_2'' a \otimes S(e'_2)  \\
 &=&  aS( e') \otimes e'' -  S(e_1')\, S( S(a')a'')  \otimes  a'''  e_1''  -  S(e'_2) \otimes e_2'' a \\
&&  - e'' \otimes a S(e') +  a'''  e_1'' \otimes S(e_1')\, S( S(a')a'')  +   e_2'' a \otimes S(e'_2)  \\
 &=&  aS( e') \otimes e'' -  S(e_1')  \otimes  a e_1''  -  S(e'_2) \otimes e_2'' a \\
&&  - e'' \otimes a S(e') +  a e_1'' \otimes S(e_1')  +  e_2'' a \otimes S(e'_2)  \\
&=&  aS( e') \otimes e'' -  S(e_1')  \otimes  a e_1''  -  S(e'_2) \otimes a  e_2''  \\
&&  - e'' \otimes a S(e') +  a e_1'' \otimes S(e_1')  +    ae_2''  \otimes S(e'_2)  \\
&=&  aS( e') \otimes e'' -  S(e')  \otimes  a e''   - e'' \otimes a S(e') +  a e'' \otimes S(e').
\end{eqnarray*}

\up
\end{proof}

Let $\vert \check{A}\vert$ be the quotient of the module $\check{A}$ by the submodule spanned by the  class of $1\in A$.

\begin{lemma} \label{lem:delta_q-skew}
Let $\rho$ be a Fox pairing in $A$ such that $\rho + \rho^t= \rho_e$ for some $e \in \K 1 \subset A$.
Then, for any $q\in \QDer(\rho)$ such that $q(\ker p) \subset \K 1 $, the map $\delta_q$ induces a linear map
$\vert \delta_q\vert: \vert \check{A}\vert  \to \vert \check{A} \vert \otimes \vert \check{A}\vert$ with skew-symmetric values. 
\end{lemma}

\begin{proof}
Since $p$ is a coalgebra map, we have $\Delta(\ker(p)) \subset \ker (p\otimes p)$ and, since $p$ is surjective,
$\ker (p\otimes p) =(\ker p)\otimes \widetilde{A}+   \widetilde{A} \otimes (\ker p)$.
Therefore, for any $\widetilde{a}\in \ker(p)$,
$$
 d_q (\widetilde{a}) = p(\widetilde{a}') S( q(\widetilde{a}'')') \otimes q(\widetilde{a}'')''
$$
belongs to $A \otimes (\K 1)$. 
It follows that $\delta_q: \widetilde{A}  \to \check{A} \otimes \check{A}$ induces a map $\delta_q: A \to \vert \check{A} \vert \otimes \vert \check{A} \vert$;
moreover, $\delta_q(1_A)=0$ since $q(1_{\widetilde{A}})=0$.
Thus it now remains to prove that $\delta_q(\widetilde{a}\, \widetilde{b}) = \delta_q(\widetilde{b}\, \widetilde{a})$ for any $\widetilde{a}, \widetilde{b}\in \widetilde{A}$. We have
\begin{eqnarray*}
d_q\big(\widetilde{a}\, \widetilde{b}\big) 
&=& p(\widetilde{a}' \widetilde{b}')\, S\big( q(\widetilde{a}'' \widetilde{b}'')'\big) \otimes q(\widetilde{a}'' \widetilde{b}'')'' 
\end{eqnarray*}
and, since $q(\widetilde{a}'' \widetilde{b}'') = p(\widetilde{a}'') q(\widetilde{b}'')  + q(\widetilde{a}'') p(\widetilde{b}'') + \rho(p(\widetilde{a}''),p(\widetilde{b}''))$, we get
\begin{eqnarray*}
d_q\big(\widetilde{a}\, \widetilde{b}\big) 
&=& p(\widetilde{a}' \widetilde{b}') S( p(\widetilde{a}'')' q(\widetilde{b}'')'  ) \otimes p(\widetilde{a}'')'' q(\widetilde{b}'')'' \\ 
&& +  p(\widetilde{a}' \widetilde{b}') S(  q(\widetilde{a}'')' p(\widetilde{b}'')' ) \otimes   q(\widetilde{a}'')'' p(\widetilde{b}'')'' \\
&& + p(\widetilde{a}' \widetilde{b}') S( \rho(p(\widetilde{a}''),p(\widetilde{b}''))') \otimes \rho(p(\widetilde{a}''),p(\widetilde{b}''))'' \\
&=& p(\widetilde{a}') p( \widetilde{b}') S(q(\widetilde{b}'')'  )   S( p(\widetilde{a}'')')  \otimes p(\widetilde{a}'')'' q(\widetilde{b}'')'' \\ 
&& +  p(\widetilde{a}') p( \widetilde{b}') S(p(\widetilde{b}'')') S(  q(\widetilde{a}'')'  ) \otimes   q(\widetilde{a}'')'' p(\widetilde{b}'')'' \\
&& +  p(\widetilde{a}') p(\widetilde{b}') S( \rho(p(\widetilde{a}''),p(\widetilde{b}''))') \otimes \rho(p(\widetilde{a}''),p(\widetilde{b}''))'' \\
&=& a' p(\widetilde{b}') S(q(\widetilde{b}'')'  )   S( a'' )  \otimes a''' q(\widetilde{b}'')'' \\ 
&& +  p(\widetilde{a}')  b' S(b'') S(  q(\widetilde{a}'')'  ) \otimes   q(\widetilde{a}'')'' b''' \\
&&  +  a' b' S( \rho(a'',b'')') \otimes \rho(a'',b'')''  \\
&=& a' p(\widetilde{b}') S(q(\widetilde{b}'')'  )   S( a'' )  \otimes a''' q(\widetilde{b}'')'' \\ 
&& +  p(\widetilde{a}')   S(  q(\widetilde{a}'')'  ) \otimes   q(\widetilde{a}'')'' b \\
&& +   a' b' S( \rho(a'',b'')') \otimes \rho(a'',b'')''  \ \in A \otimes A
\end{eqnarray*}
and, modulo commutators, this is equal to
\begin{eqnarray*}
d_q\big(\widetilde{a}\, \widetilde{b}\big) 
&=& p(\widetilde{b}') S(q(\widetilde{b}'')'  )     \otimes  q(\widetilde{b}'')'' a \\
&& +  p(\widetilde{a}')   S(  q(\widetilde{a}'')'  ) \otimes   q(\widetilde{a}'')'' b + b' S( \rho(a'',b'')') a'\otimes \rho(a'',b'')''   \ \in \check{A} \otimes \check{A}.
\end{eqnarray*}
We deduce that
\begin{eqnarray*}
\delta_q\big(\widetilde{a}\, \widetilde{b}\big)  - \delta_q\big(\widetilde{b}\, \widetilde{a}\big) &=& b' S( \rho(a'',b'')') a'\otimes \rho(a'',b'')'' - \rho(a'',b'')'' \otimes b' S( \rho(a'',b'')') a' \\
& & - a' S( \rho(b'',a'')') b'\otimes \rho(b'',a'')'' + \rho(b'',a'')'' \otimes a' S( \rho(b'',a'')') b'.  
\end{eqnarray*}
It can be proved using the formula \eqref{eq:re-involution} that
$$
\forall x,y \in A, \quad
y' S( \rho(x'',y'')') x'\otimes \rho(x'',y'')''  =  \rho^t(y'',x'')'' \otimes x' S( \rho^t(y'',x'')') y'
$$
(see \cite[Equation (6.1.2)]{MT_dim_2}),  and it follows that
\begin{eqnarray*}
\delta_q\big(\widetilde{a}\, \widetilde{b}\big)  - \delta_q\big(\widetilde{b}\, \widetilde{a}\big)
&=&  \rho^t(b'',a'')'' \otimes a' S( \rho^t(b'',a'')') b' - \rho(a'',b'')'' \otimes b' S( \rho(a'',b'')') a' \\
& & -  \rho^t(a'',b'')'' \otimes b' S( \rho^t(a'',b'')') a' + \rho(b'',a'')'' \otimes a' S( \rho(b'',a'')') b' \\
&=& - \rho_e(a'',b'')'' \otimes b' S( \rho_e(a'',b'')') a'  
 + \rho_e(b'',a'')'' \otimes a' S( \rho_e(b'',a'')') b'.
\end{eqnarray*}
A direct computation (see \cite[Lemma 6.3]{MT_dim_2}) shows that
$$
b' S( \rho_e(a'',b'')') a' \otimes \rho_e(a'',b'')'' =   S(e') \otimes a e''b   + b S\!\left(e'\right)  a \otimes e'' -b S(e') \otimes a e'' -S(e')  a \otimes e'' b
$$
and, using the assumption $e \in \K 1$,  we conclude that
\begin{eqnarray*}
\delta_q\big(\widetilde{a}\, \widetilde{b}\big)  - \delta_q\big(\widetilde{b}\, \widetilde{a}\big)
&=&  - \big(  a e''b \otimes  S(e')  + e'' \otimes  b S\!\left(e'\right)  a - a e''  \otimes b S(e')  -e'' b \otimes S(e')  a \big) \\
&& + \big(  b e''a \otimes  S(e')  + e'' \otimes  a S\!\left(e'\right)  b - b e''  \otimes a S(e')  -e'' a \otimes S(e')  b \big) \\
&= & 0.
\end{eqnarray*}

\up
\end{proof}

The map $\vert \delta_q\vert : \vert \check{A} \vert \to \vert \check{A}\vert \otimes \vert \check{A}\vert$ produced by Lemma \ref{lem:delta_q-skew}
is called the \emph{cobracket} induced by the quasi-derivation~$q$.

\subsection{Filtrations and completions}

A \emph{filtration} on a module $V$ is a decreasing sequence of submodules 
$
V= V_0 \supset V_1 \supset V_2 \supset \cdots.
$
The \emph{completion} of the filtered module $V$ is the module $ \widehat{V}:= \varprojlim_{k} V/V_k$ equipped with the filtration inherited from $V$.
The filtered module $V$ is \emph{complete} if the canonical map $V \to \widehat{V}$ is an isomorphism.
Given an integer $d$, a linear map $f:V \to W$ between filtered modules is said to be \emph{$d$-filtered}
if $f(V_k) \subset W_{k+d}$ for all integers $k\geq \max(0,-d)$.
A~linear map $f:V \to W$ is \emph{filtration-preserving} if it is $0$-filtered.
Any $d$-filtered linear map $f:V \to W$ induces  a ($d$-filtered) linear map $\hat{f}: \widehat{V} \to \widehat{W}$, which we call the \emph{completion} of $f$. 
 
Let $A$ be a \emph{filtered  Hopf algebra} with coproduct $\Delta$, counit $\varepsilon$ and antipode $S$.
This means that   $A$ is a Hopf algebra equipped with a filtration
$$
A_0 \supset A_1\supset A_2 \supset \cdots  
$$
such that  $A_1= \ker \varepsilon$, $A_i A_j \subset A_{i+j}$ for any $i,j\in \N$, and the maps {$\Delta, S$} are filtration-preserving.
For instance, the filtration defined by $A_j {:=} I^j$ for all $j\in \N$, where $I:=\ker \varepsilon$, is called the \emph{$I$-adic filtration} of $A$.

A Fox pairing $\rho: A \times A \to A$ is \emph{filtered} if the induced linear map $\rho_\otimes:A \otimes A \to A$ is $(-2)$-filtered, i.e$.$
$
\rho(A_i,A_j) \subset A_{i+j-2}
$
for any  {$i,j\in \N$} such that $i+j \geq 2$.  Given  another filtered Hopf algebra $\widetilde{A}$,
and given a filtration-preserving surjective Hopf algebra homomorphism $p: \widetilde{A} \to A$, 
we say that a quasi-derivation $q:\widetilde{A} \to A$ ruled by $\rho$ is \emph{filtered} if it is $(-2)$-filtered, i.e$.$
$
q(\widetilde{A}_i) \subset A_{i-2}
$
for any integer $i\geq 2$.

\begin{lemma} \label{lem:filtrations}
Assume that the Hopf algebra $A$ is equipped with the $I$-adic filtration.
Then any Fox pairing $\rho$ in $A$ is filtered.
Furthermore, if  $\widetilde{A}$ is equipped  either with the $I$-adic filtration or with the filtration defined by
\begin{equation} \label{eq:weights_1_2}
\widetilde{A}_k := \sum^{\lfloor k/2 \rfloor}_{i=0} \sum_{\substack{\{1,\dots, k-i\} \stackrel{\sigma}{\to} \{1,2\} \\ \sharp\, \sigma^{-1}(2) = i} }  \prod_{j=1}^{k-i} I_{\sigma(j)}
\quad \hbox{where} \ \left\{\begin{array}{c} I_1 := \ker \varepsilon \\ I_2 := \ker p\end{array}\right.
\end{equation}
for {all $k\in \N$}, then any $q\in \QDer(\rho)$    is filtered.
\end{lemma}

\begin{proof}
Let $\rho :A \times A \to A $ be a Fox pairing and assume that  $A_k = (\ker \varepsilon)^k$ for any $k\geq 1$.
Then, for any {$i,j\in \N$}, for any $a\in A_i$, $b\in A_j$ and $a',b' \in \ker \varepsilon$, we have
$$
\rho(aa',b'b) = a\rho(a',b')b \in A_{i+j}.
$$
This proves that $\rho(A_{i+1},A_{j+1})\subset A_{i+j}$ for all {$i,j \in \N$,} and it follows that $\rho$ is $(-2)$-filtered.

We now consider a quasi-derivation $q:\widetilde{A} \to A$ ruled by $\rho$. 
Assume first that  $\widetilde{A}_k = (\ker \varepsilon)^k$ for any $k\geq 1$. 
Then, for any integer $i >0$, for any $\widetilde{a} \in \widetilde{A}_i$ and $\widetilde{b} \in \ker \varepsilon$, we have
$$
q(\widetilde{a} \widetilde{b}) = \underbrace{a\,  q( \widetilde{b})}_{\in A_{i}}  + \underbrace{q(\widetilde{a})\, b}_{\in\, ? \cdot A_1} + \underbrace{\rho(a,b)}_{\in A_{i-1}}.
$$
Thus an induction on $i>0$ shows that $q(\widetilde{A}_{i+1}) \subset {A}_{i-1}$.

Let $k\geq 2$ and assume now that $\widetilde{A}_k$ is defined 
by the sum \eqref{eq:weights_1_2} indexed by $i \in \{0,\dots, \lfloor k/2 \rfloor\}$.
For any $i\in \{2,\dots, {\lfloor k/2 \rfloor} \}$, $q$ vanishes on the $i$-th summand since $q$ vanishes on $(\ker p)^2$.
{Besides, $q$ maps the $1$-st summand of \eqref{eq:weights_1_2} to  ${A}_{k-2}$ since, }
 for any $\widetilde{a} \in \ker p$, any $\widetilde{b}_1, \dots, \widetilde{b}_{k-2} \in  \ker \varepsilon$ and $r \in \{1,\dots, k-1\}$ 
we have 
$$
q(\widetilde{b}_1 \cdots \widetilde{b}_{r-1}\, \widetilde{a}\, \widetilde{b}_r \cdots \widetilde{b}_{k-2}) 
= b_1 \cdots b_{r-1}\, q( \widetilde{a})\, b_r \cdots b_{k-2} \in (\ker \varepsilon)^{k-2} =A_{k-2}.
$$ 
{Finally, the previous paragraph shows that $q$ maps the $0$-th summand of \eqref{eq:weights_1_2} to  ${A}_{k-2}$, }
and we conclude that $q$ is $(-2)$-filtered.
\end{proof}

The completion $ \widehat{A}$ of $A$ has a structure of \emph{complete Hopf algebra},
i.e$.$ it is a ``Hopf algebra'' object in the category of complete filtered modules equipped with the monoidal structure defined by the completed tensor product.
Specifically, the coproduct $\hat{\Delta}: \widehat{A} \to \widehat{A} \hat{\otimes}  \widehat{A}$ of $ \widehat{A}$
takes values in the completion of $\widehat{A} {\otimes}  \widehat{A}$. 
All the notions introduced in the previous subsections that do not involve the coproduct extend verbatim to the complete Hopf algebra $\widehat{A}$:
this includes the definition of (filtered) Fox pairing  $\hat \rho:\widehat{A} \times \widehat{A}  \to \widehat{A}$ 
and the definition of (filtered) quasi-derivation $\hat q: \widehat{\widetilde{A}} \to \widehat{A}$. 
Note that any filtered Fox pairing $\rho$ in $A$ and any filtered $q\in \QDer(\rho)$ 
induce, respectively,  a filtered Fox pairing $\hat \rho$ in $\widehat A$ and a filtered $\hat q\in \QDer(\hat \rho)$.

Furthermore, any  construction of the previous subsections that involves the coproduct  can  be adapted mutatis mutandis to the setting of complete Hopf algebras,
and each of this construction ``commutes'' with the procedure of completion.
For instance, a skew-symmetric Fox pairing $\hat \rho$ in $\widehat{A}$ induces a bracket
\begin{equation} \label{eq:<-,->}
\langle -,-\rangle_{\hat \rho}: \check{\widehat{A}} \times \check{\widehat{A}} \longrightarrow \check{\widehat{A}}
\end{equation}
by the same formula \eqref{eq:bracket}, where  $\check{\widehat{A}}$ is the quotient of ${\widehat{A}}$ by the submodule of commutators;
if $\hat \rho$ is now the completion of a skew-symmetric filtered Fox pairing $\rho$ in $A$,
the completion of the $(-2)$-filtered linear map $\langle -,-\rangle_{\hat \rho}: \check{\widehat{A}} \otimes \check{\widehat{A}} \to \check{\widehat{A}} $ induced by \eqref{eq:<-,->}
is the completion of the $(-2)$-filtered linear map  $\langle -,-\rangle_{ \rho}: \check{A}  \otimes \check{A} \to \check{A}$ induced by the bracket of $\rho$. 
Here we use the fact that the completion $\widehat{\check{\widehat{A}}}$  
of $\check{\widehat{A}} $ (with respect to the filtration that it inherits from $\widehat{A}$)
  is canonically isomorphic to  the completion $\widehat{\check A}$  of $\check{A}$
(with respect to the filtration that it inherits from $A$), which implies that 
$$
\reallywidehat{\check{\widehat{A}} \otimes \check{\widehat{ A} }}  \,  \simeq 
\widehat{\check{\widehat{A}}}\, \hat\otimes\, \widehat{\check{\widehat{ A} }} \, \simeq  \,
\widehat{\check{A} }\, \hat\otimes\, \widehat{\check{ A} } \, \simeq \, \reallywidehat{ \check{A} \otimes \check{A}}.
$$
Similarly, any quasi-derivation $\hat q:\widehat{\widetilde A} \to \widehat{A}$ ruled by a Fox pairing $\hat \rho$ in  $\widehat{A}$ induces some linear maps 
$$
d_{\hat q}: \widehat{\widetilde A} \longrightarrow \widehat{A} \hat\otimes \widehat{A}   \quad \hbox{and} \quad
\delta_{\hat q}: \widehat{\widetilde A} \longrightarrow  \widehat{\check{\widehat{A}}}\, \hat\otimes\,  \widehat{\check{\widehat{A}}}
$$
by the   formulas \eqref{eq:d_q} and \eqref{eq:delta_q}, respectively, using the completed tensor product of filtered modules
instead of the usual tensor product of modules;
when $\hat \rho$ is now the completion of a filtered Fox pairing $\rho$ in $A$ and $\hat q $ is the completion of a filtered  $q\in \QDer(\rho)$,
the above maps $d_{\hat q}$ and $\delta_{\hat q}$ are the completions of the $(-2)$-filtered maps
$$
d_{ q}: {\widetilde A} \longrightarrow {A} \otimes {A}   \quad \hbox{and} \quad
\delta_{ q}: {\widetilde A} \longrightarrow{\check A} \otimes {\check A}
$$
respectively.  Lemma \ref{lem:delta_q-skew} can also be adapted to the setting of complete Hopf algebras.

\section{Turaev's loop operations}  \label{sec:Turaev_operations}

In this section, $\K$ is a commutative ring
and $\Sigma$ is a connected oriented surface with non-empty boundary.
We review the  loop operations on $\Sigma$   that Turaev introduced in \cite{Tu_loops}.

\subsection{Turaev's intersection pairing} \label{subsec:eta}

Let $\pi:=\pi_1(\Sigma,\ast)$ where  $\ast \in \partial \Sigma$ and let $\K[\pi]$ denote the group algebra of $\pi$.
The first loop operation introduced by Turaev is the \emph{homotopy intersection pairing} 
$$
\eta: \K[\pi] \times \K[\pi] \longrightarrow \K[\pi].
$$
The map $\eta$ is bilinear and, for any $a,b\in \pi$, we define $\eta(a,b) \in \K[\pi]$ in the following way. 
Let~$\nu$ be the oriented boundary component of $\Sigma$ containing the {base point} $\ast$.
Let $\bullet, \blacktriangle\in \partial \Sigma$ be some additional points  such that $\bullet < \ast < \blacktriangle$ along $\nu$.
Given an oriented path $\gamma$ in $\Sigma$ and two simple points $p<q$ along $\gamma$,
we denote by $\gamma_{pq}$  the arc in $\gamma$ connecting $p$ to $q$, while the same arc with the opposite orientation is denoted by $\overline{\gamma}_{qp}$.
Let $\alpha$ be a loop based at $\bullet$ such that $\overline{\nu}_{\ast \bullet} \alpha \nu_{\bullet \ast}$ represents $a$
and let $\beta$ be a loop based at $\blacktriangle$ such that ${\nu}_{\ast \blacktriangle} \beta \overline{\nu}_{\blacktriangle \ast}$ represents $b$; 
we   assume that these loops  are in transverse position and that $\alpha\cap \beta$ only consists of simple points of $\alpha$ and $\beta$:
$$
\labellist
\scriptsize\hair 2pt
 \pinlabel {$\bullet$} [t] at 172 10
 \pinlabel {$\blacktriangle$} [t] at 538 10
 \pinlabel {$\alpha$} [tr] at 96 69
 \pinlabel {$\beta$} [tr] at 459 80
 \pinlabel {$p$} [b] at 347 187
  \pinlabel {$\Sigma$}  at 35 378
 \pinlabel {$\ast$} at 355 -9
 \pinlabel {$\circlearrowleft$}  at 661 46
 \pinlabel {$\nu$} [t] at 663 3
\endlabellist
\centering
\includegraphics[scale=0.2]{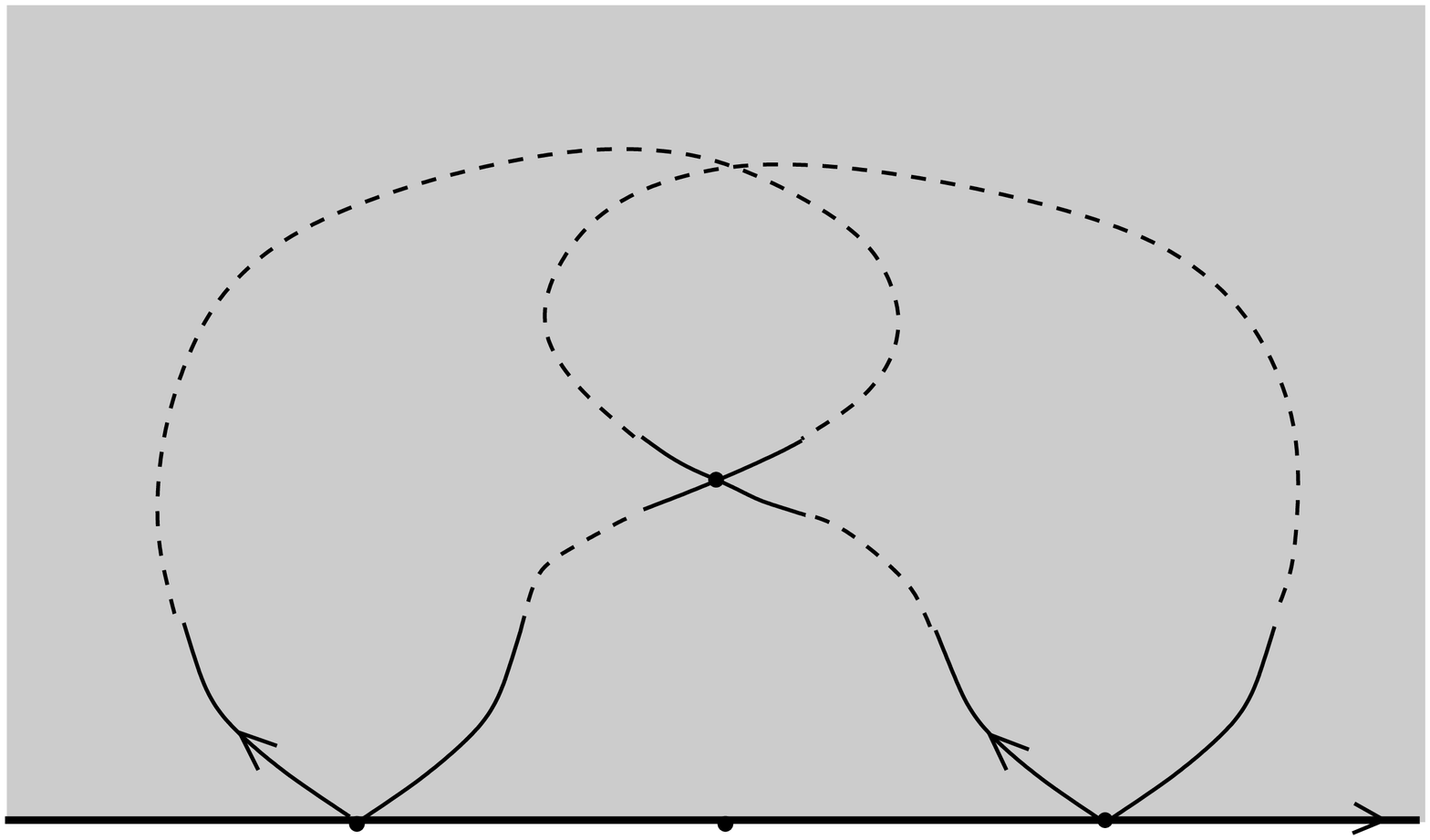}
$$
Then
\begin{equation} \label{eq:def-eta}
\eta(a,b) := \sum_{p \in \alpha \cap \beta} \varepsilon_p(\alpha,\beta)\, \overline{\nu}_{\ast \bullet} \alpha_{\bullet p} \beta_{p\blacktriangle} \overline{\nu}_{\blacktriangle \ast}
\end{equation}
where the sign $\varepsilon_p(\alpha,\beta)=\pm 1$ is equal to $+1$
 if, and only if, a unit tangent vector of $\alpha$ {at~$p$} followed by a unit tangent vector of $\beta$ {at $p$} gives a positively-oriented frame of $\Sigma$.
The homotopy intersection pairing was denoted by $\lambda$ in the article \cite{Tu_loops}, 
where it has a slightly different definition; $\lambda$ is shown there to detect pairs of elements of $\pi$ that can be represented by disjoint loops.
The pairing $\eta$ is also  implicit in the work of Papakyriakopoulos \cite{Pa} who studied Reidemeister's equivariant intersection pairings on surfaces.

As observed  by Turaev in \cite{Tu_loops}  (with  an equivalent terminology adapted to the equivalent pairing $\lambda$), 
the bilinear map $\eta$ is a Fox pairing which satisfies {$\eta+ \eta^t = - \rho_1$}
where $\rho_1$ denotes the inner Fox pairing associated to $1\in \K[\pi]$.
These properties of $\eta$ are easily deduced from the definition \eqref{eq:def-eta}.
Assuming that $1/2 \in \K$, we deduce that  the Fox pairing 
$
\eta^s:= \frac{1}{2}(\eta - \eta^t) = \eta + \rho_{1/2}
$ 
is skew-symmetric.

It follows from the facts (ii) and (iii) mentioned in Section~\ref{subsec:Fox_pairings}
that the Fox pairing $\eta$ induces a skew-symmetric bilinear operation 
$\langle -,- \rangle_\eta$ in  $\K[\pi]/\big[\K[\pi], \K[\pi]\big]$.
We identify this quotient module  with the module $\K\check{\pi}$ freely generated by the set $\check{\pi}$ of conjugacy classes in~$\pi$.
A straightforward computation shows that the resulting operation in $\K\check{\pi}$ is the \emph{Goldman bracket}~\cite{Go} 
$$
\langle-,-\rangle_{\operatorname{G}}: \K\check{\pi} \times \K\check{\pi} \longrightarrow \K\check{\pi},
$$
which is defined  by
$$
\langle a , b\rangle_{\operatorname{G}} := \sum_{p\in \alpha \cap \beta} \varepsilon_p(\alpha,\beta)\, \alpha_{p} \beta_{p}.
$$
Here $\alpha,\beta$ are free loops representing some conjugacy classes $a,b\in \check \pi$ and meeting transversally in a finite number of simple points,
and $\alpha_p \beta_p$ denotes the loop $\alpha$ based at $p$ which is concatenated with the loop $\beta$ based at $p$.

\subsection{Turaev's self-intersection map} \label{subsec:mu}

We  endow the base point $\ast\in \partial \Sigma$  with  the unit vector  $\vec{\ast}$ tangent to $\overline{\partial \Sigma}$,
and we  consider the fundamental group $\overrightarrow{\pi}:= \pi_1(U\Sigma,\vec{\ast})$ of the unit tangent bundle of $\Sigma$ based at  $\vec{\ast}$.
The second loop operation introduced by Turaev is the \emph{homotopy self-intersection map}
$$
\vec{\mu}: \K[\overrightarrow{\pi}] \longrightarrow \K[\pi].
$$
The map $\vec{\mu}$ is linear and, for any  $\vec{a} \in \overrightarrow{\pi}$, we define $\vec{\mu}(\vec{a}) \in  \K[\pi]$ as follows.
Let $\alpha:[0,1] \to \Sigma$ be an immersion  with only finitely many transverse double points 
such that $\dot\alpha(0)= \dot\alpha(1)  = \vec{\ast}$, 
and  the unit tangent vector field of $\alpha$ represents $a$:
$$
\labellist
\scriptsize\hair 2pt
 \pinlabel {$\Sigma$} [t] at 40 396
 \pinlabel {$\circlearrowleft$}  at 59 53
 \pinlabel {$\alpha$} [r] at 166 84
 \pinlabel {$p$} [b] at 356 302
 \pinlabel {$\vec{\ast}$} [t] at 348 8
\endlabellist
\centering
\includegraphics[scale=0.2]{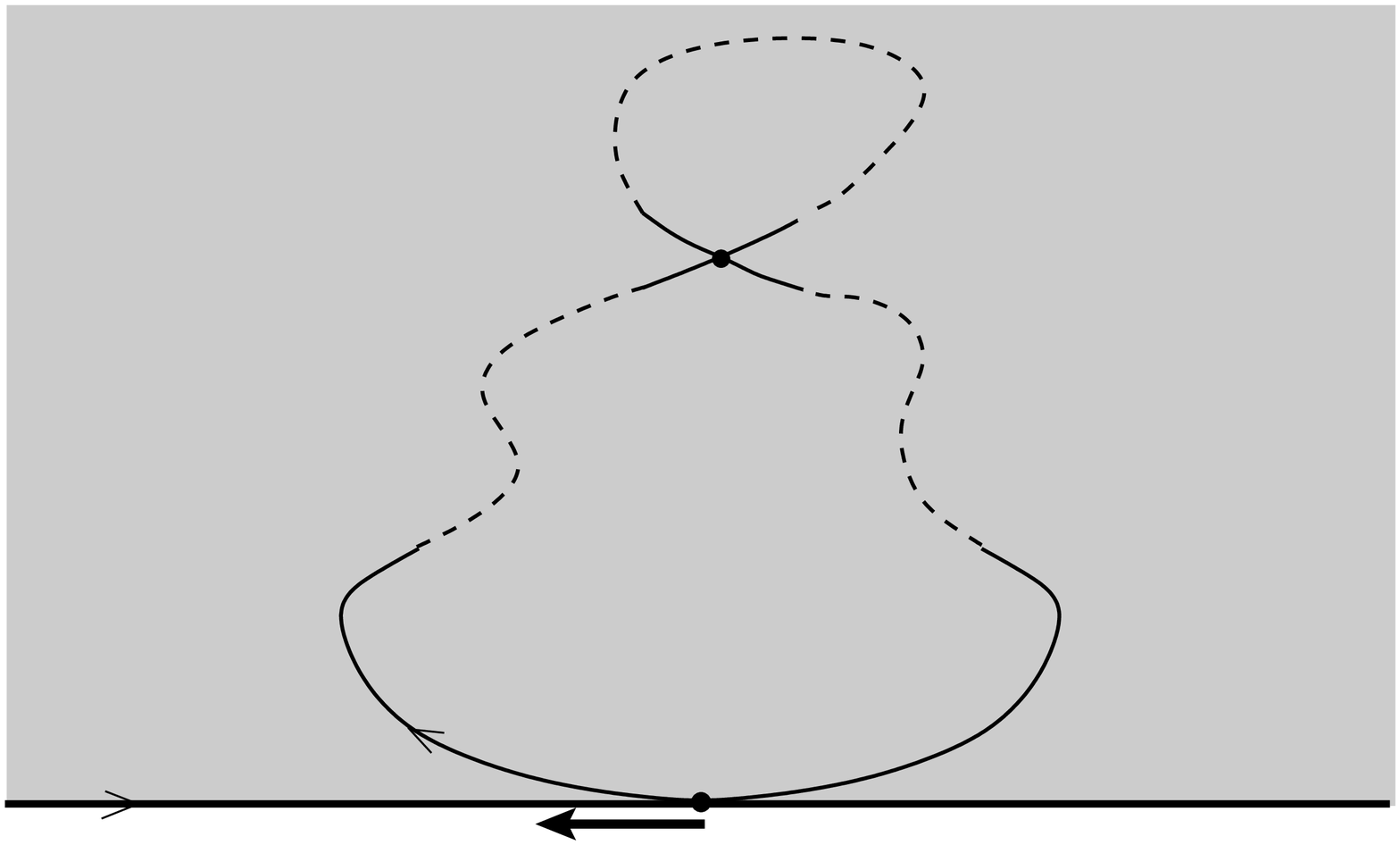}
$$
Then
\begin{equation} \label{eq:def-mu}
\vec{\mu} (\vec{a}) :=
 \sum_{p \in P_\alpha} \varepsilon_p(\alpha)\, (\alpha_{\ast p} \alpha_{p \ast})  + a
\end{equation}
where  $a\in \pi$ is the projection of $\vec a \in \overrightarrow{\pi}$, $P_\alpha$ denotes the set of double points of $\alpha$
and, for every $p\in P_\alpha$,  we use the following notations:
$\alpha_{\ast p}$ is the arc in $\alpha$ running from $\ast$ to the first occurence of $p$ while $\alpha_{p \ast}$ is the arc in $\alpha$ running
from the second occurence of $p$ to $\ast$; the sign $\varepsilon_p(\alpha)=\pm 1$ is equal to $+1$ 
if and only if the first unit tangent vector of $\alpha$ at $p$ followed by the second unit  tangent vector of $\alpha$ at $p$ gives a positively-oriented frame of~$\Sigma$.
The original version of the map $\vec{\mu}$, which was denoted by $\mu$ in \cite{Tu_loops},
 is defined in a different way on $\K[\pi]$ rather than $\K[\overrightarrow{\pi}]$;
 the map $\mu$ is shown there to detect elements of $\pi$ that can be represented by simple loops.
 The above framed version $\vec{\mu}$  has been considered  in  \cite{BP}.

It can be verified from the definition \eqref{eq:def-mu} that  $\vec{\mu}$ has the following properties: first, 
\begin{equation} \label{eq:multiplicativity_mu}
\forall \vec{a},\vec{b} \in \overrightarrow{\pi}, \quad 
\vec{\mu}\big(\vec{a}\, \vec{b}\, \big) = a\, \vec{\mu}( \vec{b}) + \vec{\mu}(\vec{a})\, b + \eta(a,b)
\end{equation}
where $a,b\in \pi$ denote the projections of $\vec{a},\vec{b}\in \overrightarrow{\pi}$; second,
\begin{equation} \label{eq:skew-symmetry}
\forall \vec{a} \in \overrightarrow{\pi}, \quad  
\vec{\mu}\big((\vec{a})^{-1}\big) = - \overline{\vec{\mu}(\vec a)} - 1 +a^{-1}
\end{equation}
where we denote by $x\mapsto \overline{x}$  the antipode of the group algebra $\K[\pi]$.
Similar properties have already been  observed in \cite{Tu_loops} for the unframed version of $\vec{\mu}$.
These properties can be rephrased as follows: first, $\vec{\mu}$ is a quasi-derivation ruled by the Fox pairing $\eta$,
if we consider the homomorphism $p:\K[\overrightarrow{\pi}] \to \K[\pi]$ of Hopf algebras  induced by the bundle projection $U\Sigma \to \Sigma$;
 second, we have  {$\vec{\mu}^t = - \vec{\mu} + q_{-1,0}$.}
Therefore, assuming that $1/2 \in \K$,  the  map 
$
\vec{\mu}^s := \frac{1}{2} (\vec{\mu}-\vec{\mu}^t)= \vec\mu + q_{1/2,0}
$
is a skew-symmetric quasi-derivation ruled by $\eta^s$.

We now compute the map $d_{\vec{\mu}}: \K[\overrightarrow{\pi}] \to \K[\pi] \otimes \K[\pi]$ 
induced by the quasi-derivation $\vec{\mu}$. It follows directly from \eqref{eq:d_q} and \eqref{eq:def-mu} that
\begin{eqnarray*}
 d_{\vec{\mu}}(\vec{a}) &= & \sum_{p \in P_\alpha} \varepsilon_p(\alpha)\, a (\alpha_{\ast p} \alpha_{p \ast})^{-1} \otimes  \alpha_{\ast p} \alpha_{p \ast} + a a^{-1} \otimes a \\
&=&  \sum_{p \in P_\alpha} \varepsilon_p(\alpha)\,  \alpha_{\ast p}  \alpha_{pp}  \overline{\alpha}_{p\ast} \otimes  \alpha_{\ast p} \alpha_{p \ast} +  1 \otimes a
\end{eqnarray*}
where $\alpha_{pp}$ is the arc in $\alpha$ running from the first occcurence of $p$ to its second occurence.
Next, the map $\delta_{\vec{\mu}}:  \K[\overrightarrow{\pi}] \to \K \check{\pi} \otimes \K\check{\pi}$ induced by $\vec{\mu}$ is given by 
\begin{equation} \label{eq:delta_mu} 
\delta_{\vec{\mu}} (\vec{a}) = \sum_{p \in P_\alpha} \left({\alpha_p^{+}} \otimes {\alpha_p^{-}} - {\alpha_p^{-}} \otimes {\alpha_p^{+}} \right) +1 \otimes a -a \otimes 1
\end{equation}
where $\alpha_p^{+}$ and $\alpha_p^{-}$ denote the free loops into which  each point $p\in P_\alpha$ splits $\alpha$,
with the condition that the tangent vector of $\alpha$ at $p$ pointing towards $\alpha_p^+$ 
followed by the tangent vector of $\alpha$ at $p$ pointing towards $\alpha_p^-$
give a positively-oriented frame of $\Sigma$.   

Let $\vert \K\check{\pi} \vert := \K\check{\pi}/(\K 1)$.
According to Lemma~\ref{lem:delta_q-skew}, the map $\delta_{\vec{\mu}}$ 
induces a  linear map $\vert\delta_{\vec{\mu}}\vert: \vert \K\check{\pi}\vert \to \vert \K\check{\pi}\vert \otimes \vert \K\check{\pi}\vert$ with skew-symmetric values.
We deduce from \eqref{eq:delta_mu} that, for any  $a\in \check{\pi}$ represented by a free loop $\alpha$ with only finitely many transverse double points,
$$
\vert\delta_{\vec{\mu}}\vert(\vert a\vert ) = \sum_{p \in P_\alpha} \left(\vert \alpha_p^{+}\vert  \otimes \vert \alpha_p^{-}\vert  - \vert \alpha_p^{-} \vert \otimes \vert \alpha_p^{+}\vert  \right) 
$$
where $\vert\! -\! \vert:  \K\check{\pi} \to  \vert \K\check{\pi}\vert $ denotes the canonical projection.
Let $\vert\check{\pi}\vert$ denote the set $\check{\pi}$ deprived of (the conjugacy class of) $1\in \pi$: we can identify
$\vert \K\check{\pi} \vert $  with the module $\K\vert\check{\pi}\vert$ freely generated by~$\vert\check{\pi}\vert$.
Then the  previous identity for $a\in \vert\check{\pi}\vert$ writes
$$
\vert\delta_{\vec{\mu}}\vert(a) =
\sum_{p \in P_\alpha \setminus P_{\alpha}^0} \left({\alpha_p^{+}} \otimes {\alpha_p^{-}} - {\alpha_p^{-}} \otimes {\alpha_p^{+}} \right)
$$
where $P_\alpha^0$ is the subset of $P_\alpha$ consisting of the points $p$ 
such that one of the free loops $\alpha^{\pm}_p$ is null-homotopic.
We conclude that $\vert\delta_{\vec{\mu}}\vert$ coincides with the Turaev cobracket $\delta_{\operatorname{T}}$ introduced in~\cite{Tu_skein}.
 In other words, we have the following commutative diagram:
\begin{equation}  \label{eq:delta's}
\xymatrix{
 \K[\overrightarrow{\pi}] \ar[r]^-{\delta_{\vec{\mu}}}  \ar[d]_-{ \vert - \vert \circ p }&  \K\check{\pi} \otimes  \K\check{\pi} \ar[d]^-{\vert - \vert \otimes \vert -  \vert} \\
 \K\vert\check{\pi}\vert \ar[r]^-{\delta_{\operatorname{T}}} &  \K\vert\check{\pi}\vert \otimes  \K\vert\check{\pi}\vert
}
\end{equation}
where the left vertical map $\vert\! -\! \vert: \K[\pi] \to \K\vert\check{\pi}\vert$ is the composition of the canonical projection $\K[\pi] \to \K\check{\pi}$
with $\vert\! -\! \vert:  \K\check{\pi} \to  \K\vert\check{\pi}\vert$.

\subsection{Completions of Turaev's loop operations}

The Hopf algebra $\K[\pi]$ is endowed with the $I$-adic filtration:
\begin{equation}   \label{eq:I-adic}
\forall k\in \N, \quad (\K[\pi])_k := I^k \quad \hbox{where } I:= \ker \varepsilon.
\end{equation}
According to the first statement of Lemma \ref{lem:filtrations}, the Fox pairing $\eta$ is filtered
so that it induces a filtered Fox pairing $\hat \eta$ in the complete Hopf algebra
$
\widehat{\K[\pi]} = \varprojlim_k \K[\pi]/(\K[\pi])_k .
$
The bracket $\langle -,- \rangle_{\hat \eta}$ induced by $\hat \eta$ is the completion of the bracket induced by $\eta$, 
namely the completion of  the Goldman bracket:
$$
\langle - , - \rangle_{\operatorname{G}}^{\widehat{} }: \reallywidehat{\K\check \pi} \times \reallywidehat{\K\check \pi} \longrightarrow  \reallywidehat{\K\check \pi}
$$
Here $\reallywidehat{\K\check \pi}$ denotes the completion of ${\K\check \pi} \simeq \K[\pi]/\big[\K[\pi],\K[\pi]\big]$ with respect to the filtration that it inherits from $\K[\pi]$.

Recall that $p: \K[\overrightarrow{\pi}] \to \K[\pi]$ is the Hopf algebra homomorphism induced by the bundle projection $U\Sigma \to \Sigma$.
We endow the Hopf algebra $\K[\overrightarrow{\pi}]$  either with its $I$-adic filtration, 
or  with the filtration \eqref{eq:weights_1_2} where $\widetilde{A} := \K[\overrightarrow{\pi}]$.
Then, according to the second statement of Lemma~\ref{lem:filtrations}, the quasi-derivation $\vec{\mu}$ is filtered  in either case
so that it induces a filtered quasi-derivation $\hat{\vec{\mu}}$  ruled by $\hat{\eta}$ in  the complete Hopf algebra 
$
{\widehat{\K[\overrightarrow{\pi}]} }= \varprojlim_k \K[\overrightarrow{\pi}]/(\K[\overrightarrow{\pi}])_k .
$
The cobracket $\vert \delta_{\hat{\vec{\mu}}} \vert$  induced by $\hat{\vec{\mu}}$ is the completion of the cobracket induced by $\vec{\mu}$,
namely the completion of the Turaev cobracket
$$
{\hat{\delta}_{\operatorname{T}}}:  \reallywidehat{\K\vert  \check{\pi }\vert} 
\longrightarrow  \reallywidehat{\K\vert  \check{\pi }\vert}  \hat \otimes \ \reallywidehat{\K\vert  \check{\pi }\vert}.
$$
Here $\reallywidehat{\K\vert {\check{\pi  }} \vert} $ 
denotes the completion of $\K\vert\check{\pi}\vert \simeq \vert \K\check{\pi}\vert \simeq \K[\pi]\big/\big( {\big[ \K[\pi],\K[\pi] \big]} + \K 1\big)$ 
with respect to the filtration that it inherits from $\K[\pi]$.

\subsection{Remarks}

1. The ways how the Goldman bracket and the Turaev cobracket can be derived from 
Turaev's loop operations have been explained in \cite{MT_dim_2} and in \cite{KK_intersection,MT_draft}, respectively.
Kawazumi has also studied a framed version of an operation which is equivalent to Turaev's self-intersection map $\mu$:
the resulting framed version of the Turaev cobracket ${\delta}_{\operatorname{T}}$ is  announced in \cite{Ka_announcement}.

2. For our purposes in Sections \ref{sec:special}--\ref{sec:proofs} where $\Sigma$ will be a punctured disk,
we only need to consider the $I$-adic filtration on $\K[\overrightarrow{\pi}]$.
Nonetheless, the filtration \eqref{eq:weights_1_2} should be useful for surfaces of positive genus.

3. All the properties of $\eta$ and $\vec\mu$ that have been claimed without proof in this section
(namely, the facts that $\eta$ is a Fox pairing, that $\vec{\mu}$ is  a quasi-derivation ruled by $\eta$ and  the formulas 
computing {$\eta^t$} and $\vec{\mu}^t$) will follow from the results in Section~\ref{sec:Turaev-dim3}.

\section{Three-dimensional formulas for Turaev's loop operations}  \label{sec:Turaev-dim3}

In this section, $\K$ is a commutative ring and $\Sigma$ is a connected oriented surface with non-empty boundary.
We set $\pi:=\pi_1(\Sigma,\ast)$ where $\ast \in \partial \Sigma$. 
We give some three-dimensional formulas for Turaev's loop operations using pure braids in $\Sigma$.

\subsection{Preliminaries}

We recall the notion of ``clasper'' introduced by Habiro in  \cite{Ha}. This notion will be useful in the sequel to describe local modifications on tangles.

A \emph{clasper} $C$ for a framed tangle $T \subset \Sigma \times [0,1]$ is a copy   in the exterior of $T$ of the following surface
$$
\labellist
\scriptsize \hair 2pt
 \pinlabel {leaf} [t] at 74 2
 \pinlabel {edge} [t] at 329 1
 \pinlabel {leaf} [t] at 604 2
\endlabellist
\centering
\includegraphics[scale=0.2]{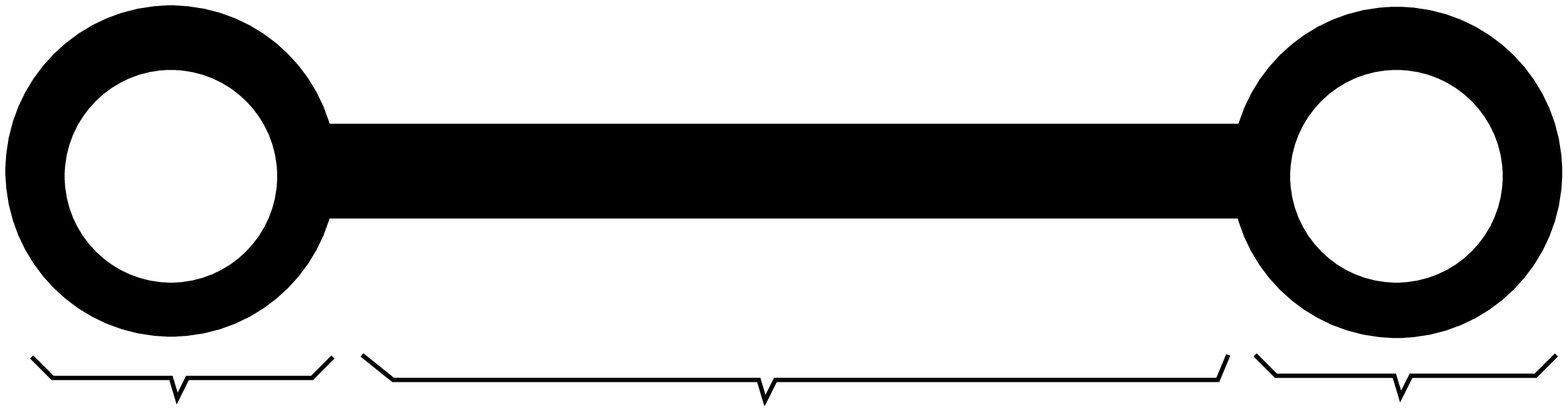}
$$
\vspace{0.1cm}

\noindent
decomposed into an \emph{edge} and two \emph{leaves}.
We assume that each leaf bounds a disk in $\Sigma \times [0,1]$ meeting $T$ transversely in finitely many points.
The \emph{surgery} along $C$ is the framed tangle $T_C$ obtained from $T$ by the following local modification:
$$
\labellist
\scriptsize\hair 2pt
 \pinlabel {$\cdots$} at 92 255
 \pinlabel {$\cdots$} at 99 18
 \pinlabel {$\cdots$}  at 574 251
 \pinlabel {$\cdots$}  at 581 16
 \pinlabel {$T$} at 650 239
\pinlabel {$C$}  at 360 165
 \pinlabel {$\leadsto$}  at 900 136
 \pinlabel {$\cdots$}  at 1125 255
 \pinlabel {$\cdots$}  at 1130 18
 \pinlabel {$\cdots$} at 1600 251
 \pinlabel {$\cdots$} at 1600 18
 \pinlabel {$T_C$}  at 1700 230
\endlabellist
\centering
\includegraphics[scale=0.2]{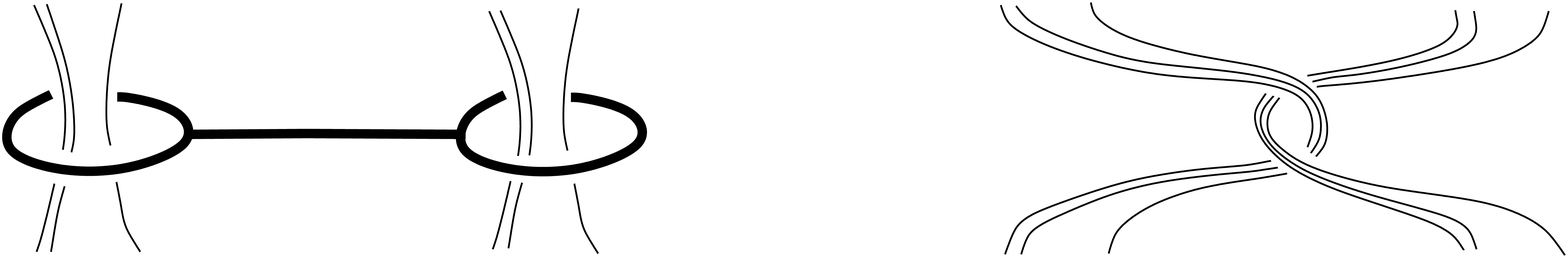}
$$
The clasper $C$ is said to be \emph{simple} if a single string of $T$ goes through each leaf of $T$.
(Note that, in the terminology of \cite{Ha}, our  ``claspers'' should be called ``strict basic claspers.'')

\subsection{A three-dimensional formula for Turaev's intersection pairing} \label{subsec:lambda-dim3}

Let $B_n(\Sigma)$ (respectively, $PB_n(\Sigma)$) denote the $n$-strand  braid group (respectively, the $n$-strand pure braid group)
in $\Sigma$ corresponding to the choice of $n$ points  in the interior of $\Sigma$. 
We will mainly need the case $n=2$, and denote  the two interior points  of $\Sigma$ by $u$ and $v$. 
We choose some simple arcs $\gamma_{\ast u}$  and $\gamma_{\ast v}$ connecting $\ast$ to $u$ and $v$ respectively,
such that $\gamma_{\ast u} \cap \gamma_{\ast v} =\{ \ast \}$ and the unit tangent vectors of  $\gamma_{\ast u}$  and $\gamma_{\ast v}$ at $\ast$ in this order
gives a negatively-oriented frame of $\Sigma$:
$$
\labellist
\scriptsize\hair 2pt
 \pinlabel {$\ast$}   at 364 -6
 \pinlabel {$u$} [b] at 225 248
 \pinlabel {$v$} [b] at 493 248
 \pinlabel {$\gamma_{\ast u}$} [tr] at 278 128
 \pinlabel {$\gamma_{\ast v}$} [tl] at 434 123
 \pinlabel {$\Sigma$} at 47 356
 \pinlabel {$\circlearrowleft$} at 640 356
\endlabellist
\centering
\includegraphics[scale=0.18]{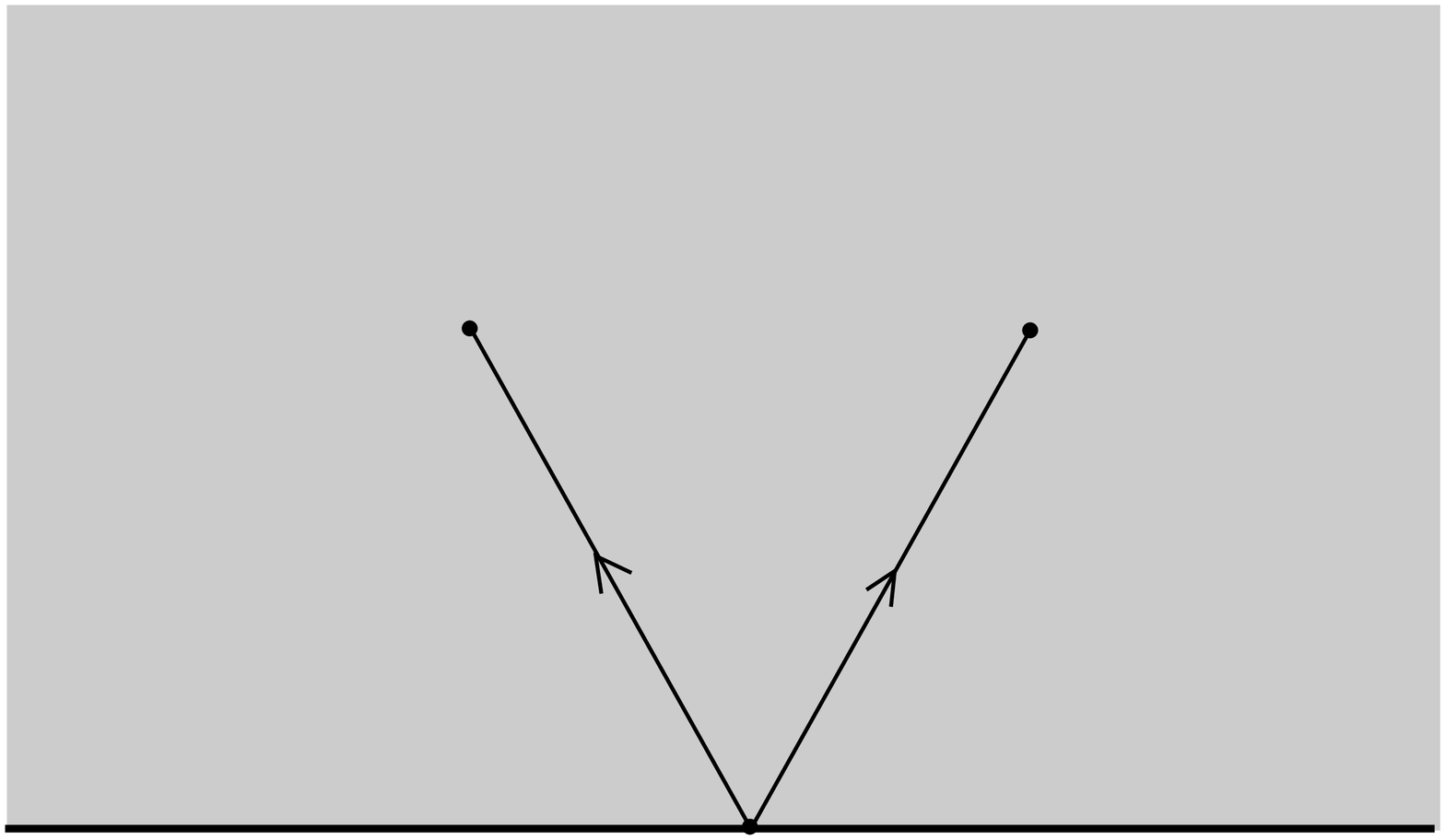}
$$
Set $\pi^u := \pi_1(\Sigma\setminus \{u\},\ast)$ and $\pi^v := \pi_1(\Sigma\setminus \{v\},\ast)$.
There is a natural embedding of ${\pi_1(\Sigma \setminus \{v\}, u)}$ into $PB_2(\Sigma)$ whose image consists of the braids with strand $v$ ``straight vertical'';
by composition with the isomorphism $ \pi^v \simeq\pi_1(\Sigma \setminus \{v\}, u) $ defined by the path $\gamma_{\ast u}$, we obtain  an embedding of groups
\begin{equation} \label{eq:pi^v}
\pi^v \hookrightarrow PB_2(\Sigma).
\end{equation}
 By exchanging the roles of $u$ and $v$,  we define in a similar way an embedding of groups 
\begin{equation} \label{eq:pi^u}
\pi^u \hookrightarrow PB_2(\Sigma).
\end{equation}
Fix now a point $\bullet \in \partial \Sigma$  before $\ast$ along the oriented boundary of $\Sigma$: 
by ``pushing'' the initial point of $\gamma_{\ast u}$ towards $\bullet$, we get a new simple arc $\gamma_{\bullet u}$ connecting $\bullet$ to~$u$.
Let $\Sigma^u$ be the surface obtained by removing from $\Sigma$ a regular {neighborhood} of $\gamma_{\bullet u}$:
$$
\labellist
\small\hair 2pt
 \pinlabel {$\Sigma^u$} [tl] at 23 379
 \pinlabel {\scriptsize $\gamma_{\bullet u}$}  at 147 173
 \pinlabel {$u$} [b] at 224 255
 \pinlabel {$v$} [b] at 494 255
 \pinlabel {\scriptsize $\gamma_{\ast v}$}   at 400 165
 \pinlabel {$\bullet$}   at 90 -5
 \pinlabel {$\ast$}   at 359 -5
\endlabellist
\centering
\includegraphics[scale=0.22]{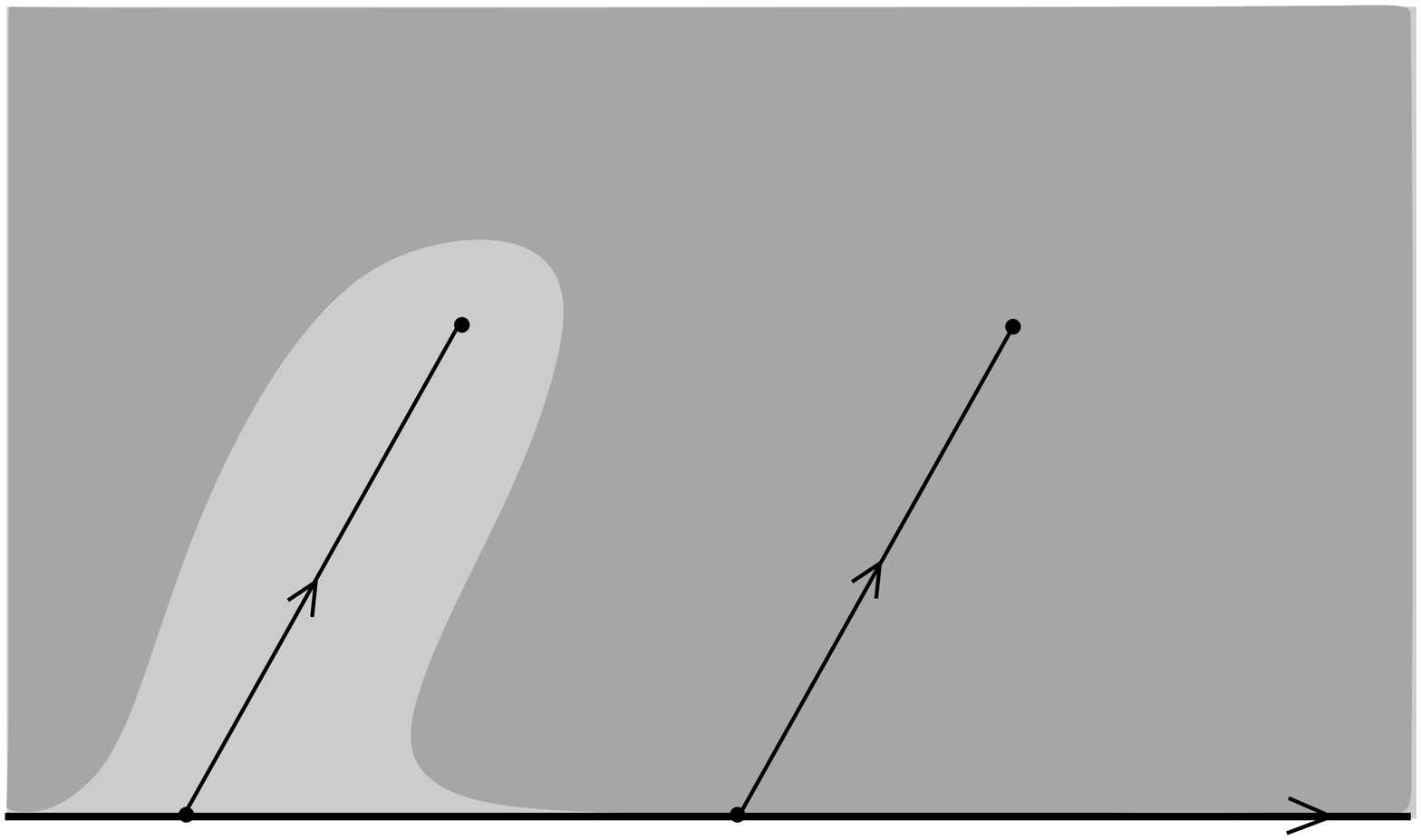}
$$
The obvious retraction $\Sigma \to \Sigma^u$ composed with the inclusion $\Sigma^u \subset \Sigma\setminus\{u\}$ induces an injective group homomorphism
$$
\iota^u:\pi \hookrightarrow \pi^u, 
$$
whose image consists of the  loops that do not cut $\gamma_{\bullet u}$.
Fix another point $\blacktriangle \in \partial \Sigma$ after $\ast$ along the oriented boundary of $\Sigma$:
by ``pushing'' the initial point of $\gamma_{\ast v}$ towards $\blacktriangle$, we get a new simple arc $\gamma_{\blacktriangle v}$ connecting $\blacktriangle$ to~$v$:
$$
\labellist
\small\hair 2pt
\pinlabel {$\Sigma^v$} [tl] at 23 379
 \pinlabel {$u$}  at 225 269
 \pinlabel {$v$}  at 493 269
 \pinlabel {$\ast$}  at 359 -5
 \pinlabel {\scriptsize $\blacktriangle$}  at 628 -5
 \pinlabel {\scriptsize $\gamma_{\ast u}$} at 253 152
 \pinlabel {\scriptsize $\gamma_{\blacktriangle v}$}  at 500 168
\endlabellist
\centering
\includegraphics[scale=0.22]{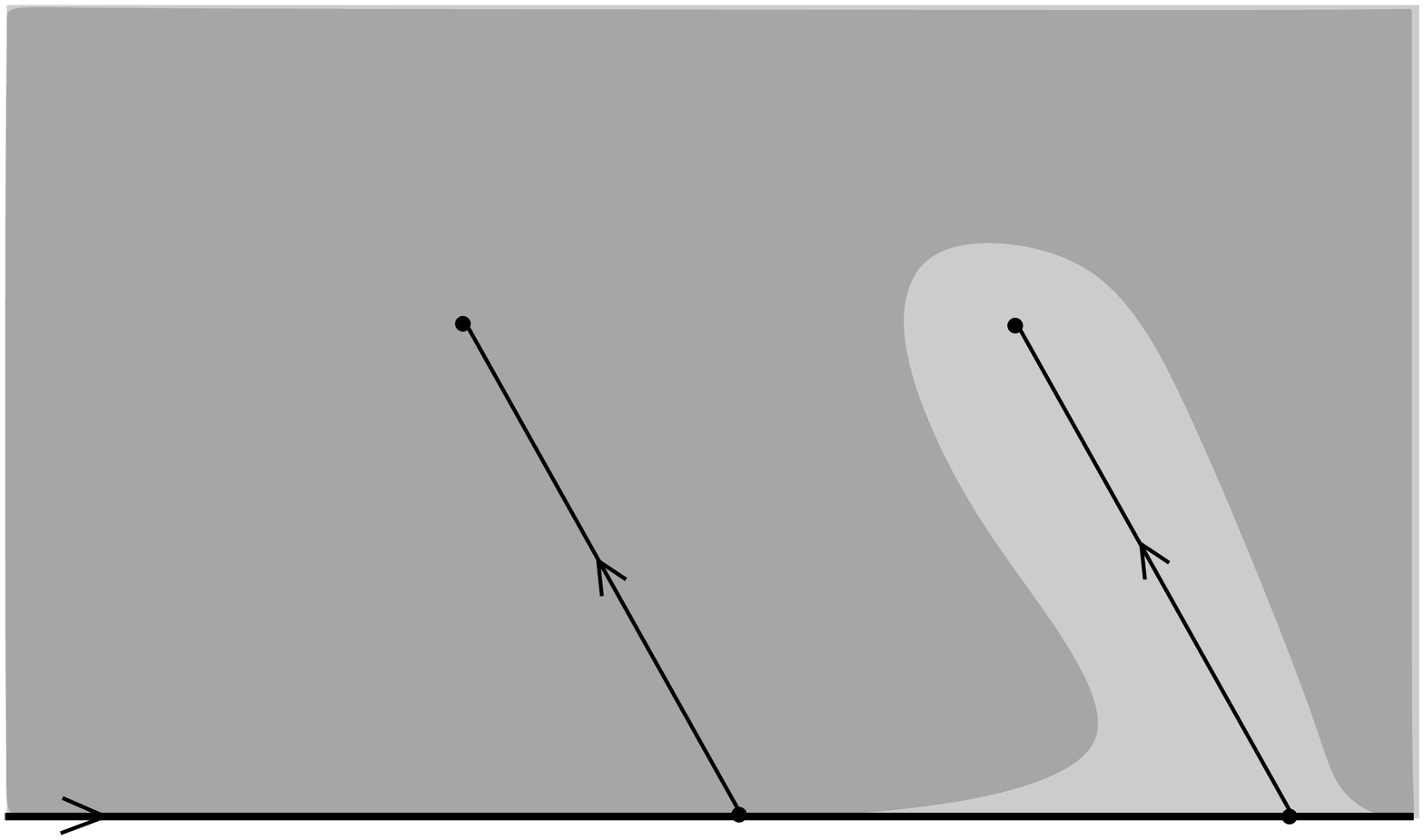}
$$
In a way similar to the definition of  $\iota^u$,  we define an injective  homomorphism of groups
$$
\iota^v:\pi \hookrightarrow \pi^v
$$
whose image consists of the  loops that do not cut $\gamma_{\blacktriangle v}$.
In the sequel, the homomorphisms \eqref{eq:pi^v}  and  \eqref{eq:pi^u}  are regarded as true inclusions, 
so that $\iota^u$ and $\iota^v$ also define some injective group homomorphims
$$
\iota^v:\pi \hookrightarrow PB_2(\Sigma) \quad  \hbox{and}  \quad \iota^u:\pi \hookrightarrow PB_2(\Sigma).
$$
We have the free-product decompositions
\begin{equation} \label{eq:fpd}
\pi^u = \iota^u(\pi) * F(z^u) \quad \hbox{and} \quad \pi^v = \iota^v(\pi) * F(z^v)
\end{equation}
where $(z^u)^{-1} \in \pi^u$ and $(z^v)^{-1}\in \pi^v$ are represented by the oriented boundaries of small neighborhoods of $\gamma_{\ast u}$ and $\gamma_{\ast v}$ respectively:
$$
\labellist
\scriptsize \hair 2pt
 \pinlabel {$\ast$}   at 359 -5
 \pinlabel {$u$} [b] at 225 248
 \pinlabel {$v$} [b] at 493 248
 \pinlabel {$z^u$} [tr] at 252 140
 \pinlabel {$z^v$} [tl] at 465 140
 \pinlabel {$\Sigma$} at 47 356
  \pinlabel {$\circlearrowleft$} at 640 356
\endlabellist
\centering
\includegraphics[scale=0.22]{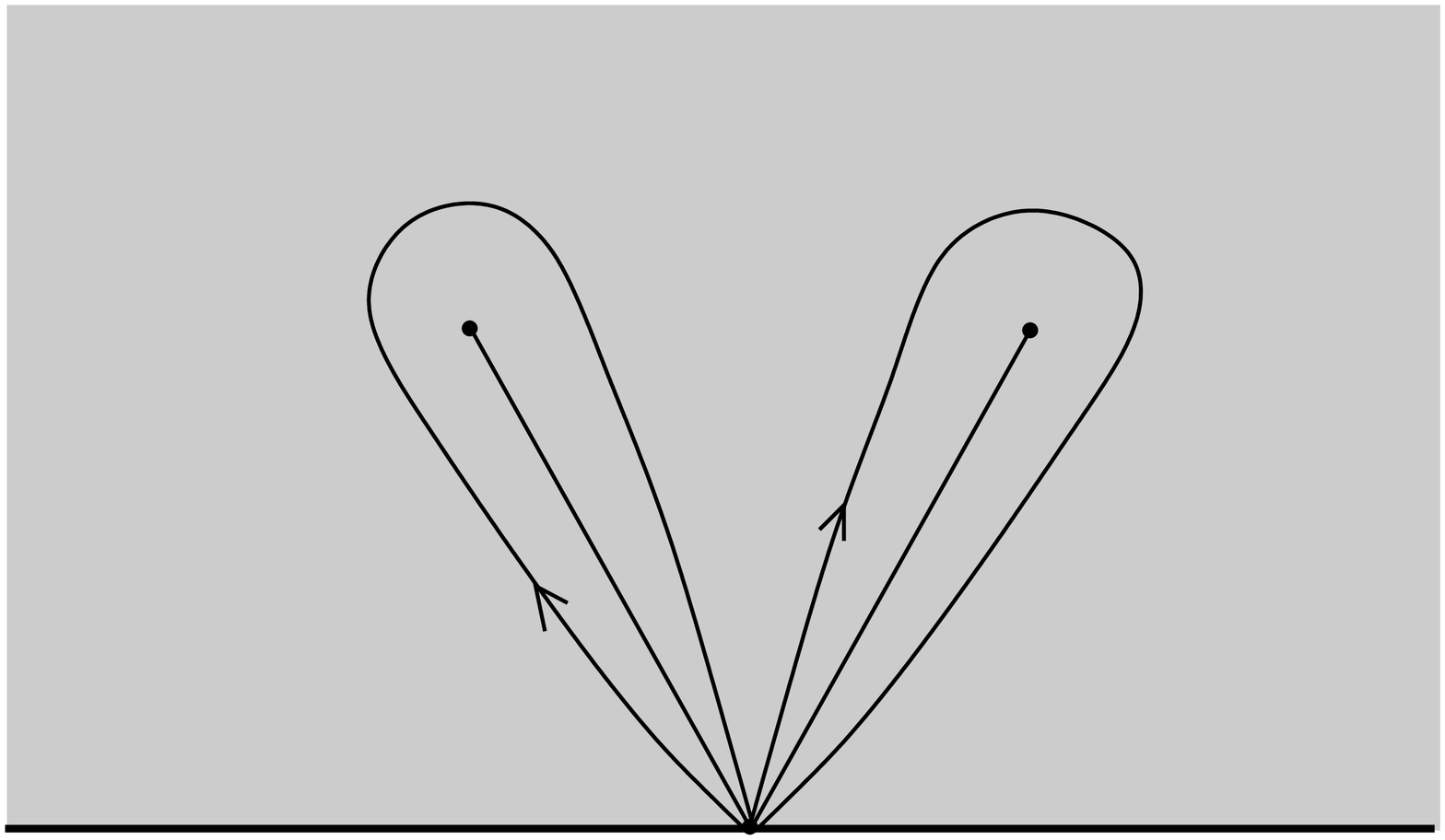}
$$
Note also that $z^u=z^v \in  \pi^u \cap \pi^v \subset PB_2(\Sigma)$; {this pure  braid} can  be obtained from the trivial $2$-strand braid by surgery along the following simple clasper:
$$
\labellist
\scriptsize\hair 2pt
 \pinlabel {$\ast$} [t] at 325 0
 \pinlabel {$\gamma_{\ast u}$} [r] at 284 37
 \pinlabel {$\gamma_{\ast v}$} [l] at 374 40
 \pinlabel {$\Sigma \times \{0\}$} [l] at 643 44
 \pinlabel {$\Sigma \times \{1\}$} [l] at 634 338
\endlabellist
\centering
\includegraphics[scale=0.18]{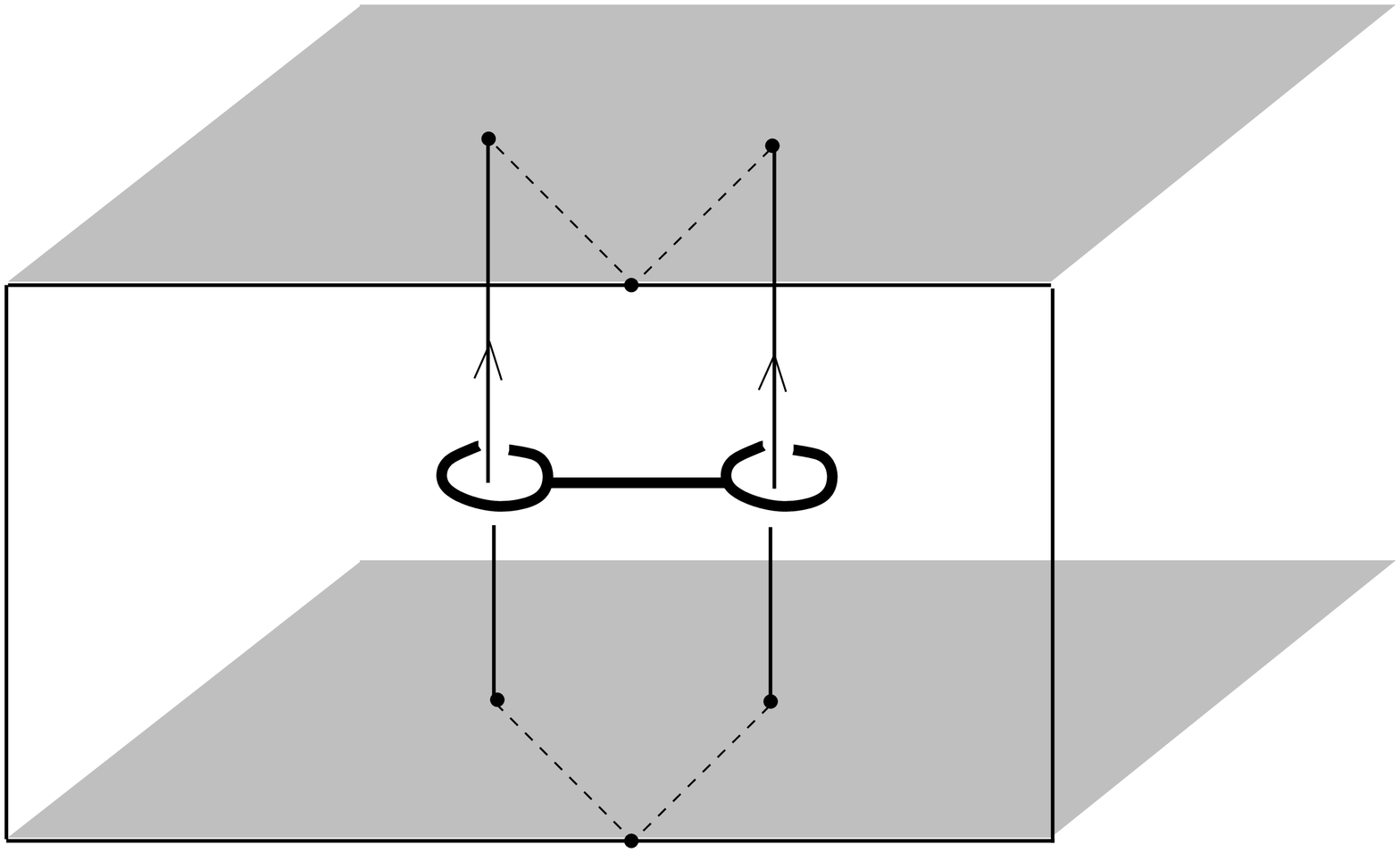}
$$
Finally, observe that the projections
$$
 \pi^u \stackrel{p^u}{\longrightarrow} \iota^u(\pi) \simeq \pi \quad \hbox{and} \quad  \pi^v \stackrel{p^v}{\longrightarrow} \iota^v(\pi) \simeq \pi
$$
defined by the free-product decompositions \eqref{eq:fpd}
are induced by the inclusions $\Sigma \setminus \{u\} \subset\Sigma$ and $\Sigma \setminus \{v\} \subset\Sigma$, respectively.

\begin{lemma} \label{lem:U_V}
For any $x \in \pi^u \cap \pi^v \subset PB_2(\Sigma)$, we have
$$
 p^v\left( \frac{\partial x}{\partial z^v}\right) =   \overline{p^u\left( \frac{\partial x}{\partial z^u}\right) } = \overline{ p^v\left( \frac{\partial x^\times}{\partial z^v} \right)} \in \K[\pi]
$$
where $x^\times := \sigma x \sigma^{-1} $ is the conjugate of $x$ by the  elementary positive braid~$\sigma \in B_2(\Sigma)$:
$$
\labellist
\scriptsize\hair 2pt
 \pinlabel {$\ast$} [t] at 325 0
 \pinlabel {$\gamma_{\ast u}$} [r] at 284 37
 \pinlabel {$\gamma_{\ast v}$} [l] at 374 40
 \pinlabel {$\Sigma \times \{0\}$} [l] at 643 44
 \pinlabel {$\circlearrowleft$}   at 600 100
 \pinlabel {$\Sigma \times \{1\}$} [l] at 634 338
  \pinlabel {$\sigma$}  at 400 209
\endlabellist
\centering
\includegraphics[scale=0.18]{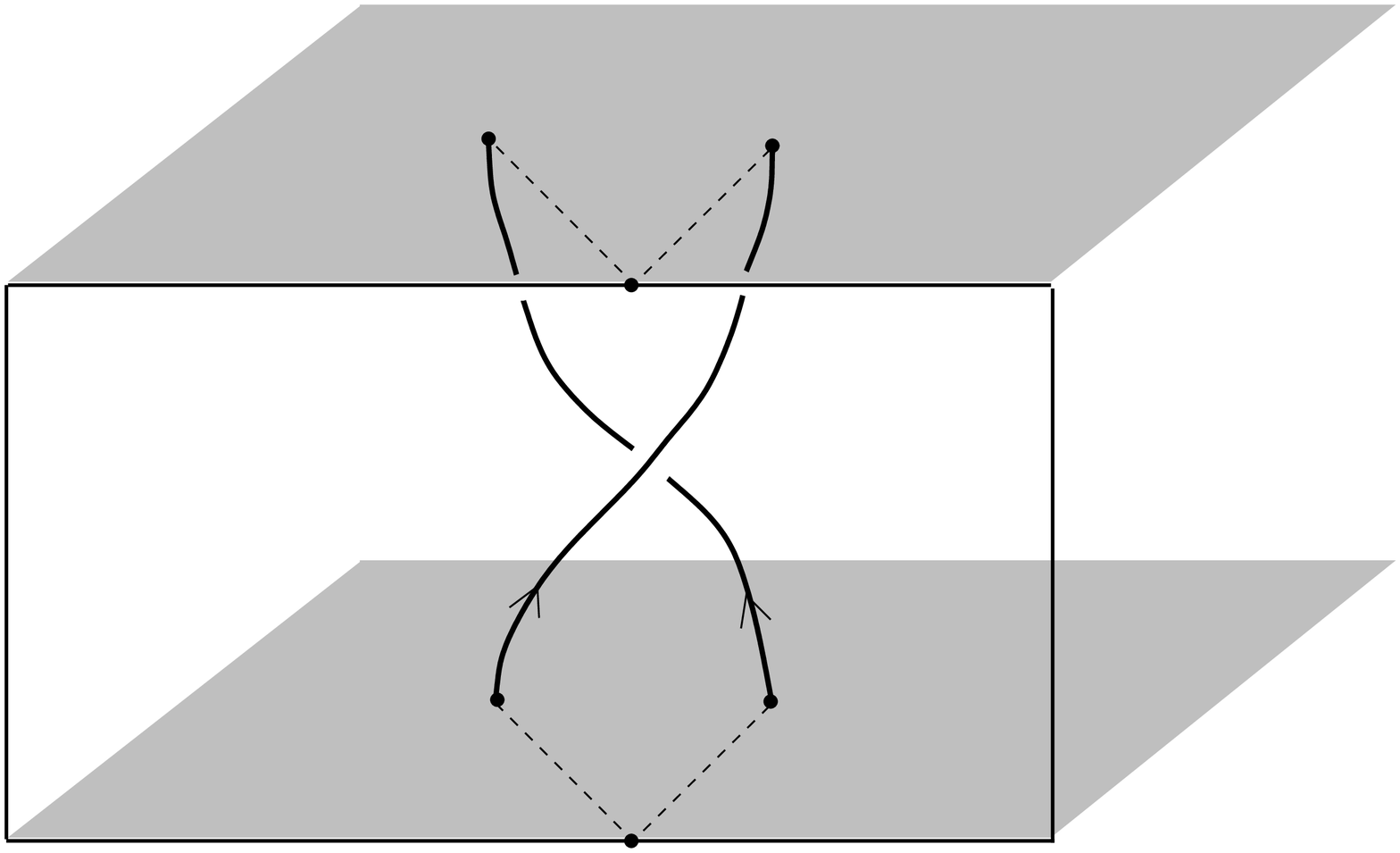}
$$
\end{lemma}

\begin{proof}
We start with the following observation. Let $C$ be a simple clasper for the trivial $2$-strand braid $T\subset \Sigma \times [0,1]$.
Surgery along $C$ produces a $2$-strand string-link $T_C$. 
We can regard $T_C$ as a string-knot in $(\Sigma \setminus \{v\}) \times [0,1]$, which defines an element
$$
[T_C]^v \in \pi_1(\Sigma\setminus\{v\}, u) \mathop{\simeq}^{\gamma_{\ast u}} \pi^v,
$$
or, alternatively,  we can regard it as a string-knot in $(\Sigma \setminus \{u\}) \times [0,1]$, which defines an element
 $$
[T_C]^u \in \pi_1(\Sigma\setminus\{u\}, v) \mathop{\simeq}^{\gamma_{\ast v}} \pi^u.
$$
For instance, if $T_C$ turns out to be a $2$-strand pure braid, then we have $[T_C]^v = T_C$ and $[T_C]^u = T_C$ 
via the inclusions \eqref{eq:pi^v} and \eqref{eq:pi^u} respectively.
Let $c$ be a loop in $\Sigma \times [0,1]$ based at $(\ast ,0)$ which {follows} the edge of $C$ (oriented from the strand $u$ to the strand $v$) in the following way:
$$
\labellist
\small\hair 2pt
 \pinlabel {$u$} [b] at 249 370
 \pinlabel {$v$} [b] at 397 370
 \pinlabel {$c$}  at 320 79
 \pinlabel {$C$}  at 450 112
 \pinlabel {$T$} at 425 260
  \pinlabel {$\ast$} [t] at 325 0
   \pinlabel {$\Sigma \times \{0\}$} [l] at 643 44
 \pinlabel {$\Sigma \times \{1\}$} [l] at 634 338
\endlabellist
\centering
\includegraphics[scale=0.25]{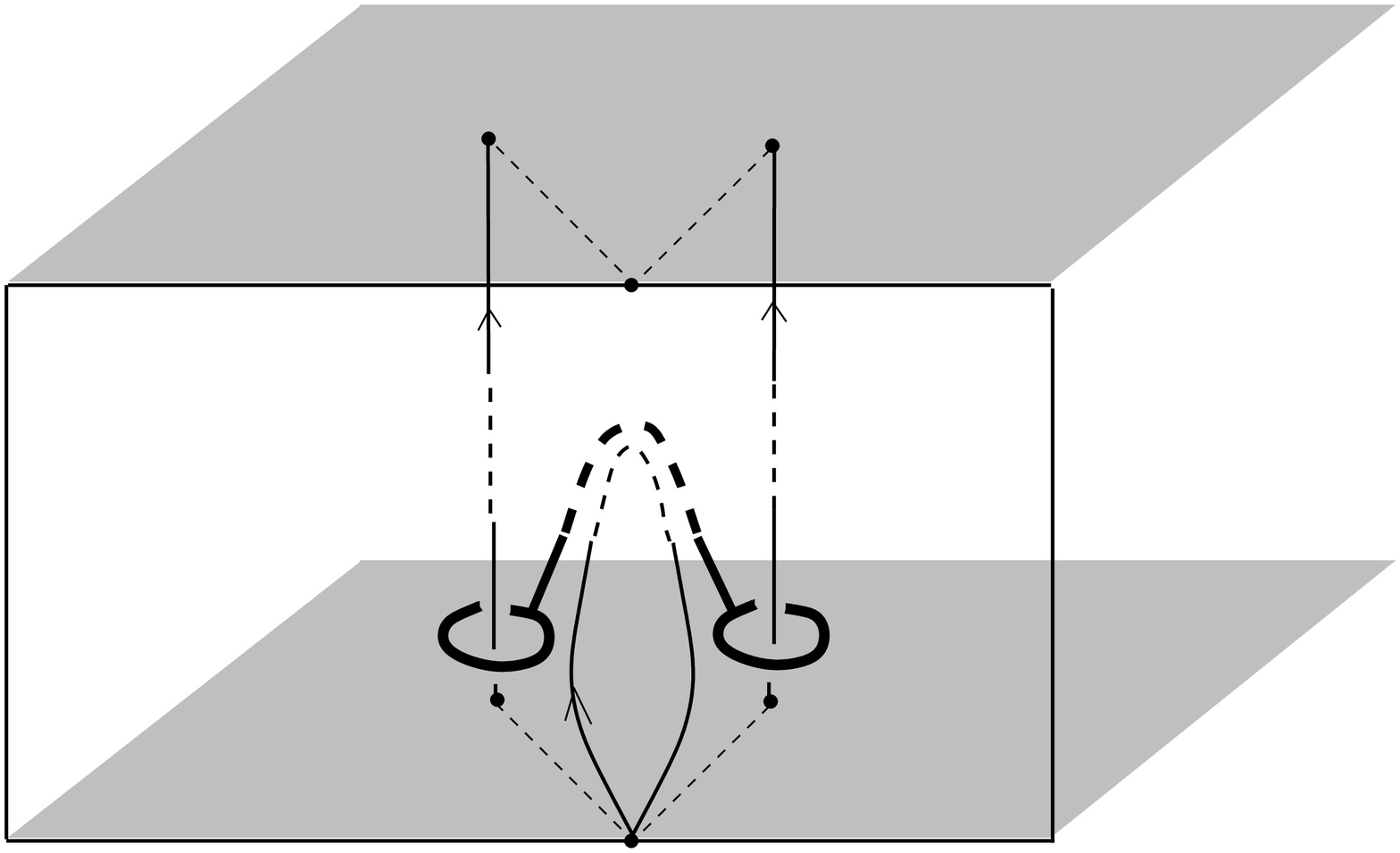}
$$

\vspace{0.2cm}
\noindent
This loop represents an element $c\in\pi_1\big((\Sigma \times [0,1]) \setminus T,(\ast,0)\big) \simeq \pi_1(\Sigma\setminus \{u,v\},\ast)=:\pi^{uv}$. 
Then, our observation is that
\begin{equation} \label{eq:observation}
[T_C]^v =  p_u(c)\, z^v\,   (p_u(c))^{-1} \in \pi^v
\quad \hbox{and} \quad 
[T_C]^u =  (p_v (c))^{-1}\, z^u\,   p_v(c) \in \pi^u
\end{equation}
where $p_u: \pi^{uv} \to \pi^v$ and $p_v: \pi^{uv} \to \pi^u$ are the group homomorphisms 
induced by the inclusions $\Sigma \setminus \{u,v\} \subset \Sigma \setminus \{v\}$ and $\Sigma \setminus \{u,v\} \subset \Sigma \setminus \{u\}$, respectively.

Let now $x$ be an arbitrary element of  $\pi^u \cap \pi^v {\subset PB_2(\Sigma)}$. Then $x\in \pi^v$ satisfies $p^v(x)=1$, so that $x$ belongs to the normal subgroup of $\pi^v$ generated by $z^v$:
therefore, there exist a finite ordered subset $\{s_j\}_{j \in J}$ of $\pi^v$ and some signs $\varepsilon_j \in \{+1,-1\}$ for all $j\in J$ such that
\begin{equation} \label{eq:formula_x}
x = \prod_{j\in J} s_j \, (z^v)^{\varepsilon_j}\, s_{j}^{-1} \in \pi^v.
\end{equation}
By application of Lemma \ref{lem:Fox}, we obtain
\begin{equation} \label{eq:res1}
p^v\left(\frac{\partial x}{ \partial z^v} \right)= \sum_{j \in J} \varepsilon_j\, p^v(s_j).
\end{equation}
For any $j\in J$, the pure braid $s_j \, z^v\, s_j^{-1} \in \pi^u \cap \pi^v \subset PB_2(\Sigma)$ can be obtained from the trivial braid $T$
by surgery along a simple clasper, since $z^v$ has this property. We deduce from our initial observation that there exists a $c_j\in \pi^{uv}$ such that
$$
s_j \, z^v\, s_j^{-1} = p_u(c_j)\, z^v\,   (p_u(c_j))^{-1} \in \pi^v
\quad \hbox{and} \quad 
s_j \, z^v\, s_j^{-1} =  (p_v (c_j))^{-1}\, z^u\,   p_v(c_j) \in \pi^u.
$$
By applying Lemma \ref{lem:Fox} to the first identity, we obtain $p^v(s_j)=p(c_j)$ where $p: \pi^{uv} \to \pi$
is the group homomorphism induced by the inclusion $\Sigma \setminus \{u,v\} \subset \Sigma$.
Setting $r_j := (p_v(c_j))^{-1} \in \pi^u$, the second identity writes $s_j \, z^v\, s_j^{-1} = r_j z^u r_j^{-1}$.
Note that $r_j$ satisfies $p^u(r_j)= (p(c_j))^{-1} = (p ^v (s_j))^{-1} \in \pi$.

Thus we have found for all $j \in J$ an element $r_j\in \pi^u$ such that $p^u(r_j)= (p^v(s_j))^{-1}\in \pi$ and
$$
x \by{eq:formula_x} \prod_{j\in J} r_j\, (z^u)^{\varepsilon_j}\, r_j^{-1} \in \pi^u.
$$
This second formula for $x$ has two consequences. First, by a new application of Lemma \ref{lem:Fox}, we deduce that
\begin{equation} \label{eq:res2}
p^u \left(\frac{\partial x}{ \partial z^u}\right) = \sum_{j \in J} \varepsilon_j\, p^u(r_j) = \sum_{j \in J} \varepsilon_j\, (p^v(s_j))^{-1}.
\end{equation}
Second, it implies that
\begin{eqnarray*}
x^\times \ = \  \sigma x \sigma^{-1} & = & \prod_{j\in J}  ( \sigma r_j \sigma^{-1})\,  (\sigma (z^u)^{\varepsilon_j} \sigma^{-1})\,  (\sigma r_j^{-1} \sigma^{-1}) \\
&=& \prod_{j\in J}  ( \sigma r_j \sigma^{-1})\,  (z^v)^{\varepsilon_j} \,  (\sigma r_j^{-1} \sigma^{-1}).
\end{eqnarray*}
The map $\pi^u \to \pi^v$ defined by $y \mapsto \sigma y \sigma^{-1}$ is the group homomorphism
defined by $z^u \mapsto z^v$ and $\iota^u(a) \mapsto (z^v)^{-1} \iota^v(a) z^v$ for all $a\in \pi$:
in particular, we have $p^u(y) = p^v(\sigma y \sigma^{-1})$ for all $y\in \pi^u$.
By applying Lemma \ref{lem:Fox} a third and last time, we obtain 
\begin{equation} \label{eq:res3}
p^v\left( \frac{\partial\, x^\times}{\partial z^v}\right) = \sum_{j\in J} \varepsilon_j\, p^v( \sigma r_j \sigma^{-1})
= \sum_{j\in J} \varepsilon_j\, p^u(r_j) = \sum_{j\in J} \varepsilon_j\, (p^v(s_j))^{-1}.
\end{equation}
The lemma now  follows from \eqref{eq:res1}, \eqref{eq:res2} and \eqref{eq:res3}.
\end{proof}

We now define a  bilinear pairing $\eta': \K[\pi] \times \K[\pi] \to \K[\pi]$ by setting, for any $a,b\in \pi$,
$$
\eta'(a,b) := p^v\left(\frac{\partial \left[ \iota^u(b^{-1}),\iota^v(a)\right]}{\partial z^v}\right)
$$
where $\left[ \iota^u(b^{-1}),\iota^v(a)\right] \in \pi^u \cap \pi^v$ denotes the commutator in the group $PB_2(\Sigma)$.

\begin{lemma}
For any $a,b\in \pi$, we have
$$
\eta'(a,b) =  p^v\left( \frac{\partial\, \iota^{u}(b^{-1}) \iota^v(a) \iota^{u}(b) }{\partial z^v}\right) =
\overline{p^u\left(\frac{\partial \left[ \iota^u(b^{-1}),\iota^v(a)\right]}{\partial z^u}\right)}  = \overline{b^{-1}p^u\left(\frac{\partial\,  \iota^v(a) \iota^{u}(b) \iota^{v}(a^{-1})}{\partial z^u}\right) }.
$$
Moreover, $\eta'$ is a Fox pairing in $\K[\pi]$.
\end{lemma}

\begin{proof}
We prove the first identity: for any $a,b\in \pi$,
\begin{eqnarray*}
p^v\left(\frac{\partial \left[ \iota^u(b^{-1}),\iota^v(a)\right]}{\partial z^v}\right) &=& p^v\left( \frac{\partial\, \big(\iota^{u}(b^{-1}) \iota^v(a) \iota^{u}(b)\big)\,  \iota^v(a^{-1})}{\partial z^v}\right) \\
&=& p^v\left( \frac{\partial\, \iota^{u}(b^{-1}) \iota^v(a) \iota^{u}(b) }{\partial z^v} +  \iota^{u}(b^{-1}) \iota^v(a) \iota^{u}(b) \frac{\partial \iota^v(a^{-1})}{\partial z^v} \right) \\
&=&  p^v\left( \frac{\partial\, \iota^{u}(b^{-1}) \iota^v(a) \iota^{u}(b) }{\partial z^v}\right).
\end{eqnarray*}
The third identity is proved similarly, and the second identity follows  from Lemma \ref{lem:U_V}.

We now show that $\eta'$ is a Fox pairing using the previous identities.
 For any $a_1,a_2,b\in \pi$, we have
\begin{eqnarray*}
\eta'(a_1a_2,b) &=&  p^v\left( \frac{\partial\, \iota^{u}(b^{-1}) \iota^v(a_1)  \iota^v(a_2) \iota^{u}(b) }{\partial z^v}\right) \\
&=&  p^v\left( \frac{\partial\, \big(\iota^{u}(b^{-1}) \iota^v(a_1)  \iota^u(b)\big)\, \big(\iota^u(b^{-1}) \iota^v(a_2) \iota^{u}(b) \big)}{\partial z^v}\right) \\ 
&=&  p^v\left( \frac{\partial\, \iota^{u}(b^{-1}) \iota^v(a_1)  \iota^u(b)}{\partial z^v}\right) + 
  p^v\left(  \iota^{u}(b^{-1}) \iota^v(a_1)  \iota^u(b) \frac{\partial\, \iota^u(b^{-1}) \iota^v(a_2) \iota^{u}(b) }{\partial z^v}\right) \\
 &=& \eta'(a_1,b) + a_1 \eta'(a_2,b)
\end{eqnarray*}
and, for any $a,b_1,b_2 \in \pi$, we have
\begin{eqnarray*}
\overline{\eta'(a,b_1 b_2)} &=& (b_1 b_2)^{-1}\, p^u\left(\frac{\partial\,  \iota^v(a) \iota^{u}(b_1b_2)  \iota^{v}(a^{-1})}{\partial z^u}\right) \\
&=&  b_2^{-1} b_1^{-1}\, p^u\left(\frac{\partial\, \big(  \iota^v(a) \iota^{u}(b_1)   \iota^{v}(a^{-1})\big)\,  \big(\iota^v(a)\iota^u(b_2)  \iota^{v}(a^{-1})\big)}{\partial z^u}\right) \\
&=& b_2^{-1} b_1^{-1}\, p^u\left(\frac{\partial\,  \iota^v(a) \iota^{u}(b_1)   \iota^{v}(a^{-1}) }{\partial z^u}
+  \iota^v(a) \iota^{u}(b_1)   \iota^{v}(a^{-1}) \, \frac{ \partial\,  \iota^v(a)\iota^u(b_2)  \iota^{v}(a^{-1})}{\partial z^u}\right)  \\
&=& b_2^{-1} \overline{\eta'(a,b_1)} +  \overline{\eta'(a,b_2)}
\end{eqnarray*}
which is equivalent to the identity $\eta'(a,b_1b_2) = \eta'(a,b_1) b_2 + \eta'(a,b_2)$.
\end{proof}

\begin{theorem} \label{th:3d-eta}
We have $\eta'=\eta$.
\end{theorem}

\begin{proof}
Let $a,b\in \pi$. 
Recall that we have three points $\bullet < \ast < \blacktriangle$ along the oriented boundary component $\nu$ of $\Sigma$ that contains $\ast$.
Let $\alpha_u$ and $\beta_v$ be some loops based at $u$ and $v$, respectively,
such that $\gamma_{\ast u}\alpha_u \overline{\gamma}_{u \ast}$ and $\gamma_{\ast v}\beta_v \overline{\gamma}_{v \ast}$ represent $a$ and $b$ respectively,
$\alpha_u$ and  $\beta_v$ meet transversely in a finite set of simple points, and we have the following local picture in a neighborhood of $\ast$:
$$
\labellist
\scriptsize \hair 2pt
 \pinlabel {$\bullet$} [t] at 92 1
 \pinlabel {$\ast$} [t] at 359 2
 \pinlabel {$\blacktriangle$} [t] at 630 1
 \pinlabel {$u$} [b] at 224 255
 \pinlabel {$v$} [b] at 493 256
 \pinlabel {$\alpha_u$} [r] at 126 310
 \pinlabel {$\beta_v$} [l] at 596 316
 \pinlabel {$\gamma_{\bullet u}$} [r] at 159 132
 \pinlabel {$\gamma_{\ast u}$} [l] at 280 144
 \pinlabel {$\gamma_{\ast v}$} [r] at 445 143
 \pinlabel {$\gamma_{\blacktriangle v}$} [l] at 555 149
\endlabellist
\centering
\includegraphics[scale=0.2]{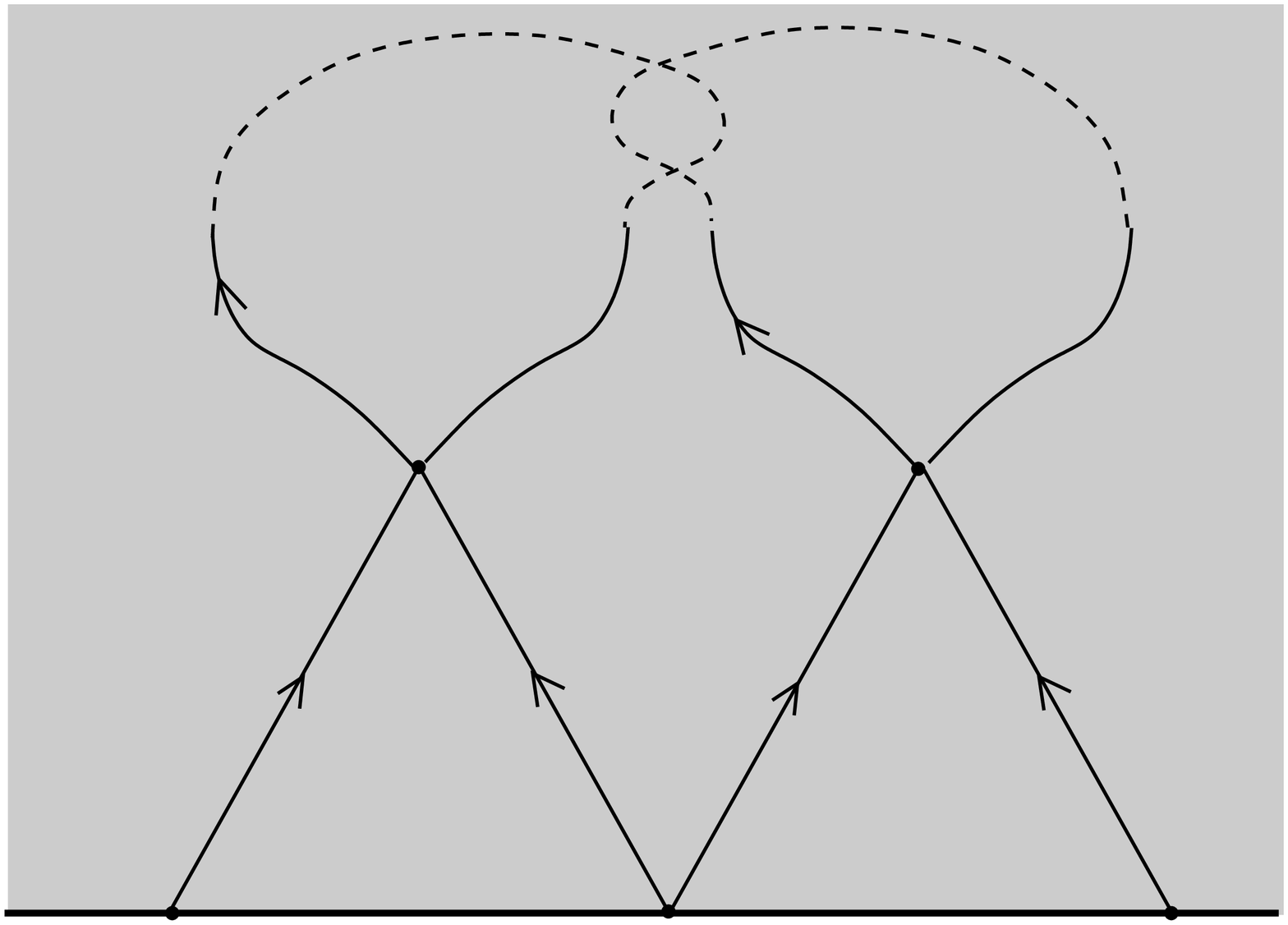}
$$
Set
$\alpha := \gamma_{\bullet u} \alpha_{u} \overline{\gamma}_{u \bullet}$
and $\beta  := \gamma_{\blacktriangle v} \beta_{v} \overline{\gamma}_{v \blacktriangle}$:
then $\eta(a,b) \in \Z[\pi]$ is given by formula~\eqref{eq:def-eta}.

In order to compute $\eta'(a,b)$, we need to consider  $\left[\iota^u(b^{-1}), \iota^v(a)\right] \in PB_2(\Sigma)$.
This  pure braid on two strands $u$ and $v$ can be schematically represented as follows:
$$
\begin{array}{cc}
 \overline\alpha_{u} & \uparrow_v \\
 \uparrow_u & \beta_{v} \\
{\alpha}_{u} & \uparrow_v\\
 \uparrow_u & \overline{\beta}_{v} 
\end{array}
$$
This pure braid  is isotopic to the $2$-strand string-link
$$
\begin{array}{cc}
 \overline\alpha_{u} & \uparrow_v \\
{\alpha}_{u} & \uparrow_v\\
 \delta_{u} & \beta_{v} \\
 \uparrow_u & \overline{\beta}_{v} 
\end{array}
$$
whose second ``slice'' $\delta_{u}\, {\beta}_{v}$  is obtained from $ \uparrow_u \beta_{v}$
by several surgeries along simple claspers:
specifically, there is one clasper for every intersection point $p\in \alpha \cap \beta=\alpha_u \cap \beta_v$. 
The string-link consisting of the first and second slices in the above diagram
$$
\begin{array}{cc}
\delta_{u} & \beta_{v} \\
 \uparrow_u & \overline{\beta}_{v} 
\end{array}
$$  
defines a loop in $\Sigma \setminus \{v\}$ based at $u$ and, 
using the  observation  at the beginning of the proof of Lemma~\ref{lem:U_V},
 we see that this loop  represents  the following product of conjugates of  $Z^v := \overline{\gamma}_{u\ast }z^v\gamma_{\ast u}$
 in the group $\pi_1(\Sigma \setminus \{v\},u)$:
$$
\prod_{p \in \alpha \cap \beta} K_p \,
(Z^v)^{\varepsilon_{p}(\alpha, \beta)}\,  K_p^{-1} \ \in \pi_1(\Sigma \setminus \{v\},u).
$$
Here the product over $\alpha \cap \beta$ is ordered following the positive direction of $\alpha$ (starting from~$u$); for every $p \in \alpha \cap \beta$,
we define the sign $\varepsilon_p(\alpha,\beta)$ as in formula \eqref{eq:def-eta}
and $K_p$ is a certain element of $\pi_1(\Sigma \setminus\{v\}, u)$ which   is mapped to $\alpha_{up} {\beta}_{pv} \overline\gamma_{v\ast} \gamma_{\ast u} \in \pi_1(\Sigma,u)$ by the canonical projection.
We deduce that $\left[\iota^u(b^{-1}), \iota^v(a)\right]$ as an element of $\pi^v=\pi_1(\Sigma \setminus \{v\}, \ast)$ is equal to
$$
\prod_{p \in \alpha \cap \beta} {k_p}\, (z^v)^{\varepsilon_{p}(\alpha, \beta)}\,  k_p^{-1}
$$
where $k_p \in  \pi^v$ has the property to be sent to  $\gamma_{\ast u }\alpha_{up} {\beta}_{pv} \overline{\gamma}_{v\ast}\in \pi$ by the projection $p^v$.
It follows from Lemma \ref{lem:Fox} that
$$
\eta'(a,b) =  p^v\left(\frac{\partial \left[ \iota^u(b^{-1}),\iota^v(a)\right]}{\partial z^v}\right) 
= \sum_{p\in \alpha \cap \beta} \varepsilon_{p}(\alpha, \beta)\,  p^v(k_p) = \eta(a,b) \in \K[\pi].
$$

\up
\end{proof}

\begin{example}
Assume that $\Sigma$ is a compact connected oriented surface of genus $1$ with one boundary component.
Represent $\Sigma$ as a holed square whose opposite edges are linearly identified,
and let $a,b\in \pi$ be represented by the following loops:
$$
\labellist
\small\hair 2pt
 \pinlabel {$\circlearrowleft$}  at 400 405
 \pinlabel {$\ast$}  at 295 180
 \pinlabel {$b$}  at 260 298
 \pinlabel {$a$} at 155 240
\endlabellist
\centering
\includegraphics[scale=0.25]{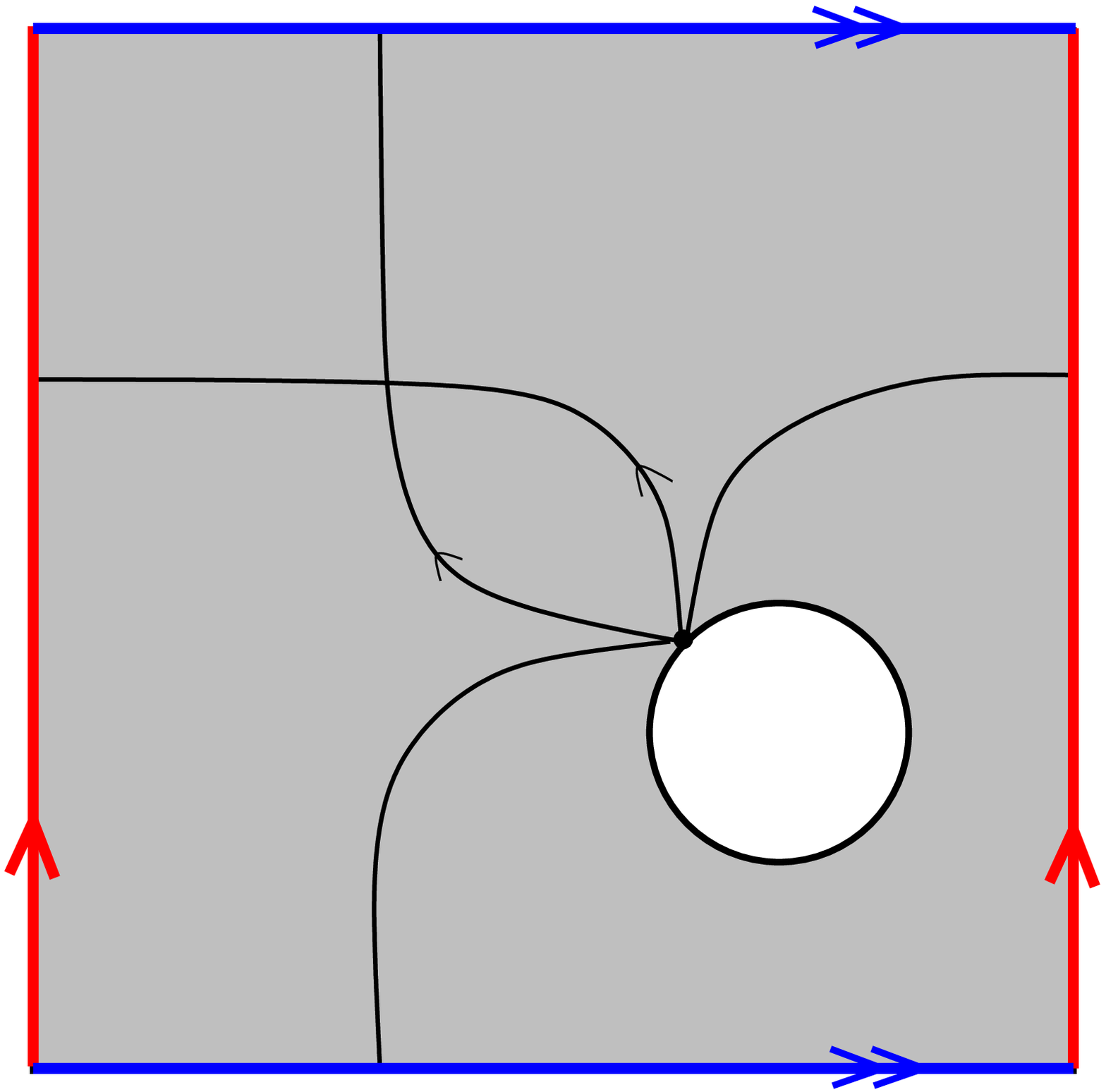}
$$
Then the pure braid $\left[ \iota^u(b^{-1}),\iota^v(a)\right]$ is represented by the following projection diagrams:
$$
\begin{array}{c}
\labellist
\scriptsize\hair 4pt
 \pinlabel {$\ast$}  at 295 180
 \pinlabel {$\circlearrowleft$} [tr] at 436 436
 \pinlabel {$u$} [t] at 154 270
 \pinlabel {$v$} [r] at 216 267
 \pinlabel {$\gamma_{\ast u}$} [t] at 210 240
 \pinlabel {$\gamma_{\ast v}$} [l] at 243 229
\endlabellist
\centering
\includegraphics[scale=0.3]{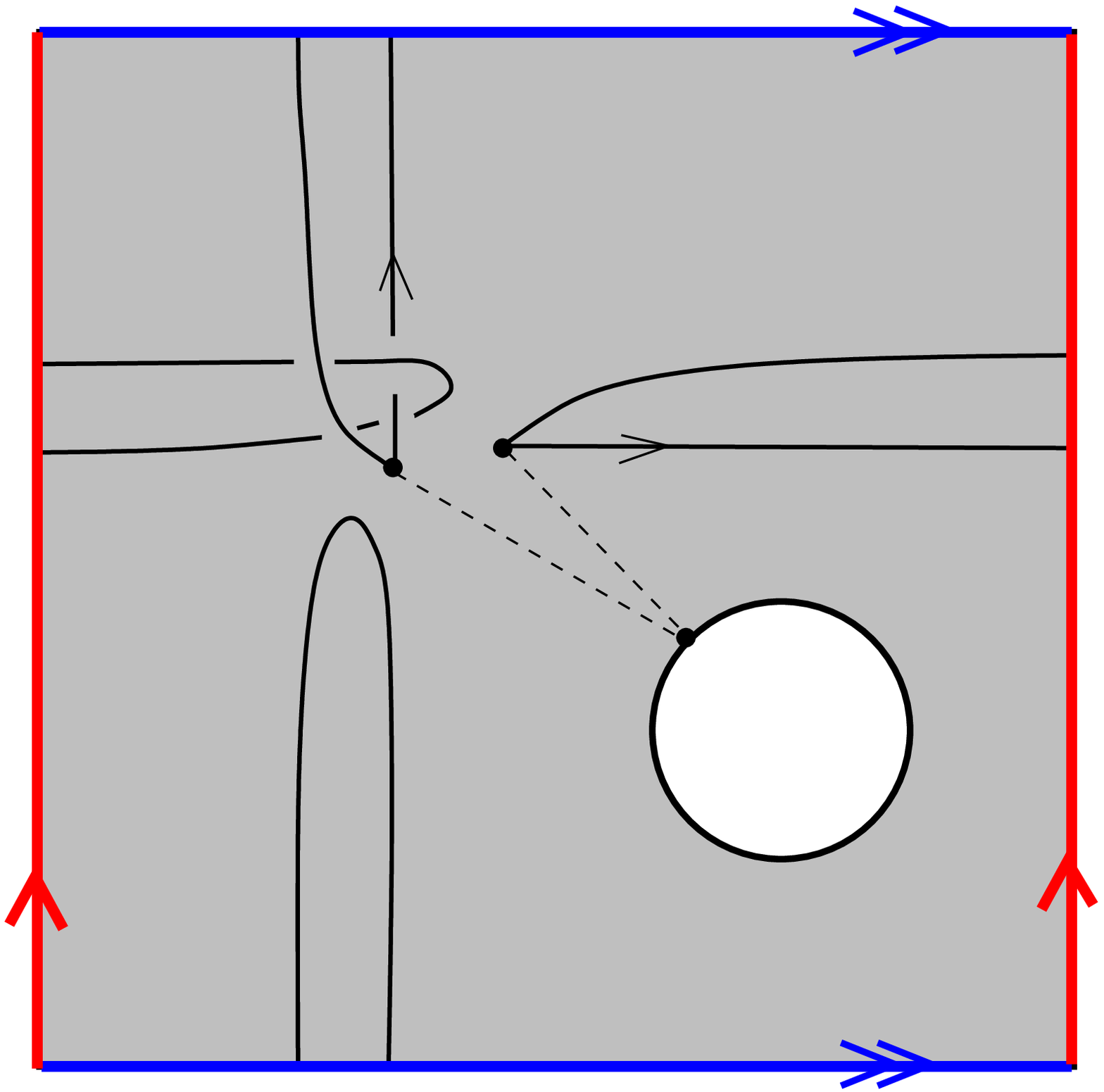}
\end{array}
\quad \stackrel{\hbox{\scriptsize isotopic}}{\simeq} \quad
\begin{array}{c}
\labellist
\scriptsize\hair 4pt
 \pinlabel {$\ast$}  at 295 180
 \pinlabel {$\circlearrowleft$} [tr] at 436 436
 \pinlabel {$u$} [t] at 154 266
 \pinlabel {$v$} [r] at 216 267
 \pinlabel {$\gamma_{\ast u}$} [t] at 206 240
 \pinlabel {$\gamma_{\ast v}$} [l] at 243 229
\endlabellist
\centering
\includegraphics[scale=0.3]{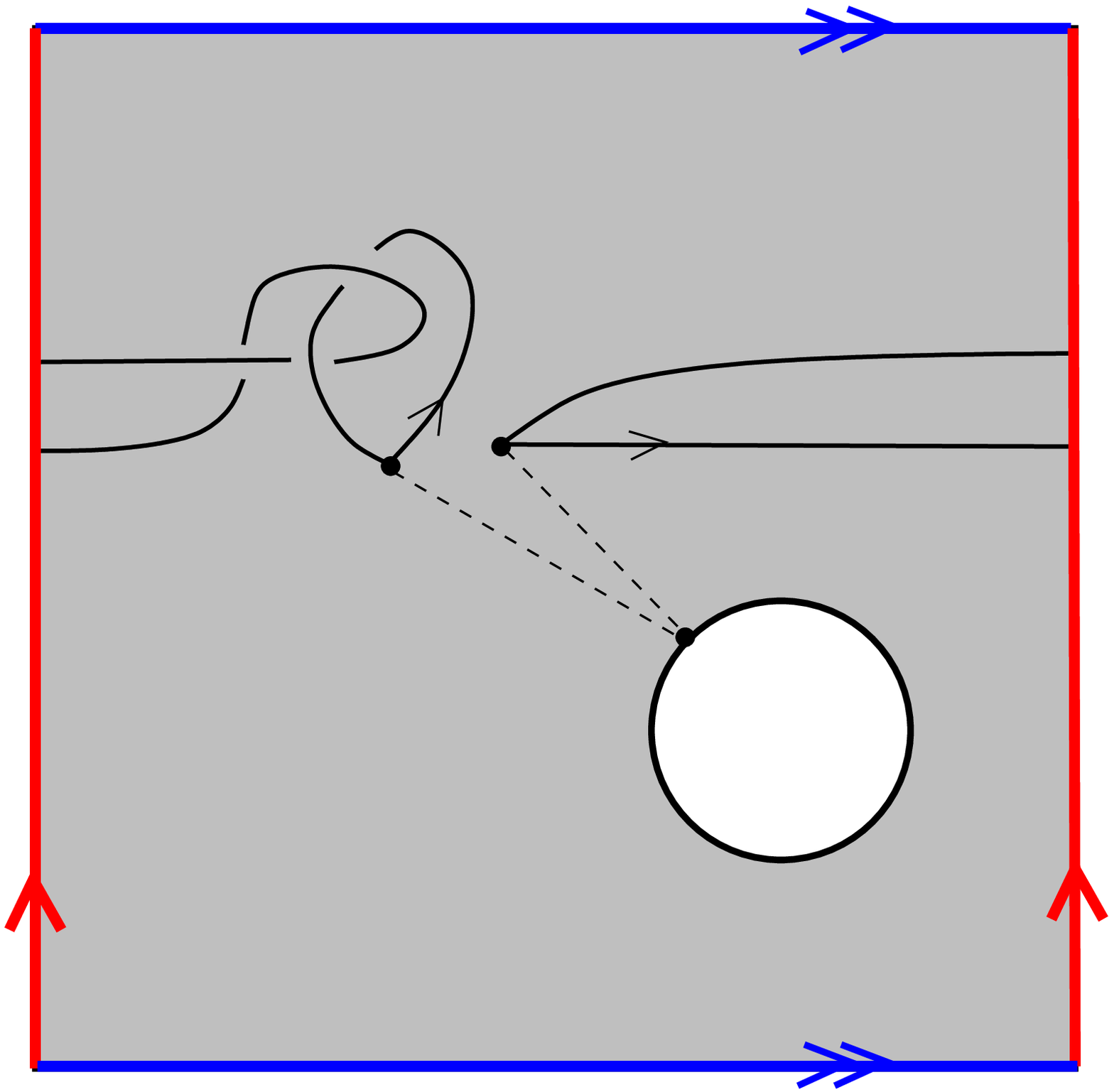}
\end{array}
$$
Therefore, we have  $\left[ \iota^u(b^{-1}),\iota^v(a)\right] =  \iota^u(b^{-1})\, z^u\, \iota^u(b)$ which implies that
$$
\eta'(a,b) = \overline{p^u\left(\frac{\partial \left[ \iota^u(b^{-1}),\iota^v(a)\right]}{\partial z^u}\right)}   = \overline{b^{-1}} = b = \eta(a,b),
$$
as expected from Theorem \ref{th:3d-eta}.
\end{example}

\subsection{A three-dimensional formula for Turaev's self-intersection map} \label{subsec:mu-dim3}

We use the same notations as in Section~\ref{subsec:lambda-dim3}.
 The point $\ast$ endowed with  the unit vector tangent  to $\overline{\partial \Sigma}$ is denoted by $\vec{\ast}$.
Set $\overrightarrow{\pi}:= \pi_1(U\Sigma,\vec{\ast})$ and denote the natural projection $\overrightarrow{\pi} \to \pi$ by $\vec{x} \mapsto x$.
Consider the group homomorphism
$$
c:\overrightarrow{\pi} \longrightarrow PB_2(\Sigma)
$$
that is defined as follows.
Let $\vec \gamma_{\ast v}$ be  a unit  vector field along the arc $\gamma_{\ast v}$ which is equal to $\vec\ast$ at {the} point $\ast$,
is nowhere tangent to $\gamma_{\ast v}$ and gives a unit  vector $\vec{v}$ at  point $v$ ``ending at''~$u$: 
this path in $U\Sigma$ induces an isomorphism between $\overrightarrow{\pi}$ and $ \pi_1(U\Sigma,\vec{v})$.
The latter group can be interpreted as the framed $1$-strand braid group based   at $\vec{v}$ so that, by the ``doubling'' operation, 
we have a group homomorphism  $ \pi_1(U\Sigma,\vec{v}) \to PB_2(\Sigma)$. Then $c$ is  the composition of this homomorphism with the previous isomorphism $\overrightarrow{\pi} \simeq \pi_1(U\Sigma,\vec{v})$.

We define a linear map $\vec{\mu}': \K[\overrightarrow{\pi}] \to \K[\pi]$ by setting, for any $\vec a \in \overrightarrow{\pi}$,
$$
\vec{\mu}'(\vec a) := p^v\left(\frac{\partial\, \iota^u(a^{-1})\, c(\vec{a})\, \iota^v(a^{-1})}{\partial z^v}\right).
$$ 

\begin{lemma} 
For any $\vec a \in \overrightarrow{\pi}$, we have
$$
\vec{\mu}'(\vec a) = p^v\left(\frac{\partial\, \iota^u(a^{-1})\, c(\vec{a})}{\partial z^v}\right) = 
\overline{p^u\left(\frac{\partial\, \iota^u(a^{-1})\, c(\vec{a})\, \iota^v(a^{-1})}{\partial z^u}\right)}
= \overline{a^{-1}  p^u\left(\frac{\partial\,  c(\vec{a})\, \iota^v(a^{-1})}{\partial z^u}\right)}.
$$
Moreover, $\vec{\mu}'$ is a quasi-derivation ruled by $\eta'$ satisfying  $(\vec{\mu}')^t = - \vec{\mu}' + q_{-1,0}$.
\end{lemma}

\begin{proof}
 We prove the first identity of the first statement: for any $\vec{a} \in \overrightarrow{\pi}$, we have
\begin{eqnarray*}
p^v\left(\frac{\partial\, \iota^u(a^{-1})\, c(\vec{a})\, \iota^v(a^{-1})}{\partial z^v}\right) &=& 
p^v\left(\frac{\partial\, \iota^u(a^{-1})\, c(\vec{a})}{\partial z^v} + \iota^u(a^{-1})\, c(\vec{a}) \frac{\partial \iota^v(a^{-1})}{\partial z^v}\right) \\
&=& p^v\left(\frac{\partial\, \iota^u(a^{-1})\, c(\vec{a})}{\partial z^v}\right).
\end{eqnarray*}
The third identity of the first statement is proved similarly, and the second identity follows  from Lemma \ref{lem:U_V}.

We now show that $\vec{\mu}' \in\QDer(\eta')$ using the first statement. For any $\vec{a},\vec{b} \in \overrightarrow{\pi}$, we have
\begin{eqnarray*}
\vec{\mu}'\big( \vec{a} \vec{b}\big) &=& p^v\left(\frac{\partial\, \iota^u( b^{-1}) \iota^u(a^{-1})\, c(\vec{a}) c( \vec{b}  )}{\partial z^v}\right) \\
&=& p^v\left(\frac{\partial\, \big(\iota^u( b^{-1}) \iota^u(a^{-1}) c(\vec{a})  \iota^u( b)\big)\, \big(\iota^u( b^{-1}) c( \vec{b}  )\big)}{\partial z^v}\right) \\
&=& p^v\left(\frac{\partial\, \iota^u( b^{-1}) \iota^u(a^{-1}) c(\vec{a})  \iota^u( b) }{\partial z^v}\right)  +
 p^v\left( \iota^u( b^{-1}) \iota^u(a^{-1}) c(\vec{a})  \iota^u( b)\,\frac{\partial\,  \iota^u( b^{-1}) c( \vec{b}  )}{\partial z^v}\right) \\
&=& p^v\left(\frac{\big(\partial\, \iota^u( b^{-1}) \iota^v(a) \iota^u(b) \big)\, \big(\iota^u(b^{-1})  \iota^v(a^{-1}) \iota^u(a^{-1}) c(\vec{a})  \iota^u( b) \big)}{\partial z^v}\right)  +  a \vec{\mu}'(\vec{b}) \\
&=& p^v\left(\frac{\partial\, \iota^u( b^{-1}) \iota^v(a) \iota^u(b) }{\partial z^v}\right)  
+  a p^v\left(\frac{\partial\, \iota^u(b^{-1})  \iota^v(a^{-1}) \iota^u(a^{-1}) c(\vec{a})  \iota^u( b) }{\partial z^v}\right) +  a \vec{\mu}'(\vec{b}) \\
&=& \eta'(a,b)
+  a p^v\left(\frac{\partial\, \iota^u(b^{-1})  \iota^v(a^{-1}) \iota^u(a^{-1}) c(\vec{a})  \iota^u( b) }{\partial z^v}\right) +  a \vec{\mu}'(\vec{b}).
\end{eqnarray*}
To proceed, we observe that $ \iota^u(b^{-1})  \iota^v(a^{-1}) \iota^u(a^{-1}) c(\vec{a})  \iota^u( b)$ belongs to $\pi^u \cap \pi^v$ so that Lemma~\ref{lem:U_V} applies:
\begin{eqnarray*}
 p^v\left(\frac{\partial\, \iota^u(b^{-1})  \iota^v(a^{-1}) \iota^u(a^{-1}) c(\vec{a})  \iota^u( b) }{\partial z^v}\right) &=&
\overline{  p^u\left(\frac{\partial\, \iota^u(b^{-1})  \iota^v(a^{-1}) \iota^u(a^{-1}) c(\vec{a})  \iota^u( b) }{\partial z^u}\right) }\\
  &=&  \overline{ p^u\left(\frac{\partial\,  \iota^v(a^{-1}) \iota^u(a^{-1}) c(\vec{a})  \iota^u( b) }{\partial z^u}\right) }\ b\\
   &=&  \overline{ p^u\left(\frac{\partial\,  \iota^v(a^{-1}) \iota^u(a^{-1}) c(\vec{a})  }{\partial z^u}\right) } \ b \\
      &=&  p^v\left(\frac{\partial\,  \iota^v(a^{-1}) \iota^u(a^{-1}) c(\vec{a})  }{\partial z^v}\right)  b\\
      &=& a^{-1} p^v\left(\frac{\partial\, \iota^u(a^{-1}) c(\vec{a})  }{\partial z^v}\right)  b 
      \ = \ a^{-1} \vec{\mu}'(\vec{a}) b.
\end{eqnarray*}
We deduce that $ \vec{\mu}'\big( \vec{a} \vec{b}\big) = \eta'(a,b) +   \vec{\mu}'(\vec{a})\, b +  a\, \vec{\mu}'(\vec{b})$.

Finally, we show that $(\vec{\mu}')^t = - \vec{\mu}' + q_{-1,0}$.
For any $\vec{a} \in \overrightarrow{\pi}$, we deduce from Lemma \ref{lem:U_V} that
\begin{eqnarray*}
 \overline{\vec{\mu}'(\vec a)} &=& \overline{p^v\left(\frac{\partial\ \iota^u(a^{-1})\, c(\vec{a})\, \iota^v(a^{-1})}{\partial z^v}\right)} \\
 &=&   p^v\left(\frac{\partial\ \sigma \iota^u(a^{-1})\, c(\vec{a})\, \iota^v(a^{-1}) \sigma^{-1}}{\partial z^v}\right) \\
 &=&   p^v\left(\frac{\partial\ \sigma \iota^u(a^{-1})\sigma^{-1}\, \sigma c(\vec{a}) \sigma^{-1}\, \sigma\iota^v(a^{-1}) \sigma^{-1}}{\partial z^v}\right) \\
 &=& p^v\left(\frac{\partial\  (z^v)^{-1 }\iota^v(a^{-1}) z^v\, c(\vec{a})\, \iota^u(a^{-1})}{\partial z^v}\right) \\
&=&   -1 + p^v\left(\frac{\partial\  \iota^v(a^{-1}) z^v\, c(\vec{a})\, \iota^u(a^{-1})}{\partial z^v}\right) \\
  &=& -1 + a^{-1} p^v\left(\frac{\partial\   z^v\, c(\vec{a})\, \iota^u(a^{-1})}{\partial z^v}\right) \\
 &=& -1 + a^{-1} + a^{-1} p^v\left(\frac{\partial\   c(\vec{a})\, \iota^u(a^{-1})}{\partial z^v}\right) \\
 &=& -1 + a^{-1} + a^{-1} p^v\left( - c(\vec{a})\, \iota^u(a^{-1}) \frac{\partial\   \big(c(\vec{a})\, \iota^u(a^{-1})\big)^{-1}}{\partial z^v}\right) \\
 &=& -1 + a^{-1} - p^v\left( \frac{\partial\   \iota^u(a)\, c((\vec{a})^{-1})}{\partial z^v}\right) \ = \ -1 +a^{-1} -\vec{\mu}'((\vec{a})^{-1}).
\end{eqnarray*}
We deduce that $ \vec{\mu}'(\vec a)  = -1 +a -\overline{\vec{\mu}'((\vec{a})^{-1})}= -1 +a - (\vec{\mu}')^t(\vec a)$.
\end{proof}

\begin{theorem} \label{th:3d-mu}
We have $\vec{\mu}' = \vec{\mu}$.
\end{theorem}

\begin{proof}
Let $\vec a\in \overrightarrow{\pi}$.
We choose an immersion  $\alpha:[0,1] \to \Sigma$  with only finitely many transverse double points such that $\dot\alpha(0) = \dot\alpha(1)  = \vec{\ast}$, 
and  the unit tangent vector field of $\alpha$ represents $\vec{a}$. 
We can also assume that, in a  neighbourhood  of $\ast$, the curve $\alpha$ has the following local picture:
$$
\labellist
\scriptsize\hair 2pt
 \pinlabel {$\vec \ast$} [t] at 440 2
 \pinlabel {$u$} [b] at 340 170
 \pinlabel {$v$} [b] at 542 172
 \pinlabel {$\alpha$} [r] at 166 207
 \pinlabel {$\gamma_{\ast u}$} [r] at 357 126
 \pinlabel {$\gamma_{\ast v}$} [l] at 523 124
\endlabellist
\centering
\includegraphics[scale=0.2]{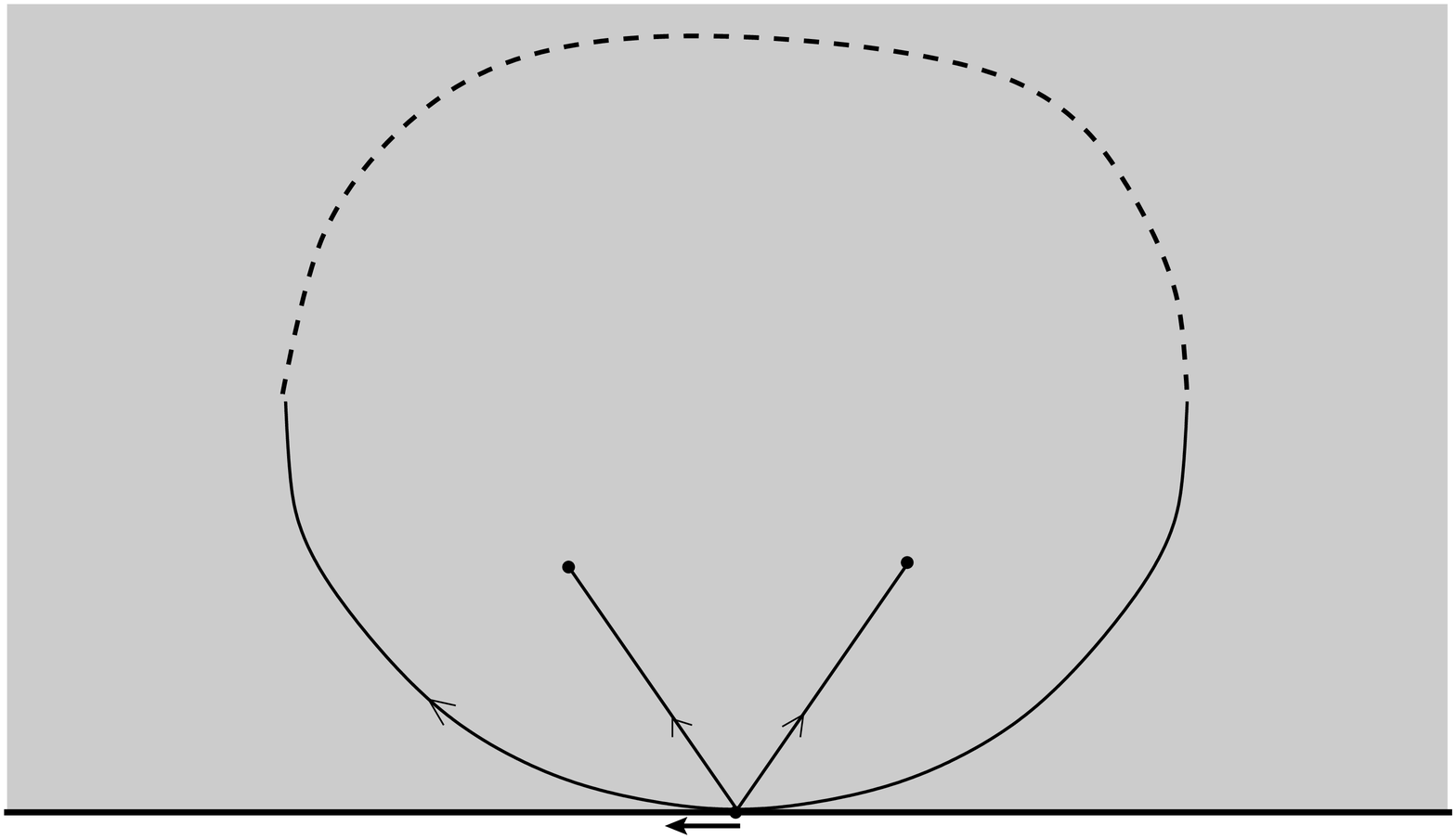}
$$
Then $\vec{\mu}(\vec{a})$  is given by formula \eqref{eq:def-mu}.

In order to compute  $\vec{\mu}'(\vec{a})$, we need  more notations: let $\alpha_v$ be the loop in $U \Sigma$ based at $\vec{v}$
obtained by ``pushing'' away from $\gamma_{\ast u}$ the base point of $\alpha$ along $\vec\gamma_{\ast v}$.
{(Here $\vec v$ and $\vec\gamma_{\ast v}$ are defined as in the definition of the map $c$.)}
A loop  $\alpha_u$ in $U \Sigma$ based at $\vec{u}$ is defined similarly:
$$
\labellist
\scriptsize\hair 2pt
 \pinlabel {$\vec u$} [b] at 339 169
 \pinlabel {$\alpha_u$} [r] at 166 265
 \pinlabel {$\vec\ast$} [t] at 441 8
 \pinlabel {$\vec \gamma_{\ast u}$} [l] at 250 116
\endlabellist
\centering
\includegraphics[scale=0.22]{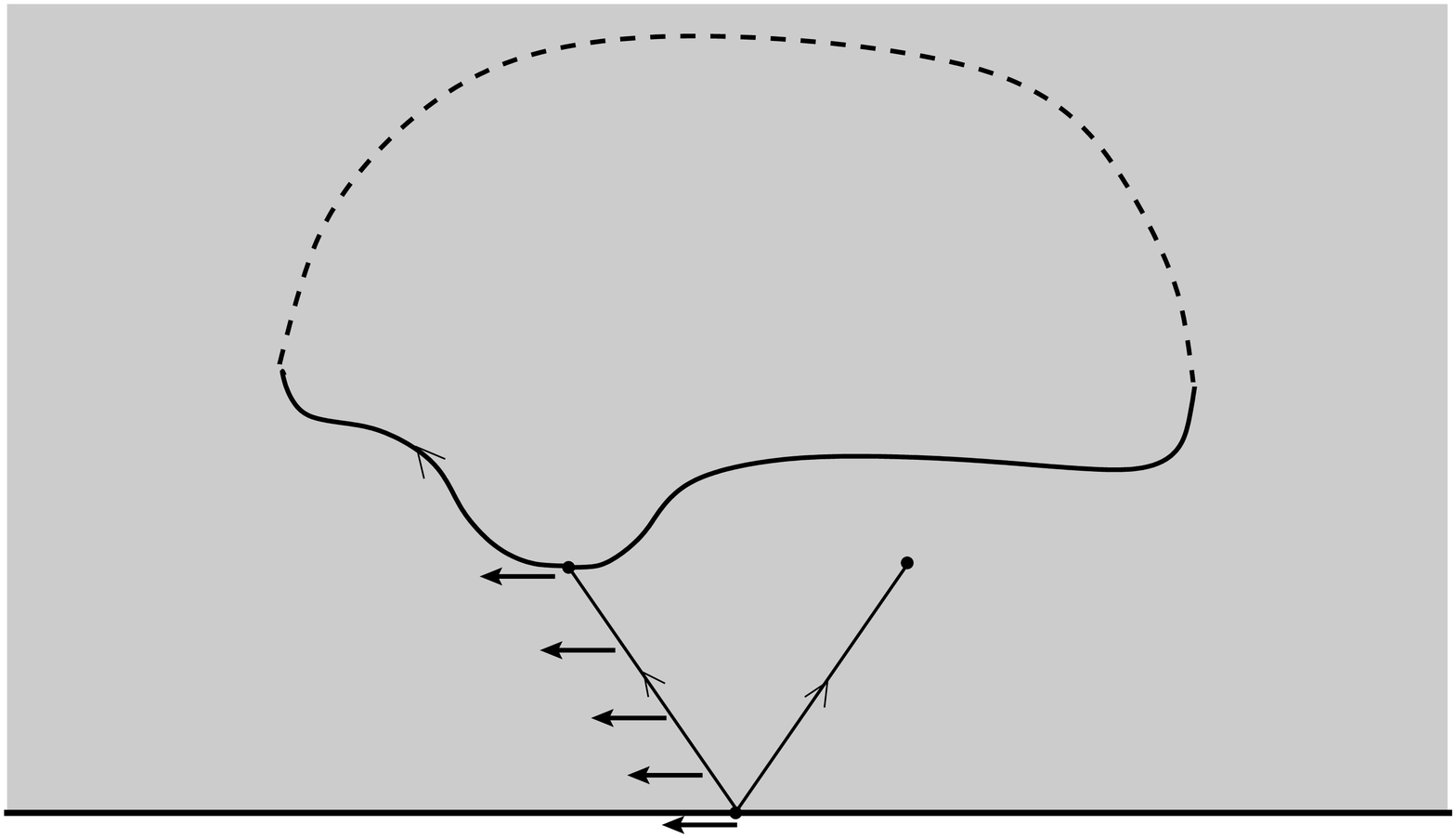} \qquad \quad
\labellist
\scriptsize\hair 2pt
 \pinlabel {$\vec\gamma_{\ast v}$} [l] at 506 92
 \pinlabel {$\vec \ast$} [t] at 439 8
 \pinlabel {$\vec v$} [b] at 542 174
 \pinlabel {$\alpha_v$} [r] at 161 249
\endlabellist
\centering
\includegraphics[scale=0.22]{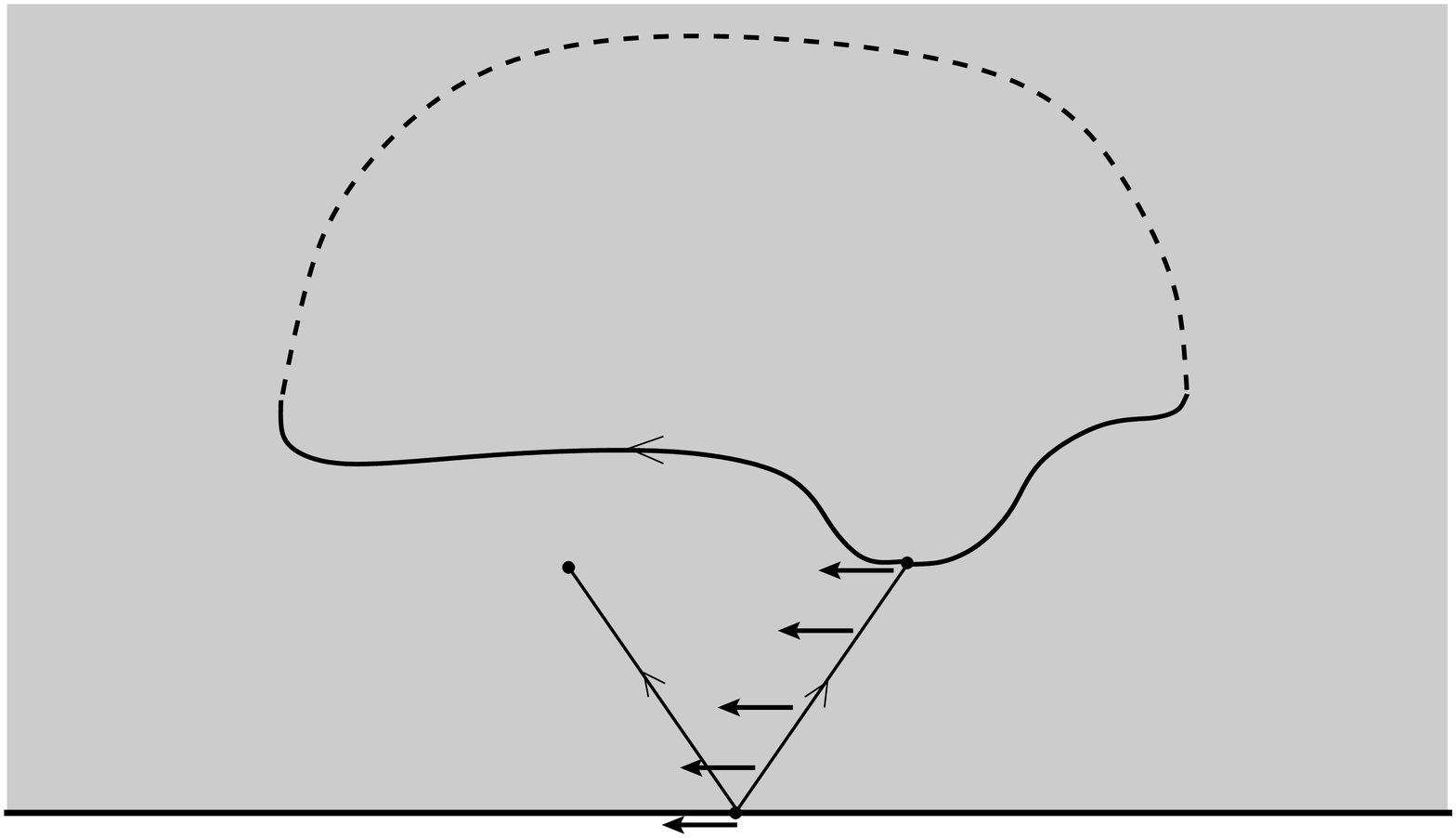}
$$
Observe that  $\iota^u(a)\in \pi^u \subset PB_2(\Sigma)$ is represented by the loop $\alpha_v$ in $\Sigma\setminus\{u\}$ and, similarly,
$\iota^v(a)\in \pi^v \subset  PB_2(\Sigma)$ is represented by the loop $\alpha_u$ in $\Sigma\setminus\{v\}$. 
Thus, as an element of $PB_2(\Sigma)$, $\iota^u(a^{-1})\, c(\vec{a})\, \iota^v(a^{-1})$ is represented by the  $2$-strand pure braid
$$
\begin{array}{cc}
\overline{\alpha}_u  &  \uparrow_v \\
(\alpha_{v})^+  & \alpha_{v} \\
 \uparrow_u  &  \overline{\alpha}_{v}  
\end{array}
$$
where $(\alpha_{v})^+   \alpha_{v}$ is the $2$-strand pure braid whose strand $v$ is $\alpha_v$
and whose strand $u$ is the parallel of $\alpha_v$ obtained by pushing along the unit tangent vector field of $\alpha_v$.
The above pure braid is isotopic to the $2$-strand string-link 
$$
\begin{array}{cc}
\overline{\alpha}_u  &  \uparrow_v \\
(\alpha_{v})^+ & \uparrow_v \\
\delta_{u} &  \alpha_{v} \\
 \uparrow_u  &   \overline{\alpha}_v
\end{array}
$$
whose second ``slice'' $\delta_u\, \alpha_v$  is obtained from $\uparrow_u \alpha_v$ by several surgeries along simple claspers:
specifically, there is one clasper for every double point of $\alpha_v$. The string-link 
$$
\begin{array}{cc}
\delta_{u} &  \alpha_{v} \\
 \uparrow_u  &  \overline{\alpha}_v
\end{array}
$$
defines a loop in $\Sigma \setminus \{v\}$ based at $u$ and, using the  observation  at the beginning of the proof of Lemma~\ref{lem:U_V},
we see that this loop represents the following product of conjugates of $Z^v := \overline{\gamma}_{u\ast }z^v\gamma_{\ast u}$ in the group $ \pi_1(\Sigma \setminus \{v\},u)$:
$$
\prod_{p \in P_{\alpha}} K_p\, (Z^v)^{\varepsilon_p(\alpha)}\, K_p^{-1} \in \pi_1(\Sigma \setminus \{v\},u)
$$
Here the product is over the set $P_{\alpha}$ of double points of $\alpha$ (which is ordered along the positive direction of $\alpha$ starting from $\ast$),
the notation $\varepsilon_p(\alpha)$ is the same as in formula \eqref{eq:def-mu} and 
$K_p$ is a certain element of $\pi_1(\Sigma \setminus \{v\},u)$ 
which is mapped to $\overline{\gamma}_{ u \ast} \alpha_{\ast p}\alpha_{p\ast} \gamma_{\ast u}\in \pi_1(\Sigma,u)$ by the canonical projection.
Observe next that, as loops in $\Sigma\setminus \{v\}$ based at $u$,  $(\alpha_v)^+$ is homotopic to $\alpha_u Z^v$; therefore the loop of $\Sigma \setminus \{v\}$ based at $u$
$$ 
\begin{array}{cc}
\overline{\alpha}_u  &  \uparrow_v \\
(\alpha_{v})^+ & \uparrow_v
\end{array}
$$
represents  $\alpha_u Z^v (\alpha_u)^{-1} \in \pi_1(\Sigma \setminus \{v\},u)$.
We deduce that, as an element of $\pi^v$, the braid $\iota^u(a^{-1})\, c(\vec{a})\, \iota^v(a^{-1})$  is equal to 
$$
\Big(\prod_{p \in P_\alpha} k_p\,  (z^v)^{\varepsilon_p(\alpha)}\,  k_p^{-1}\Big) \  (\iota^v(a)\, z^v\, \iota^v(a^{-1}))
$$ 
where, for all $p\in P_\alpha$,  $k_p\in \pi^v$ has the property to be sent to $\alpha_{\ast p} \alpha_{p\ast}\in \pi$ under $p^v$.
Hence, by application of Lemma~\ref{lem:Fox}, we obtain
$$
\vec{\mu}'(\vec a) =
p^v\left(\frac{\partial\ \iota^u(a^{-1})\, c(\vec{a})\, \iota^v(a^{-1})}{\partial z^v}\right)  
= \sum_{p \in P_\alpha} \varepsilon_p(\alpha)\, (\alpha_{\ast p} \alpha_{p \ast})  + a = \vec{\mu}(\vec{a}).
$$

\up
\end{proof}

\begin{example}
Consider the homotopy class $\vec a\in \overrightarrow{\pi} $ of the small framed loop  based at $\vec{\ast}$
\begin{equation} \label{eq:digamma}
\begin{array}{c}
\labellist
\scriptsize\hair 2pt
 \pinlabel {$\Sigma$} [t] at 40 396
 \pinlabel {$\circlearrowleft$}  at 59 70
 \pinlabel {$\vec{\ast}$}  at 348 8
\endlabellist
\centering
\includegraphics[scale=0.2]{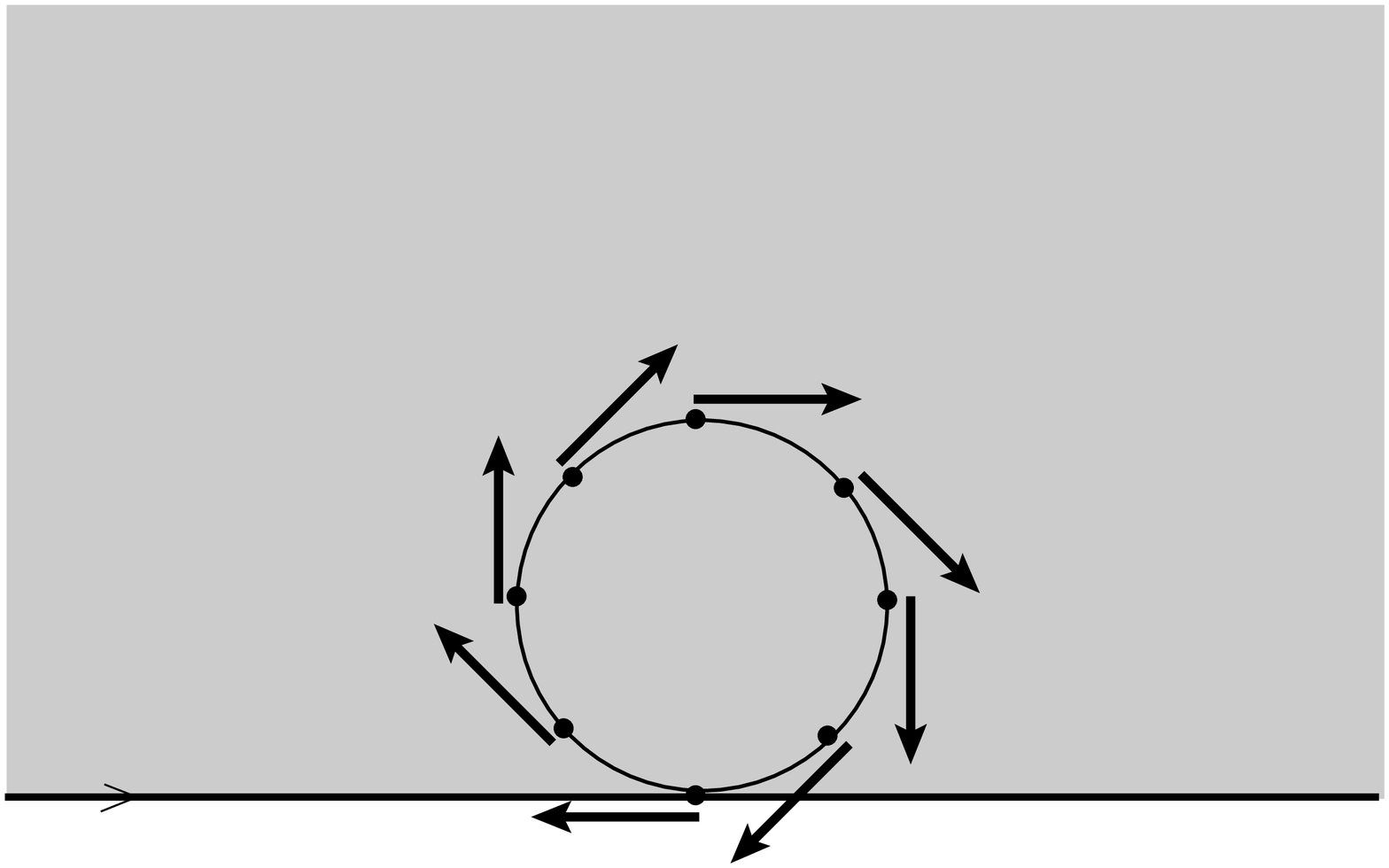}
\end{array}
\end{equation}
(which generates the fundamental group of the fiber of $U\Sigma \to \Sigma$).
Then $c(\vec{a})=  {z^v}$ and $a=1$, so that we have 
$$
\vec\mu'(\vec a) = {p^v\left(\frac{\partial\, \iota^u(a^{-1})\, c(\vec{a})\, \iota^v(a^{-1})}{\partial z^v}\right)
= p^v\left(\frac{\partial z^v }{\partial z^v}\right)} = 1 =  \vec{\mu}(\vec{a}),
$$
as expected from Theorem \ref{th:3d-mu}.
\end{example}

\subsection{Remarks}

1. The facts that $\vec{\mu}' \in\QDer(\eta')$ and  $(\vec{\mu}')^t  + \vec{\mu}' = q_{-1,0}$
easily imply that $\eta' + (\eta')^t= \rho_{-1}$.
Thus we have recovered  by our {three-dimensional} methods 
all the fundamental properties of $\eta=\eta'$ and $\vec{\mu}=\vec{\mu}'$ that have been stated in Section \ref{sec:Turaev_operations}.

2. It is proved in \cite{MT_dim_2} that the skew-symmetrized version $\eta^s$ of $\eta$ 
induces a ``quasi-Poisson double bracket'' in the sense of Van den Bergh \cite{VdB}.
In particular, $\eta^s$ satisfies a kind of ``non-commutative version'' of the Jacobi identity:
it would be interesting to reprove this identity using the {three-dimensional} methods of this section.

\section{Special expansions}  \label{sec:special}

In this section, the ground ring is a commutative field $\K$ of characteristic zero
and $\Sigma$ is a disk with finitely-many punctures numbered from $1$ to $p$.
We formulate the notion of ``special expansion'' which is implicit in \cite{AT,AET},
and we explain its relevance for formal descriptions of the homotopy intersection pairing.

\subsection{Unframed special expansions} \label{subsec:unframed_special}

Set $\pi:= \pi_1(\Sigma,\ast)$ where $\ast \in \partial \Sigma$,
 let $\nu \in \pi$ be the homotopy class of the oriented curve ${\partial \Sigma}$
and, for any $i\in\{1,\dots,p\}$, let $\check \zeta_i$ be the conjugacy class in $\pi$
that is defined by a small counter-clockwise loop around the $i$-th puncture.
Let $H:= H_1(\Sigma;\K)$, let $z_i\in H$ denote the homology class of $\check \zeta_i$ and set $z:= z_1+ \cdots + z_p$.
We denote by  $T(H)$ the  tensor algebra over $H$ and by $T(\!(H)\!)$ the degree-completion of $T(H)$.
The usual Hopf algebra structure of $T(H)$ extends to a complete Hopf algebra structure on $T(\!(H)\!)$ 
whose coproduct $\hat\Delta$, counit $\hat\varepsilon$ and antipode $\hat S$  are defined by 
$$
\forall k\in H, \quad \hat \Delta(k) := k \hat \otimes 1 + 1 \hat \otimes k, \quad \hat \varepsilon(k) :=0, \quad \hat S(k):= -k.
$$

A \emph{special expansion} of $\pi$ is a  map $\theta: \pi \to T(\!(H)\!)$ with the following properties:
\begin{enumerate}
\item[(i)] for all $x,y\in \pi$, $\theta(xy) = \theta(x)\, \theta(y)$;
\item[(ii)] for each $i\in \{1,\dots,p\}$, and given a representative $\zeta_i\in \pi$ of $\check \zeta_i$, 
there exists a primitive element $u_i \in T(\!(H)\!)$ such that $\theta( \zeta_i) = \exp(u_i) \exp(z_i) \exp(-u_i)$;
\item[(iii)]   $\theta(\nu)=\exp(z)$.
\end{enumerate}
Let $ \widehat{\K[\pi]}$ denote the completion of $\K[\pi]$ with respect to the $I$-adic filtration \eqref{eq:I-adic}.
Conditions (i) and (ii)  imply that $\theta$ induces an isomorphism of complete Hopf algebras $\hat \theta: \widehat{\K[\pi]} \to T(\!(H)\!)$
which, at the level of graded Hopf algebras, gives the canonical isomorphism \eqref{eq:canonical_alg_iso}.

Using the fact that $\pi$ is freely generated by some representatives $\zeta_1,\dots,\zeta_p$ of $\check\zeta_1,\dots,\check\zeta_p$,
and by proceeding by successive finite-degree approximations, it is not difficult to construct instances of special expansions.
This kind of expansions of the free group $\pi$ appears implicitly in \cite{AT} and more explicitly in \cite{AET},
in relation with ``special automorphisms'' and ``special derivations'' of  free Lie algebras.

\subsection{Framed special expansions} \label{subsec:framed_special}

Set $\overrightarrow{\pi}:= \pi_1(U \Sigma, \vec\ast\, )$  where $\vec \ast$ is the unit vector tangent to $\overline{\partial \Sigma}$ at $\ast$.
The bundle projection $U\Sigma \to \Sigma$ has a section defined by a unit vector field on the disk closure of $\Sigma$
(which is unique, up to homotopy). This section induces a canonical group homomorphism $s:\pi \to \overrightarrow{\pi}$ at the level of fundamental groups:
 let $w: \overrightarrow{\pi} \to \Z$ be the unique group homomorphism such that
$$
\forall \vec{x} \in \overrightarrow{\pi}, \quad \vec{x} = \digamma^{w(\vec x)} s(x).
$$
where  $x \in \pi$ denotes the projection of $\vec x$ and $\digamma^{-1} \in \overrightarrow{\pi}$ denotes 
the generator \eqref{eq:digamma} of the fundamental group of the fiber of $U \Sigma \to \Sigma$.

Consider  the algebra $\K[[C]]$  of formal power series in the indeterminate $C$:
we declare that $\deg(C):=1$ and we consider the degree-filtration on $\K[[C]]$;
then  $\K[[C]]$ has the  structure of a complete Hopf algebra with coproduct $\hat\Delta$, counit $\hat\varepsilon$ and antipode $\hat S$ defined by 
$$
\hat \Delta(C) := C \hat \otimes 1 + 1 \hat \otimes C, \quad \hat \varepsilon(C) :=0, \quad \hat S(C):= -C
$$
respectively. The complete tensor product $T(\!(H)\!)\hat \otimes \K[[C]]$ of $T(\!(H)\!)$ and $\K[[C]]$  has a natural structure of complete Hopf algebra.

A \emph{special expansion} of $\overrightarrow{\pi}$ is a  map $\vec \theta: \overrightarrow{\pi} \to T(\!(H)\!)\hat \otimes \K[[C]]$ 
with the following properties:
\begin{enumerate}
\item[(i)] for all $\vec x,\vec y\in \overrightarrow{\pi}$, $\vec \theta(\vec x\, \vec y) = \vec \theta(\vec x)\, \vec \theta(\vec y)$;
\item[(ii)] for each $i\in \{1,\dots,p\}$, and given an element $\vec\zeta_i\in {\overrightarrow{\pi}}$ 
which maps to $\check \zeta_i$ under the canonical maps $ \overrightarrow{\pi} \to \pi \to \check \pi$,
there exists a primitive element $u_i \in T(\!(H)\!)$ 
such that $\vec\theta( \vec\zeta_i) = \big(\exp(u_i) \exp(z_i) \exp(-u_i)\big) \hat\otimes \exp\big(\frac{w(\vec\zeta_i)}{2} C\big)$;
\item[(ii')] for all $k\in \Z$, $\vec \theta (\digamma^k) = 1 \hat \otimes \exp\big(\frac{k}{2}C\big)$;
\item[(iii)]  $\vec\theta(\vec{\overline \nu}\, )=\exp(-z)\hat \otimes 1$. 
\end{enumerate}
In condition (iii),  $\vec{\overline \nu} \in \overrightarrow{\pi} $ is the homotopy class of the unit vector field along the curve $\overline{\partial \Sigma}$
which has index $1$ with respect to the tangent vector field  of this oriented curve. {(In other words, we have $\vec{\overline \nu}= s(\nu^{-1})$.)}
Let $ \widehat{\K[ \overrightarrow{\pi} ]}$ denote the completion of $\K[\overrightarrow{\pi}]$ with respect to the $I$-adic filtration.
The conditions (i), (ii) and (ii') imply that $\vec \theta$ induces an isomorphism of complete Hopf algebras 
$\widehat{\vec\theta}: {\widehat{\K [ {\overrightarrow \pi} ] } \to T(\!(H)\!)\hat \otimes \K[[C]]}$
which, at the level of graded Hopf algebras, gives the canonical isomorphism that is provided by the following lemma.

\begin{lemma} \label{lem:iga}
Let $\K[C]$ be the polynomial algebra with indeterminate $C$ of degree  $1$.
There is   a unique  isomorphism of graded  Hopf algebras between $\Gr\, {\K[\overrightarrow{\pi}]}$ and $T(H) \otimes \K[C]$, 
defined in degree $1$ by $(\vec{x}-1) \mapsto [x] \otimes 1 + 1\otimes  \frac{w(\vec{x})}{2} C$ for any $\vec x\in {\overrightarrow{\pi}}$ projecting to $x \in \pi$. 
\end{lemma}

\begin{proof}
The unicity is obvious since the graded algebra $\Gr\, {\K[\overrightarrow{\pi}]}$ is generated by its degree~$1$ part.
To prove the existence, consider the group isomorphism
$\overrightarrow{\pi} \to \pi \times F(\digamma)$ defined by $\vec{x} \mapsto (x,\digamma^{w(\vec{x})})$.
This induces a  graded Hopf algebra isomorphism
$$
\Gr\, \K[\overrightarrow{\pi}] \stackrel{\simeq}{\longrightarrow} \Gr(\K[\pi \times F(\digamma)]) \simeq \Gr( \K[\pi] \otimes \K[F(\digamma)] ) \simeq \Gr( \K[\pi]) \otimes \Gr(\K[F(\digamma)] ) .
$$
Since $\Gr\, \K[\pi] \simeq T(H)$ by \eqref{eq:canonical_alg_iso} and $\Gr\, \K[F(\digamma)] \simeq \K[C]$ by the map $(\digamma-1)\mapsto C/2$,
we obtain an isomorphism between $\Gr\, \K[\overrightarrow{\pi}]$ and $T(H) \otimes \K[C]$. 
The values of this isomorphism on the degree $1$ generators $(\vec x -1)$  are easily computed.
\end{proof}

{In fact,} unframed special expansions are in one-to-one correspondence with framed special expansions. 
Indeed, for any special expansion $\theta$  of $\pi$,
it is easily verified that the map $\overrightarrow{(\theta)}_s: {\overrightarrow{\pi}\to T(\!(H)\!)\hat \otimes \K[[C]]}$ defined by
$$
\forall \vec x \in \overrightarrow{\pi}, \quad \overrightarrow{(\theta)}_s(\vec x) := \theta(x) \hat \otimes \exp\left( \frac{w(\vec x)}{2}\, C\right)
$$
is a special expansion of $\overrightarrow{\pi}$.
Conversely, any special expansion $\vec \theta$ of $\overrightarrow{\pi}$ 
induces a special expansion $p_*\big(\vec\theta\, \big)$ of $\pi$ as follows.
Consider the complete Hopf algebra map 
$p:= \id \hat \otimes \hat \varepsilon :T(\!(H)\!) \hat\otimes \K[[C]]  \to T(\!(H)\!)$;
we have 
$$
\forall \vec x \in \overrightarrow{\pi}, \ \forall k \in \Z, \quad  p\big(\vec \theta (\vec x \digamma^k) \big)
= p\left( \vec\theta(\vec x)\, \vec{\theta}(\digamma^k) \right)   = p {\big( \vec\theta(\vec x) \big)} ;
$$
therefore the composition $p\circ \vec\theta$ induces a map $p_*\big(\vec \theta\, \big): \pi \to T(\!(H)\!)$
which is clearly a special expansion. 
It is also clear that the above two constructions are inverse one to the other:
\begin{equation} \label{eq:special's}
\begin{array}{ccc}
\{\hbox{\small special expansions of $\pi$} \} &  	\longleftrightarrow& \{\hbox{\small special expansions of $\overrightarrow{\pi}$} \}  \\
\theta  & \longmapsto & \overrightarrow{(\theta)}_s \\
  p_*\big(\vec \theta\, \big) & \! \longleftarrow\joinrel\mapstochar &\vec{\theta}
\end{array}
\end{equation}

\subsection{Formal description of Turaev's intersection pairing} \label{subsec:hip_special}

{Let $\odot$ be the bilinear operation in $H$ defined by $z_i \odot z_j :=\delta_{ij}z_i$ for all $i,j \in \{1,\dots,p\}$.
It extends uniquely to a  bilinear operation $\odot$ in $T(H)$ defined by 
$(x\odot 1) = (1 \odot x ) :=0$ for any $x\in T(H)$ and  by
$$
\left(h_1  \cdots   h_m \odot  k_1  \cdots   k_n\right)
:=  h_1  \cdots   h_{m-1}\,  (h_m \odot  k_1)\,   k_2   \cdots   k_n
$$
for any integers $m,n\geq 1$ and for any $h_1,\dots,h_m,k_1,\dots,k_n \in H$.
This is a $(-1)$-filtered Fox pairing in $T(H)$, which induces a  filtered Fox pairing}
$$
(- {\odot} -): T(\!(H)\!) \times T(\!(H)\!) \longrightarrow T(\!(H)\!)
$$
in the complete Hopf algebra $T(\!(H)\!)$.
Moreover, the formal power series $s(X) \in \Q[ [X] ]$ defined~by 
\begin{eqnarray}
\notag s(X) & := & -\frac{1}{2} + \frac{1}{X} -\frac{1}{2}\coth(X/2)  \\
\label{eq:s(X)} & = &  \frac{1}{X} + \frac{1}{e^{-X}-1}   \\
\notag &=&   -\frac{1}{2}   - \sum_{k\geq 1} \frac{B_{2k}}{(2k)!} X^{2k-1}
\ = \ -\frac{1}{2} -\frac{X}{12}+ \frac{X^3}{720}-\frac{X^5}{30240} + \cdots
\end{eqnarray}
evaluated at $-z=-z_1- \cdots - z_p \in H$ induces an inner Fox pairing $\rho_{s(-z)}$ in $T(\!(H)\!)$. 

\begin{theorem} \label{th:special}
For any special expansion $\theta$ of $\pi$, we have the commutative diagram 
\begin{equation}   \label{eq:tensorial_eta}
\xymatrix{
\widehat{\K[\pi]} \times \widehat{\K[\pi]}  \ar[d]_-{\hat \theta \times \hat \theta}^-\simeq \ar[rr]^-{\hat \eta} && \widehat{\K[\pi]} \ar[d]^-{\hat \theta}_-\simeq \\
T(\!(H)\!) \times T(\!(H)\!)  \ar[rr]_-{(-\, {\odot}\, - )+ \rho_{s(-z)}} && T(\!(H)\!).
}
\end{equation}
\end{theorem}

\noindent
Theorem \ref{th:special}  was obtained in the unpublished draft \cite{MT_draft}.
The proof, which is postponed to Appendix~\ref{sec:hip_special},
is  based on the formal description of the pairing $\eta$
in the case of a compact connected oriented surface with one boundary component  \cite{MT_twists}. 
We also give in Section~\ref{sec:proofs} a very different proof of Theorem~\ref{th:special}, 
but assuming that $\theta$ is  the special expansion arising from the Kontsevich integral.

\subsection{Remarks}
1. Recall from Section \ref{subsec:eta} that the bracket induced by the Fox pairing $\eta$ 
is the  Goldman bracket $\langle - , - \rangle_{\operatorname{G}}$.
Then it can be deduced from Theorem \ref{th:special} that  the following diagram is commutative for any special expansion $\theta$:
\begin{equation} \label{eq:formal_Goldman}
\xymatrix{
 \widehat{\K\check \pi} \times \widehat{\K\check \pi} \ar[d]_-{\hat \theta \times \hat \theta}^-\simeq \ar[rr]^-{\langle - , - \rangle_{\operatorname{G}}^{\widehat{} }}
  && \widehat{\K\check \pi}  \ar[d]^-{\hat \theta}_-\simeq \\
\check T(\!(H)\!) \times \check T(\!(H)\!)  \ar[rr]_-{\langle - , - \rangle_{\operatorname{N}}^{\widehat{} }} && \check T(\!(H)\!)
}
\end{equation}
Here $\check T(\!(H)\!)$ is the degree-completion of  $\check T(H) = T(H)/[T(H),T(H)]$, 
and ${\langle - , - \rangle_{\operatorname{N}}^{\widehat{} }}$ is the completion of the bilinear map
$\langle - , - \rangle_{\operatorname{N}}: \check T(H) \times \check T(H) \to \check T(H)$ defined by
\begin{eqnarray*}
\big\langle h_1  \cdots  h_m , k_1  \cdots  k_n \big\rangle_{\operatorname{N}} 
&=&  - \sum_{i=1}^m \sum_{j=1}^n   k_{j+1} \cdots k_n  k_1 \cdots k_{j-1}\, (k_j {\odot} h_i)\,  h_{i+1} \cdots h_m  h_1 \cdots h_{i-1}\\
&& +  \sum_{i=1}^m \sum_{j=1}^n  h_{i+1} \cdots h_m  h_1 \cdots h_{i-1}\,  (h_i {\odot} k_j)\,  k_{j+1} \cdots k_n k_1 \cdots k_{j-1}
\end{eqnarray*}
for all integers $m,n\geq 1$ and all $h_1,\dots, h_m, k_1,\dots,k_n \in H$.
It turns out that  $\langle - , - \rangle_{\operatorname{N}}$ is  the necklace Lie bracket \cite{BLB,Gi}
associated to  a star-shaped quiver consisting of one ``central'' vertex  connected by $p$ edges to $p$ ``peripheral'' vertices.

2. The  formal description \eqref{eq:formal_Goldman}  of  the Goldman bracket of a punctured disk
has also been  obtained by Kawazumi and Kuno, {as} announced in the survey paper \cite{KK_survey}. 
They consider there the general case of a compact connected oriented surface with several boundary components.

\section{The special expansion defined by the Kontsevich integral}  \label{sec:special_Z}

In this section, the ground ring is a commutative field $\K$ of characteristic zero
and $\Sigma$ is a disk with finitely-many punctures numbered from $1$ to $p$.
Following mainly \cite{HgM} and \cite{AET}, we explain  how to construct a special expansion $\theta_Z$ from the Kontsevich integral $Z$.

\subsection{Some spaces of Jacobi diagrams}

The reader is refered to the paper \cite{BN1} and to the textbook \cite{Oh} for an introduction to Jacobi diagrams.
For any integer $n\geq 1$,  let $\A(\uparrow_1 \cdots \uparrow_n)$ be the space spanned by Jacobi diagrams on the $1$-manifold  
consisting of $n$ copies of the oriented interval $\uparrow$, modulo the AS, IHX and STU relations.
Recall that $\A(\uparrow_1 \cdots \uparrow_n)$ is a graded Hopf algebra, 
whose degree is  half the total number of vertices; in the sequel, we use the same notation for the degree-completion of $\A(\uparrow_1 \cdots \uparrow_n)$.
For any (possibly equal) $i,j \in \{1,\dots,n\}$,  let
$$
t_{ij} := \begin{array}{c}
\labellist
\small\hair 2pt
 \pinlabel {$1$} [t] at 20 4
 \pinlabel {$i$} [t] at 236 4
 \pinlabel {$j$} [t] at 523 4
 \pinlabel {$n$} [t] at 739 4
  \pinlabel {$\cdots$}  at 89 47
 \pinlabel {$\cdots$}  at 378 50
 \pinlabel {$\cdots$}  at 669 51
\endlabellist
\centering
\includegraphics[scale=0.15]{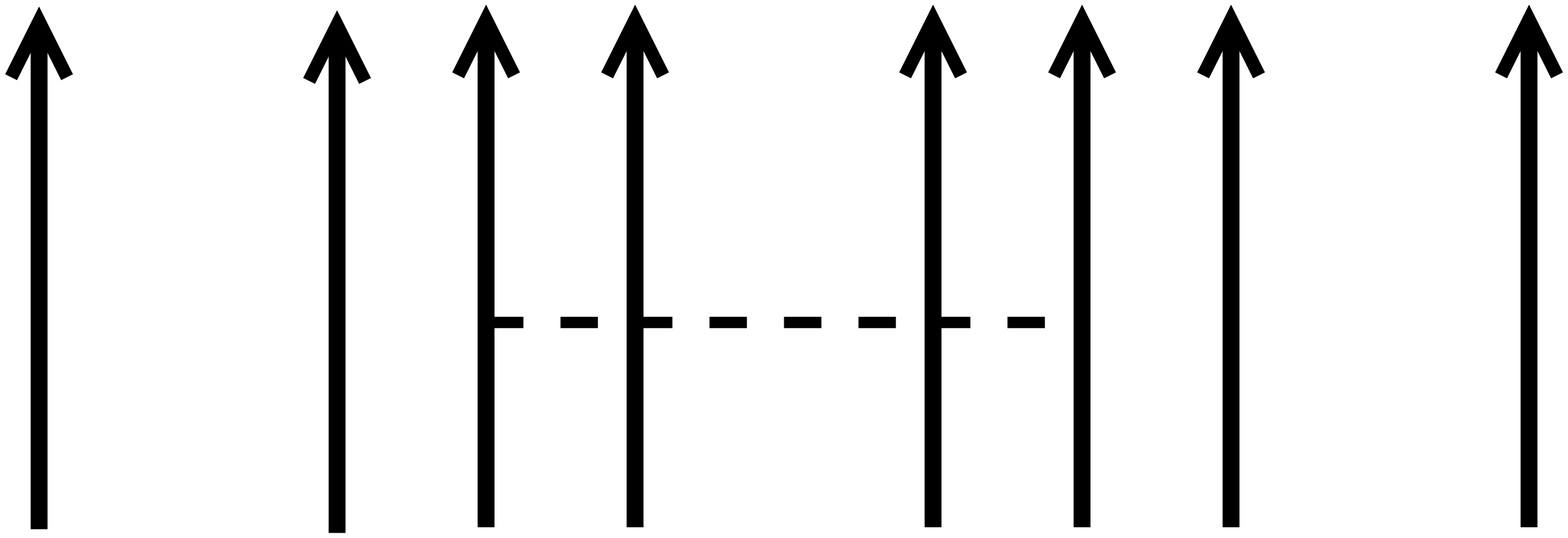}
\end{array} 
$$
\vspace{0.2cm}

\noindent
be the Jacobi diagram with only one edge (a \emph{chord}) connecting $\uparrow_i$ to $\uparrow_j$.
If $n=p+1$, we give a special role to the rightmost copy of $\uparrow$ which we label with $\ast$ (instead of $p+1$):
let $\A_{p,\ast}$ be  the closed subspace of $\A(\uparrow_1 \cdots \uparrow_p\, \uparrow_\ast)$ spanned by Jacobi diagrams 
whose all connected components touch $\uparrow_\ast$. 
Recall {from Section \ref{subsec:unframed_special}}  that $H= H_1(\Sigma; \K)$ is generated by the homology classes $z_1,\dots,z_p$ of 
some small counter-clockwise loops $\check \zeta_1,\dots, \check \zeta_p$ around the punctures.

\begin{lemma} \label{lem:A_p*}
Consider the homomorphism of complete  algebras
\begin{equation} \label{eq:a_simple_map}
T(\!(H)\!) \hat\otimes \K[[C]]   {\longrightarrow} \A_{p,\ast}
\end{equation}
that maps $1 \hat\otimes C$ to the diagram $t_{\ast \ast}$ and $z_i\hat\otimes 1$ to the  diagram $t_{i\ast}$ for all $i\in \{1,\dots,p\}$.
Then the homomorphism \eqref{eq:a_simple_map} is injective and preserves the Hopf algebra structures.
\end{lemma}

\begin{proof}
The complete algebra $T(\!(H)\!) \hat\otimes \K[[C]] $ is generated by $1 \hat\otimes C$ and $z_1\hat\otimes 1, \dots, z_p\hat\otimes 1$.
All these elements are primitive and are mapped by \eqref {eq:a_simple_map} to primitive elements of $\A(\uparrow_1 \cdots \uparrow_p\, \uparrow_\ast)$.
Therefore the algebra homomorphism \eqref{eq:a_simple_map} preserves the coproducts, counits and antipodes.

We now prove the injectivity of \eqref{eq:a_simple_map}.
Following Bar-Natan \cite{BN2}, we consider the ``homotopic reduction'' $\A^h_{p,\ast}$ of  $\A_{p,\ast}$:
specifically, we set $\A^h_{p,\ast}:= \A_{p,\ast}/ \mathcal{H}_{p,\ast}$ where $\mathcal{H}_{p,\ast}$ is the closed subspace
generated by Jacobi diagrams (without connected components unattached to $\uparrow_\ast$)
showing at least one connected component with at least one trivalent vertex and  two univalent vertices on $\uparrow_\ast$.
Note that the space $\A^{h}_{p,\ast}$ inherits an algebra structure from $\A_{p,\ast}$.
Let
$$
\rho: T(\!(H)\!) \hat\otimes \K[[C]]  \longrightarrow \A^h_{p,\ast}
$$
be the algebra map obtained by composing \eqref{eq:a_simple_map} with the canonical projection $\A_{p,\ast} \to \A^h_{p,\ast}$.
To prove the lemma, it suffices to prove that $\rho$ is an isomorphism.

Let $\A(\{1,\dots,p\},\uparrow_\ast)$ be the space spanned by Jacobi diagrams on $\uparrow_\ast$ with some univalent vertices colored by $\{1,\dots,p\}$, modulo the AS, IHX and STU relations;
let $\A'_{p,\ast}$ be the closed subspace of $\A(\{1,\dots,p\},\uparrow_\ast)$ spanned by diagrams without connected components unattached to $\uparrow_\ast$,
and let $\mathcal{H}'_{p,\ast}$ be the closed subspace spanned by diagrams of this kind 
and showing  at least one connected component with at least one trivalent vertex and two univalent vertices on~$\uparrow_\ast$.
The space $\A'_{p,\ast}$ is a complete algebra, whose multiplication is simply  defined by concatenation along the oriented  interval $\uparrow_\ast$,
and $\mathcal{H}'_{p,\ast}$  is an ideal of $\A'_{p,\ast}$: therefore the quotient $\A^{'h}_{p,\ast}:= \A'_{p,\ast}/ \mathcal{H}'_{p,\ast}$ is an algebra.
Let 
$$
\rho': T(\!(H)\!) \hat\otimes \K[[C]]  \longrightarrow \A^{'h}_{p,\ast}
$$
be the complete  algebra homomorphism that maps $1 \hat\otimes C$ to the isolated chord $t_{\ast \ast}$ and $z_i\hat\otimes 1$ to the diagram
$$
\labellist
\small\hair 2pt
 \pinlabel {$\ast$} [l] at 141 10
 \pinlabel {$i$} [r] at 5 123
\endlabellist
\centering
\includegraphics[scale=0.15]{tensor_strut}
$$
for all $i \in \{1,\dots,p\}$.
The PBW-type isomorphism (see \cite{BN1})
$$
\chi_p: \A(\{1,\dots,p\},\uparrow_\ast) \stackrel{\simeq}{\longrightarrow} \A(\uparrow_1 \cdots \uparrow_p\uparrow_\ast)
$$
maps  isomorphically $\A'_{p,\ast}$ onto $\A_{p,\ast}$, and $\mathcal{H}'_{p,\ast}$ onto $\mathcal{H}_{p,\ast}$.
It easily follows from the STU relation along the intervals $\uparrow_1,\dots,\uparrow_p$ that the following diagram is commutative:
$$
\xymatrix{
\A^{'h}_{p,\ast} \ar[r]^-{\chi_p}_\simeq & \A^h_{p,\ast} \\
{T(\!(H)\!) \hat\otimes \K[[C]]} \ar[u]^-{\rho'} \ar[ru]_-\rho & 
}
$$
Therefore we are reduced to prove that $\rho'$ is an isomorphism.

Let $\A(\{1,\dots,p,\ast\})$ be the space spanned by Jacobi diagrams whose univalent vertices are colored by $\{1,\dots,p,\ast\}$, modulo the AS and IHX relations;
let $\A''_{p,\ast}$ be the closed subspace of $\A(\{1,\dots,p,\ast\})$  spanned by diagrams without connected components uncolored by $\ast$,
and let $\mathcal{H}''_{p,\ast}$ be the closed subspace spanned by diagrams of this kind 
and showing  at least one connected component which is looped or has at least one trivalent vertex and two univalent vertices colored by $\ast$.
According to \cite[Theorem 1]{BN2}, the PBW-type isomorphism
$$
\chi_\ast: \A(\{1,\dots,p,\ast\}) \stackrel{\simeq}{\longrightarrow} \A(\{1,\dots,p\}, \uparrow_\ast)
$$
maps $\mathcal{H}''_{p,\ast}$ onto $\mathcal{H}'_{p,\ast}$, so that it induces an isomorphism from
$\A^{''h}_{p,\ast}:= \A''_{p,\ast}/\mathcal{H}''_{p,\ast}$ to $\A^{'h}_{p,\ast}$. 
Consider now the following diagram:
$$
\xymatrix{
\A^{'h}_{p,\ast} & \ar[l]_-{\chi_\ast}^-\simeq \A^{''h}_{p,\ast}  \\
T(\!(H)\!) \hat\otimes \K[[C]]\ar[u]^-{\rho'} & \ar[l]_-{\chi\hat\otimes \id}^-\simeq S\big(\!\big(\mathfrak{L}(H)\big)\!\big) \hat \otimes \K[[C]] \ar[u]_-\simeq 
}
$$
Here $\mathfrak{L}(H)$ is  the free Lie algebra generated by the space $H$,
$S\big(\!\big(\mathfrak{L}(H)\big)\!\big)$ is the complete symmetric algebra generated by the space $\mathfrak{L}(H)$
and $\chi$ denotes the usual PBW isomorphism.
The right vertical isomorphism in this diagram identifies  powers of $C$ with  disjoint union of chords \
$
\labellist
\scriptsize \hair 2pt
 \pinlabel {$\ast$} [b] at 3 78
 \pinlabel {$\ast$} [b] at 359 77
\endlabellist
\centering
\includegraphics[scale=0.05]{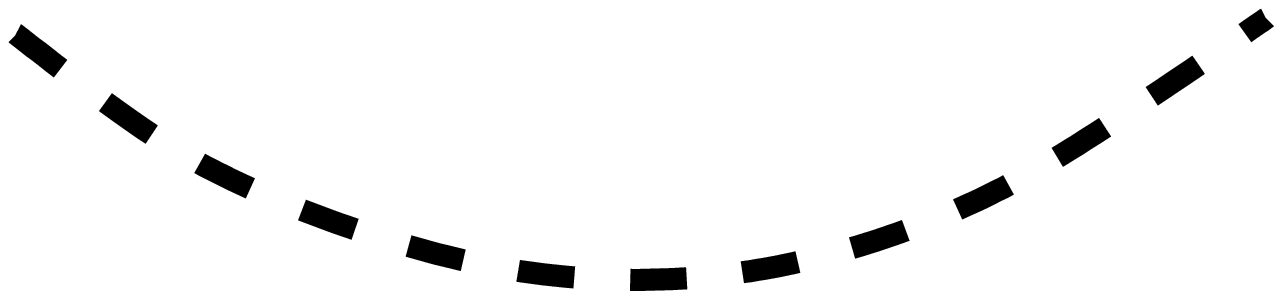},
$
and it identifies Lie words in $z_1,\dots,z_p$ 
with binary trees which are rooted at $\ast$ and whose leaves are colored by $\{1,\dots,p\}$.
It follows from the STU relation along $\uparrow_\ast$ that the above diagram is commutative. We conclude that $\rho'$ is an isomorphism.
\end{proof}

Lemma \ref{lem:A_p*} produces a canonical  isomorphism between the algebra $T(\!(H)\!) \hat\otimes \K[[C]]$ 
and the subalgebra $\horiz{\A}_{p,\ast}$ of ${\A}_{p,\ast}$ generated by $t_{\ast \ast}, t_{1\ast}, \dots, t_{p\ast}$.
Let $FI$ denote the \emph{Framing Independence} relation on  $\A_{p, \ast}$
which sets to zero any Jacobi diagram with an isolated chord on the same interval $\uparrow$.
Then Lemma \ref{lem:A_p*} implies that the homomorphism of complete graded algebras
\begin{equation} \label{eq:another_simple_map}
T(\!(H)\!)  {\longrightarrow}   \A_{p,\ast}/FI
\end{equation}
that maps $z_i$ to $t_{i\ast}$  for all $i \in \{1,\dots,p\}$, is an isomorphism onto $\horiz{\A}_{p,\ast}/FI$.

\subsection{Constructions with the Kontsevich integral}

We now consider the combinatorial version of the Kontsevich integral $Z$.
The reader is refered to the articles \cite{BN4,LM1,LM2,KT}, or to the textbook \cite{Oh} for an introduction to this invariant.
In the next paragraph, we only highlight the main features of $Z$, 
and we fix some conventions on the way how $Z$ is constructed from a Drinfeld associator.

Recall that the Kontsevich integral $Z$ is a topological invariant of framed oriented tangles in the ball $D^2 \times [0,1]$,
whose boundary points are  located at the ``bottom'' $D^2 \times \{0\}$ or at the ``top'' $D^2 \times \{1\}$.
These boundary points should be parenthesized at each extremity: 
tangles with this kind of structure are called \emph{non-associative tangles} in \cite{BN4} and  \emph{$q$-tangles} in \cite{LM1,LM2}.
Parenthesized framed tangles constitute a monoidal category 
whose objects are non-associative words in the two letters ``$+/-$'',
whose composition of morphisms is defined by vertical stacking
$$
  T\,  U := \begin{array}{|c|}\hline U \\ \hline T \\ \hline \end{array}
$$
and whose tensor product of morphisms is defined by horizontal juxtaposition
$$
T \otimes U := \begin{array}{|c|c|}  \hline T & U \\ \hline \end{array}.
$$
The Kontsevich integral $Z$ is a tensor-preserving functor from the category
of parenthesized framed  tangles to the category of Jacobi diagrams on abstract oriented $1$-manifolds, modulo the AS, IHX and STU relations.
Actually, we  will only need the restriction of $Z$ to the category of parenthesized  framed braids
(whose objects are non-associative words in the single letter ``$+$'').
In this case, $Z$ is determined by the following properties:
\begin{itemize}
\item[(i)]  the operation of  doubling one strand  in a parenthesized framed braid 
corresponds under $Z$ to the duplication $\Delta$ of Jacobi diagrams (which should not  be confused with the coproduct map);
\item[(ii)] $Z$ takes the following values on the elementary framed  braids:
$$
Z\bigg( \begin{array}{c} {}_{(++)} \\  \figtotext{20}{20}{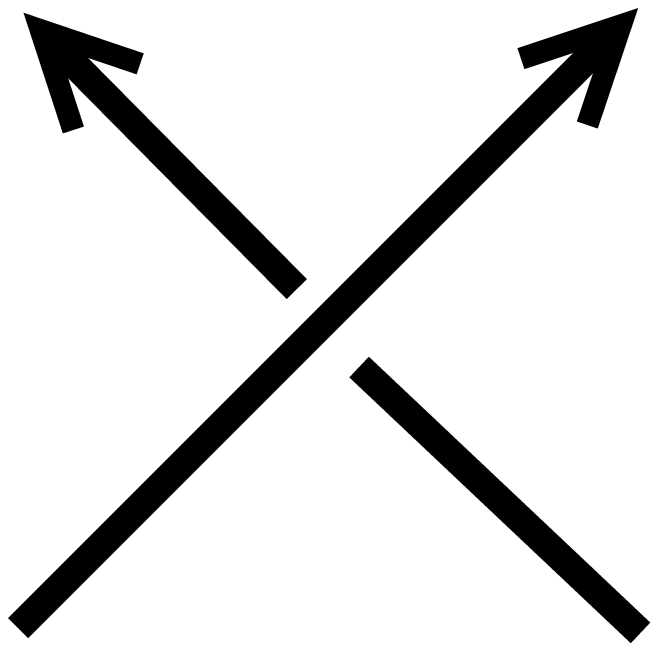} \\ {}^{(++)} \end{array} \bigg) =
\exp\Big(\figtotext{20}{20}{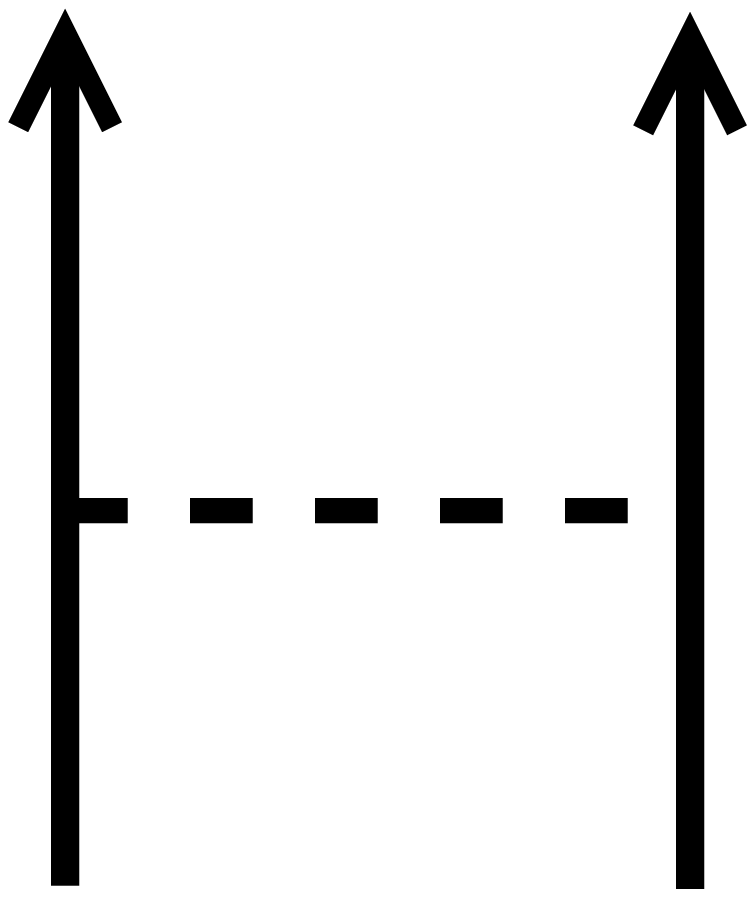}\!\!/2\Big) \figtotext{20}{20}{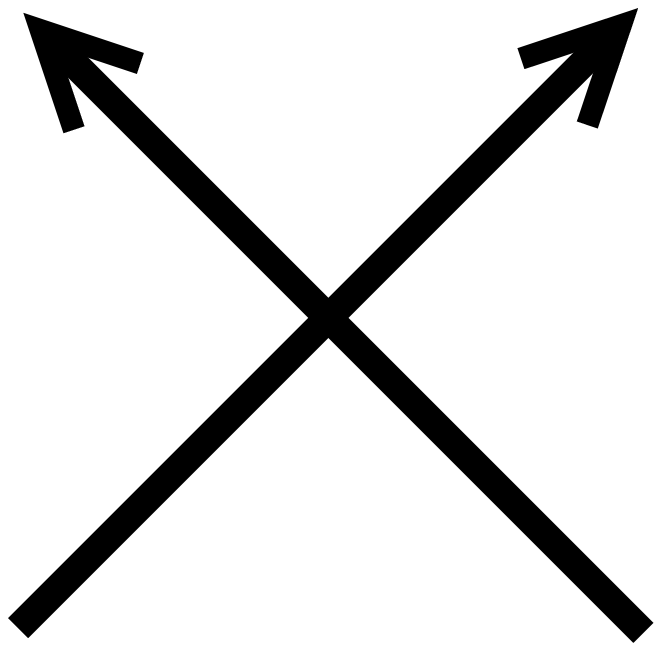}, \quad
Z\bigg( \begin{array}{c} {}_{(++)} \\  \figtotext{20}{20}{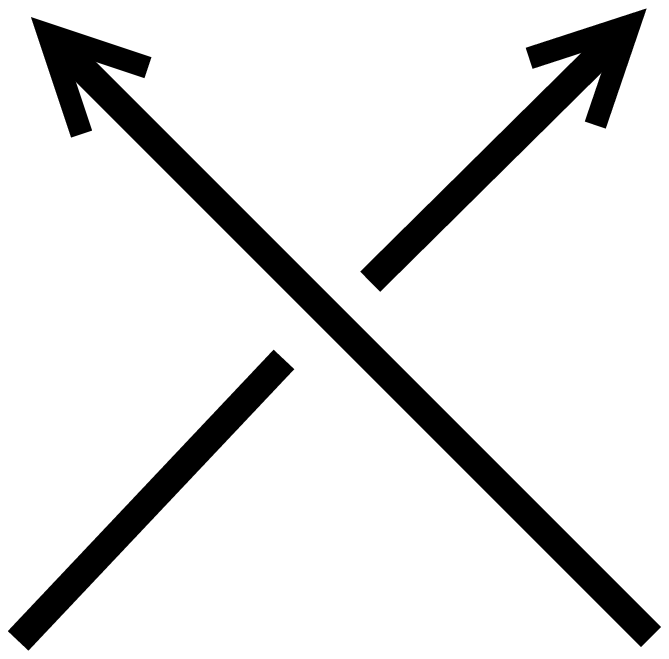} \\ {}^{(++)} \end{array} \bigg)
= \exp\Big(-\!\!\figtotext{20}{20}{t_uv}\!\!/2\Big)  \figtotext{20}{20}{crossing}, 
$$
$$
Z\bigg( \begin{array}{c} {}_{(+)} \\  \figtotext{10}{20}{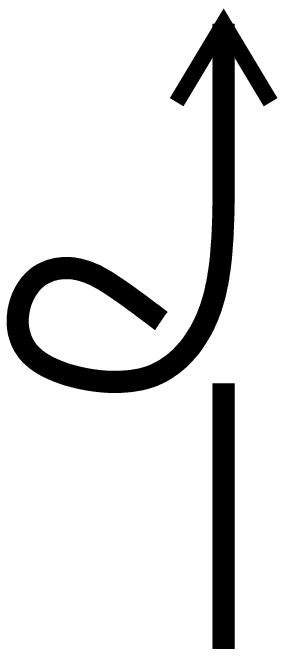} \ \\ {}^{(+)} \end{array} \bigg)
= \exp\Big( \figtotext{10}{20}{isolated_chord}\!\!/2\Big), \qquad 
Z\bigg( \begin{array}{c} {}_{(+)} \\  \figtotext{10}{20}{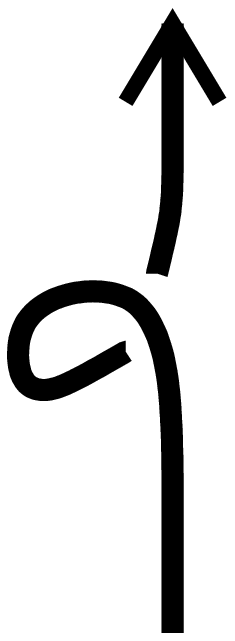} \ \\ {}^{(+)} \end{array} \bigg)
= \exp\Big( -\!\!\figtotext{10}{20}{isolated_chord}\!\!/2\Big);
$$
\item[(iii)] a Drinfeld associator $\Phi \in \A(\uparrow_1 \uparrow_2 \uparrow_3)$ is fixed and 
$$
Z\bigg(\begin{array}{c} {}_{(+ (++))} \\ \figtotext{25}{25}{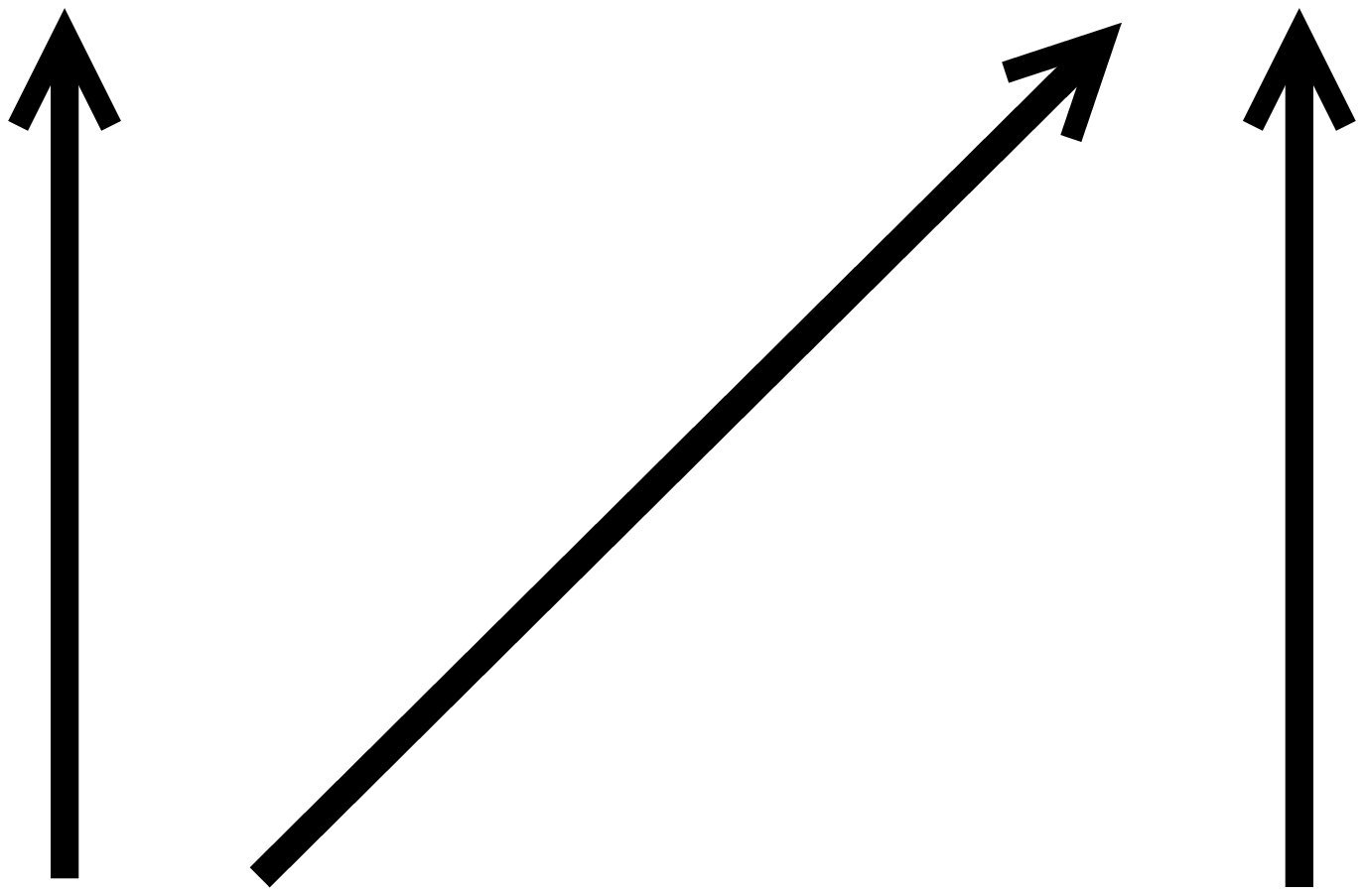} \\ {}^{((++) +)}\end{array}\bigg)  = \Phi.
$$
\end{itemize}
To be more specific about (iii), we assume here that  $\Phi$ is  \emph{horizontal},
i.e$.$ $\Phi$ is the exponential of an infinite Lie series in  the elements $t_{12}, t_{23} \in \A(\uparrow_1 \uparrow_2 \uparrow_3)$.

The Kontsevich integral produces special expansions in the following way.
Recall the notations of Sections  \ref{subsec:unframed_special} and \ref{subsec:framed_special}.
We identify $\overrightarrow{\pi} = \pi_1(U \Sigma,\vec \ast\, )$ with the framed $1$-strand braid group in $\Sigma$ 
or, equivalently, with the group of framed  $(p+1)$-strand pure braids in the disk that become trivial if the last strand is deleted;
then the Kontsevich integral restricts to a multiplicative map
$$
Z: \overrightarrow{\pi} \longrightarrow \A_{p,\ast}.
$$
 Here, and unless otherwise specified in this sequel, {\it  {we equip any  braid in the disk  with the leftmost parenthesizing $(\cdots ((++)+)\cdots +)$.}}
In the unframed case, we identify ${\pi} = \pi_1( \Sigma, \ast)$ with the $1$-strand braid group in $\Sigma$
or, equivalently, with the group $PB_{p+1}$ of unframed $(p+1)$-strand pure braids in the disk that become trivial if the last string is deleted;
then the Kontsevich integral restricts to a multiplicative map
$$
Z: {\pi} \longrightarrow  \A_{p,\ast}/FI.
$$ 

\begin{proposition} \label{prop:construction}
The maps $\vec \theta_Z:  \overrightarrow{\pi} \to T(\!(H)\!) \hat\otimes \K[[C]]$ and  $\theta_Z:  {\pi} \to T(\!(H)\!) $ defined by 
$$
\xymatrix{
{\overrightarrow{\pi}} \ar[r]^-{Z}  \ar@{-->}[rd]_-{\vec \theta_Z} & \horiz{\A}_{{p,\ast}} \ar[d]^-{\eqref{eq:a_simple_map}^{-1}}_-\simeq \\
& T(\!(H)\!) \hat\otimes \K[[C]]
}
\qquad \hbox{and} \qquad
\xymatrix{
{{\pi}} \ar[r]^-{Z}  \ar@{-->}[rd]_-{ \theta_Z} &  \horiz{\A}_{{p,\ast}} /FI \ar[d]^-{\eqref{eq:another_simple_map}^{-1}}_-\simeq \\
& T(\!(H)\!)}
$$
are special expansions, which correspond one to the other by the bijection \eqref{eq:special's}.
\end{proposition}

\begin{proof}
First of all, to give sense to the definition of the map $\theta_Z$,
we have to justify that $Z: {\pi} \to \A_{p,\ast}/FI$ takes values in the subalgebra $\horiz{\A}_{p,\ast}/FI$.
(This will also  imply that $Z: \overrightarrow{\pi} \to \A_{p,\ast}$ takes values in the subalgebra $\horiz{\A}_{p,\ast}$, 
so that the map $\vec\theta_Z$ is well-defined too.)
On this purpose, we consider for any integer $n\geq 1$ 
the Drinfeld--Kohno Lie algebra $\mathfrak{PB}_n$ of \emph{infinitesimal $n$-strand pure braids}  \cite{Kh,Dr1,Dr2}:
by definition, $\mathfrak{PB}_n$  is the Lie algebra generated by the symbols $t_{ij}=t_{ji}$ for any distinct $i,j\in \{1,\dots,n\}$ subject to the relations
$$
[t_{ij},t_{kl}]= 0, \quad [t_{ij}+t_{ik}, t_{jk}] =0
$$
for any pairwise-disjoint $i,j,k,l\in \{1,\dots,n\}$;
recall that $\mathfrak{PB}_n$ is canonically isomorphic to the graded  Lie algebra associated  to the lower central series of the $n$-strand pure braid group $PB_n$.
It is well-known that the  short exact sequence of groups
$$
1 \longrightarrow \pi_1(D_n) \longrightarrow PB_{n+1} \longrightarrow PB_n \longrightarrow 1
$$
 where  the fundamental group   $ \pi_1(D_n)$ of a disk with $n$ punctures $D_n$
is mapped to the subgroup of $(n+1)$-strand pure braids that become trivial after forgetting the last strand,
induces a short exact sequence of Lie algebras
\begin{equation} \label{eq:ses_Lie}
0 \longrightarrow \mathfrak{L}_n \longrightarrow \mathfrak{PB}_{n+1} \longrightarrow \mathfrak{PB}_n \longrightarrow 0
\end{equation}
 where the   $n$ generators of the free Lie algebra  $\mathfrak{L}_n$ {are mapped to  $t_{1,n+1},\dots, t_{n,n+1} \in \mathfrak{PB}_{n+1}$.}
By  induction on $n\geq 1$,
we  deduce from \eqref{eq:ses_Lie} and from the injectivity of  the homomorphism  \eqref{eq:another_simple_map} 
{that the algebra map }
$$
U(\mathfrak{PB}_n) \longrightarrow \A(\uparrow_1 \cdots \uparrow_n)/FI, \ t_{ij} \longmapsto t_{ij}
$$
is injective: its image $\horiz{\A}(\uparrow_1 \cdots \uparrow_n)/FI$ is the subalgebra of $\A(\uparrow_1 \cdots \uparrow_n)/FI$
generated by the diagrams $t_{ij}$ for all distinct $i,j \in \{1,\dots,n\}$. ({This injectivity is well known:} see \cite[Corollary~4.4]{BN3} or \cite[Remark 16.2]{HgM}.)
Then it follows from \eqref{eq:ses_Lie} for $n:=p$  that we have a short exact sequence of Lie algebras
\begin{equation} \label{eq:ses_Lie_bis}
0 \longrightarrow \operatorname{Prim} \horiz{\A}_{p,\ast}/FI
\longrightarrow  \operatorname{Prim} \horiz{\A}(\uparrow_1 \cdots \uparrow_{p+1})/FI \longrightarrow \operatorname{Prim} \horiz{\A}(\uparrow_1 \cdots \uparrow_{p})/FI \longrightarrow 0
\end{equation}
where $\operatorname{Prim}(-)$ denotes the primitive part of a Hopf algebra. 
We now come back to the Kontsevich integral $Z$:
the horizontality of $\Phi$ implies that the restriction of $Z$ to {$PB_{p+1}$} takes values in $\horiz{\A}(\uparrow_1 \cdots \uparrow_{p+1})/FI$.
For any $x\in \pi$, $\log Z(x) \in \operatorname{Prim} \horiz{\A}(\uparrow_1 \cdots \uparrow_{p+1})/FI$ becomes trivial if the strand $\uparrow_{p+1}=\uparrow_\ast$ is deleted from Jacobi diagrams:
we deduce that $\log Z(x)$ belongs to $\operatorname{Prim} \horiz{\A}_{p,\ast}/FI$ or, equivalently, that $Z(x)$ belongs to $\horiz{\A}_{p,\ast}/FI$.

We now prove that $\vec\theta_Z$  verifies the four conditions (i), (ii), (ii'), and (iii) of a framed special expansion. 
Since $Z: \overrightarrow{\pi} \to \horiz{\A}_{p,\ast}$ is multiplicative and since \eqref{eq:a_simple_map} is an algebra homomorphism,
$\vec\theta_Z$ satisfies (i). We show (ii'): for any $k\in \Z$, 
$\digamma^k \in \overrightarrow{\pi}$ regarded as a $(p+1)$-strand framed pure braid is trivial except that its last strand has framing number $k$;
therefore $Z(\digamma^k)= \exp(\frac{k}{2}t_{\ast \ast})$ which transforms to $1 \hat \otimes  \exp(\frac{k}{2}C)$ by \eqref{eq:a_simple_map}.
We show (iii): $\vec{\overline \nu} \in \overrightarrow{\pi}$ regarded as a $(p+1)$-strand framed pure braid 
is obtained from the  $2$-strand framed pure braid
\begin{equation}   \label{eq:gamma}
\gamma :=  \begin{array}{c}
\labellist
\small\hair 2pt
 \pinlabel {$1$} [r] at 19 8
 \pinlabel {$2$} [l] at 102 8
\endlabellist
\includegraphics[scale=0.2]{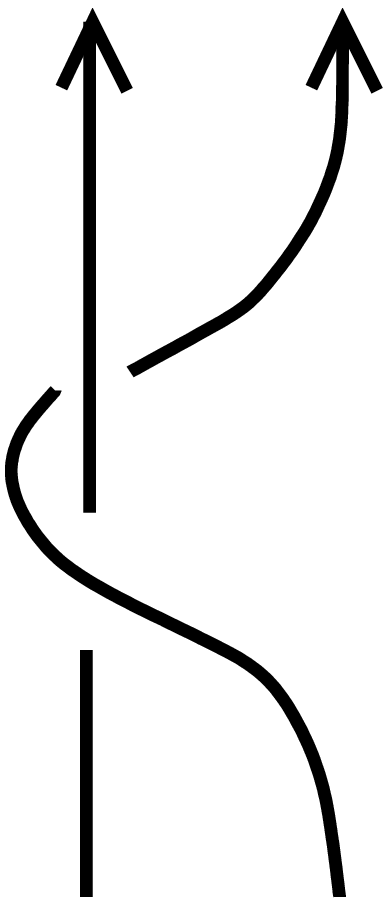} \end{array}
\end{equation}
by doubling $(p-1)$ times the first strand; therefore
$$
Z(\vec{\overline \nu}) =\Delta_{\uparrow_1 \mapsto \uparrow_1 \cdots \uparrow_p, \uparrow_2 \mapsto \uparrow_\ast } \big(Z(\gamma)\big)
=  \Delta_{\uparrow_1 \mapsto \uparrow_1 \cdots \uparrow_p, \uparrow_2 \mapsto \uparrow_\ast }\big(\exp(-t_{12})\big)
= \exp(-t_{1\ast}- \cdots -t_{p\ast})
$$
which transforms to $\exp(-z) \hat \otimes 1$ by \eqref{eq:a_simple_map}.
It remains to show (ii): let $i\in \{1,\dots,p\}$ and consider the following framed $(p+1)$-strand braids:
$$
\labellist
\small\hair 2pt
 \pinlabel {$\cdots$}  at 88 707
 \pinlabel {$\cdots$}  at 394 719
 \pinlabel {$\cdots$}  at 449 415
 \pinlabel {$\cdots$}  at 89 415
 \pinlabel {$\alpha_i:=$} [r] at 6 415
 \pinlabel {$\beta^{-1}_i:=$} [r] at 5 706
  \pinlabel {$\beta_i:=$} [r] at 5 136
   \pinlabel {$\cdots$}  at 88 136
 \pinlabel {$\cdots$}  at 394 136
  \pinlabel {$1$} [t] at 15 4
 \pinlabel {$i$} [t] at 229 4
 \pinlabel {$p$} [t] at 450 4
\endlabellist
\centering
\includegraphics[scale=0.18]{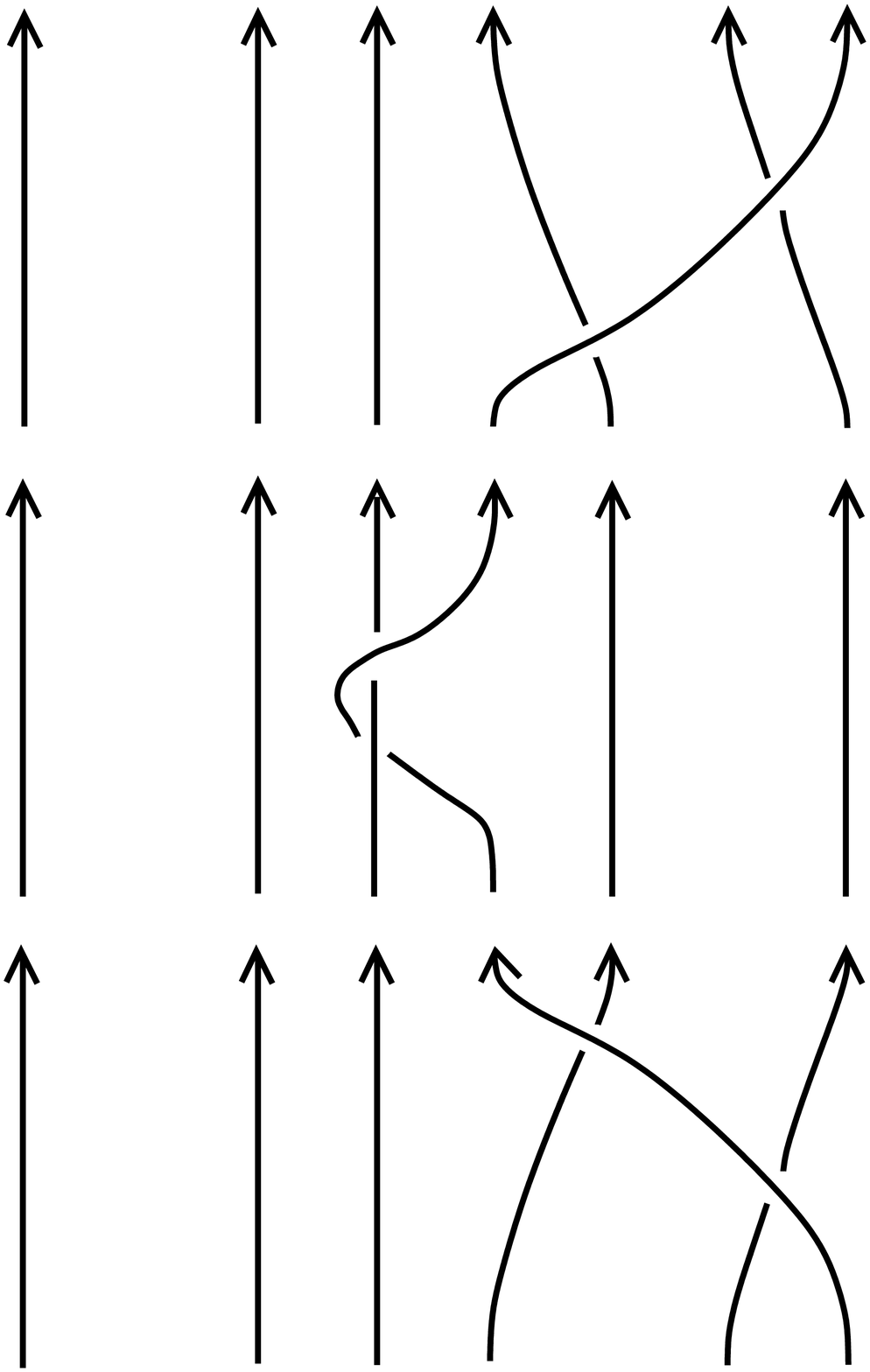}
$$

\vspace{0.2cm}
\noindent
The product $\vec\zeta_i:=\beta_i \alpha_i \beta_i^{-1}$ is a framed pure braid 
which can be regarded as an element of $\overrightarrow{\pi}$ and then represents~$\check \zeta_i$.
Let $l_{p+1}$ be the leftmost parenthesizing $(\cdots ((++)+)\cdots +)$ of length $(p+1)$ 
and let $l'_{p+1}$ be the parenthesizing 
that is obtained from $l_p$  by transforming the $i$-th letter $+$ to $(++)$.
Let ${(\beta_i)^\sim}$ be the braid $\beta_i$ with parenthesizing $l'_{p+1}$ at the top and $l_{p+1}$ at the bottom,
and let ${(\alpha_i)_\sim^\sim}$ be the braid $\alpha_i$ with parenthesizing $l'_{p+1}$ at the top and bottom. Then
\begin{eqnarray*}
Z(\vec\zeta_i) &=& Z\big({(\beta_i)^\sim}\big)\, Z\big({(\alpha_i)_\sim^\sim}\big)\, Z\big({(\beta_i)^\sim}\big)^{-1} \\
&=&  \left(Z\big({(\beta_i)^\sim}\big)\theta_i\right)\, \left(\theta_i^{-1}Z\big({(\alpha_i)_\sim^\sim}\big)\theta_i\right)\, 
\left(Z\big({(\beta_i)^\sim}\big) \theta_i\right)^{-1}
\end{eqnarray*}
where $\theta_i$ is the empty Jacobi diagram on the abstract oriented $1$-manifold\\[0.1cm]
$$
\labellist
\small\hair 2pt
 \pinlabel {$\cdots$}  at 86 131
 \pinlabel {$\cdots$}  at 402 140
 \pinlabel {$1$} [b] at 13 265
 \pinlabel {$i$} [b] at 228 266
 \pinlabel {$p$} [b] at 443 265
\endlabellist
\includegraphics[scale=0.2]{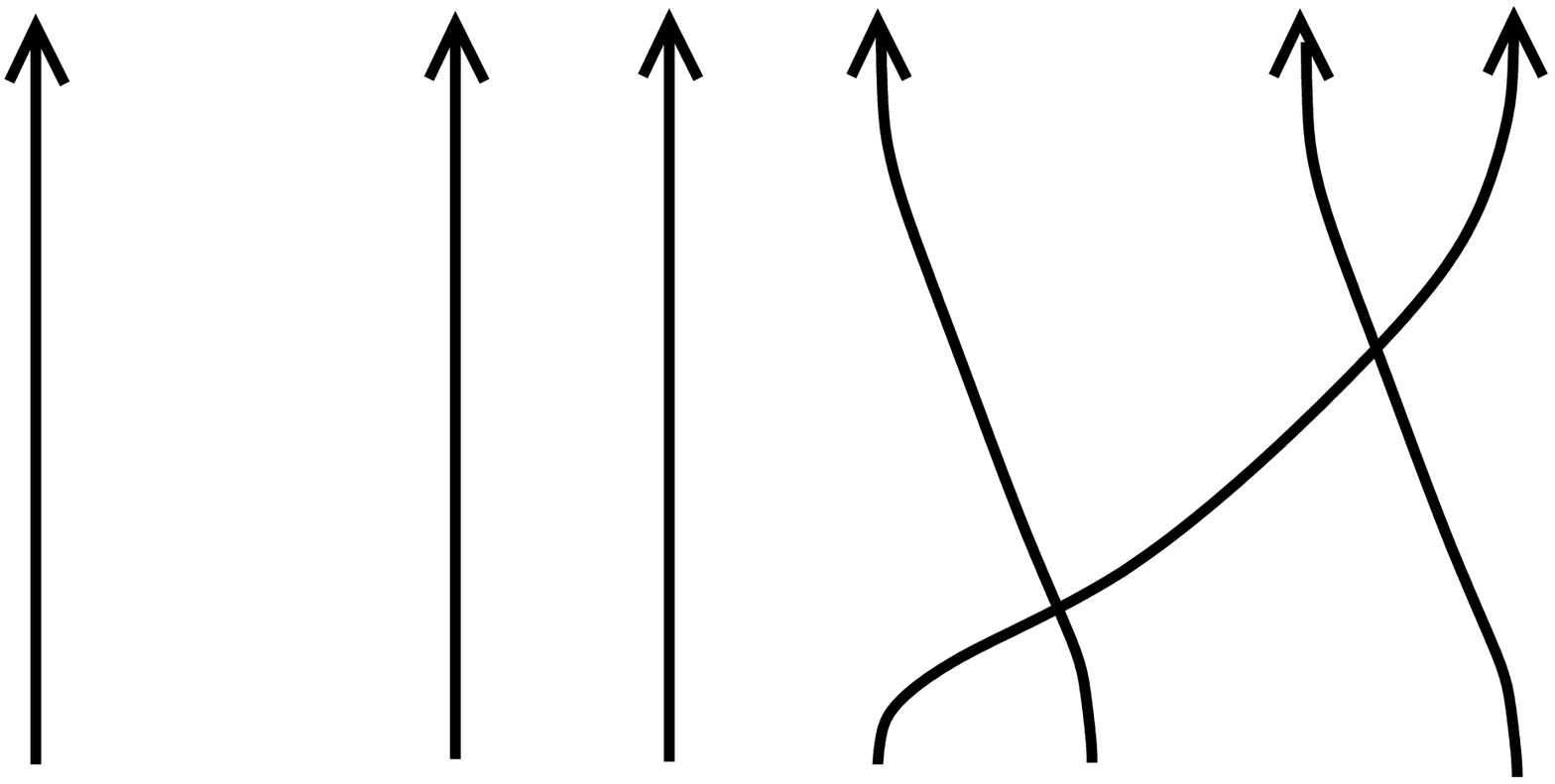}.
$$
Thus we have decomposed $Z(\vec\zeta_i)$ as a conjugate of group-like elements of $\A(\uparrow_1 \cdots \uparrow_p \uparrow_\ast)$,
and we now consider this decomposition in more details. On the one hand, we have
$$
\theta_i^{-1}Z\big({(\alpha_i)_\sim^\sim}\big)\theta_i =
\theta_i^{-1} (e^{t_{i,i+1}}) \theta_i = e^{t_{i\ast}};
$$
on the other hand, $Z\big({(\beta_i)^\sim}\big)\theta_i$ belongs to $\horiz{\A}(\uparrow_1\cdots \uparrow_p \uparrow_\ast)$ 
and is transformed to the empty Jacobi diagram if the interval  $\uparrow_\ast$ is deleted: 
it follows from \eqref{eq:ses_Lie_bis} that  $Z\big({(\beta_i)^\sim}\big)\theta_i$ belongs to $\horiz{\A}_{p,\ast}$.
Therefore, $Z\big({(\beta_i)^\sim}\big)\theta_i$ is mapped by the inverse of the isomorphism  \eqref{eq:a_simple_map} 
to a group-like element $U_i$ of $T(\!(H)\!) \hat \otimes \K[[C]$].
We conclude that $\vec\theta_Z(\vec \zeta_i) = U_i \exp(z_i) U_i^{-1} \hat \otimes 1$, 
which proves~(ii) since  we have $w(\vec{\zeta}_i)=0$ for the above choice of the representative  $\vec{\zeta}_i$ of $\check{\zeta}_i$.

Finally, the map $p_*(\vec\theta_Z): \pi \to T(\!(H)\!)$ derived from the special expansion $\vec\theta_Z$ of $\overrightarrow{\pi}$
is clearly equal to the map $\theta_Z$. Consequently, $\theta_Z$ is a special expansion of $\pi$.
\end{proof}

\subsection{Remarks} \label{subsec:remarks_special}

1) The unframed special expansion $\theta_Z$ given by Proposition \ref{prop:construction} appears implicitly in \cite{HgM} and explicitly in \cite{AET}.

2) Regard $\pi$ as a subgroup  of the group $PB_{p+1}$ of unframed $(p+1)$-strand pure braids in the disk. 
Proposition \ref{prop:construction} implies that the $I$-adic filtration of $\K[\pi]$  
coincides with the filtration inherited from the $I$-adic  filtration of $\K[PB_{p+1}]$.

\section{Formal description of Turaev's self-intersection map} \label{sec:statement}

In this section, the ground ring is a commutative field $\K$ of characteristic zero
and $\Sigma$ is a disk with finitely-many punctures numbered from $1$ to $p$.
Set $\pi:=\pi_1(\Sigma,\ast)$ and $\overrightarrow{\pi} :=\pi_1(\Sigma,\vec\ast\, )$ 
where $\ast \in \partial \Sigma$ and $\vec\ast$ denotes the unit vector tangent to $\overline{\partial \Sigma}$ at the point $\ast$.

Recall {the operation $\odot$ in $H$ which has been} introduced in Section~\ref{subsec:hip_special}.
It is easily verified that the $(-1)$-filtered linear map ${\xi}: T(\!(H)\!) \hat\otimes \K[[C]] \to T(\!(H)\!)$ defined~by 
$$
{\xi}(k_1 \cdots k_m \otimes C^n) := \left\{\begin{array}{ll}
\sum_{1 \leq i \leq m-1 } \ k_1 \cdots k_{i-1}\, (k_i {\odot} k_{i+1})\, k_{i+2} \cdots k_m & \hbox{if } n=0 \hbox{ and } m\geq 2\\
0 & \hbox{if } n=0  \hbox{ and  } m \in \{0,1\} \\
-2   k_1 \cdots  k_m & \hbox{if } n=1 \\\
0 & \hbox{if } n>1 
\end{array}\right.
$$
is a quasi-derivation ruled by the Fox pairing $(-\!{\odot}\!-)$.
Let  $s(X) \in \Q[[X]]$ be the series defined by \eqref{eq:s(X)},
and recall from Section~\ref{subsec:quasi-der} that any pair of elements $e_1,e_2 \in T(\!(H)\!)$ such that $e_1+e_2=s(-z)$ 
defines a quasi-derivation $q_{e_1,e_2}: T(\!(H)\!) \hat\otimes \K[[C]] \to T(\!(H)\!)$ ruled by the inner Fox pairing $\rho_{s(-z)}$:
 specifically, we have 
$$
\forall a \in T(\!(H)\!), \ \forall n\in \N, \quad q_{e_1,e_2}(a\otimes C^n) = 
\left\{\begin{array}{ll}
(\varepsilon(a)-a)\, e_1 + e_2\, (\varepsilon(a)-a) & \hbox{if } n=0 \\
0 & \hbox{if } n> 0.
\end{array}\right.
$$

We can now state our formal description of the homotopy self-intersection map $\vec{\mu}$.
We consider here the special expansions $\theta_Z: \K[\pi] \to T(\!(H)\!)$
and  $\vec \theta_Z: \K[\overrightarrow\pi] \to  T(\!(H)\!) \hat \otimes \K[[C]]$ induced by  the Kontsevich integral $Z$  (see Proposition \ref{prop:construction}).
Recall that the construction of  $Z$ depends on the choice of an associator $\Phi$, which is assumed to be horizontal.

\begin{theorem} \label{th:vec_mu}
There exists a  series $\phi(X) \in \K[[X]]$ depending  explicitly on $\Phi$ such that 
$$
\phi(-X) - \phi(X)-1/2 = s(X)
$$
and the following diagram is commutative:
$$
\xymatrix{
{\widehat{\K[ \overrightarrow{\pi} ] }} \ar[d]_-{\hat{\vec{\theta}}_Z}^-\simeq \ar[rr]^-{\hat{\vec{\mu}}} && {\widehat{\K[{\pi}]}}   \ar[d]^-{\hat \theta_Z}_-\simeq & \\
 T(\!(H)\!) \hat \otimes \K[[C]] \ar[rr]_-{{\xi} + q } && T(\!(H)\!)  & \hbox{where $q:= q_{-1/4+\phi(z)\, ,\, -1/4-  \phi(-z)}$}
}
$$
\end{theorem}

\begin{corollary} \label{cor:even_case}
If the associator $\Phi$ is  even,  then
the following diagram is commutative:
$$
\xymatrix{
{\widehat{\K[ \overrightarrow{\pi} ] }} \ar[d]_-{{\hat{\vec\theta}}_Z}^-\simeq \ar[rr]^-{\hat{\vec{\mu}}} && {\widehat{\K[{\pi}]}}   \ar[d]^-{\hat \theta_Z}_-\simeq & \\
 T(\!(H)\!) \hat\otimes \K[[C]] \ar[rr]_-{{\xi} + q } && T(\!(H)\!)  & \hbox{where $q:= q_{s(-z)/2\, ,\, s(-z)/2}$}
}
$$
\end{corollary}

{We shall derive from Theorem \ref{th:vec_mu} a  formal description of the Turaev cobracket $\delta_{\operatorname{T}}$.
To state this result,} recall from Section~\ref{subsec:quasi-der} the following notations:
$$ 
\check{T}(H) = T(H)/ [T(H),T(H)]  \quad \hbox{and}   \quad \vert\check{T}(H)\vert = T(H)/([T(H),T(H)]+ \K 1) {.}
$$
Let  $\delta_{\operatorname{S}}:\check{T}(H) \to \check{T}(H) \otimes \check{T}(H) $ be the linear map defined by $\delta_{\operatorname{S}}(1):=0$, $\delta_{\operatorname{S}}(k):=0$  for any $k\in H$ and by 
\begin{eqnarray*}
\delta_{\operatorname{S}}({k_1 \cdots k_m})
&:=& \sum_{1 \leq i<j \leq m}  {k_{i+1} \cdots k_{j-1}} \wedge  { k_1 \cdots k_{i-1}\, (k_i {\odot} k_j)\, k_{j+1} \cdots k_m} \\
&& -  \sum_{1 \leq i<j \leq m}  { (k_j {\odot} k_i)\, k_{i+1} \cdots k_{j-1}} \wedge { k_1 \cdots k_{i-1}  k_{j+1} \cdots k_m}
\end{eqnarray*}
for any integer $m\geq 2$ and for all $k_1,\dots, k_m\in H$,
where we use the notation 
$$
\forall x,y \in T(H), \quad x\wedge y := x \otimes y -y \otimes x \in T(H) \otimes T(H).
$$
We refer to  this map as the \emph{Schedler cobracket} 
since it is the Lie cobracket associated by Schedler to a star-shaped quiver consisting of one ``central'' vertex  connected by $p$ edges to $p$ ``peripheral'' vertices   \cite{Sc_necklace}.
The map $\delta_{\operatorname{S}}$ also induces a linear map
$\delta_{\operatorname{S}}:{\vert\check{T}(H)\vert \to {\vert\check{T}(H)\vert \otimes \vert \check{T}(H)\vert}}$ 
defined by  $\delta_{\operatorname{S}}(\vert k\vert):=0$  for any $k\in H$ and
\begin{eqnarray*}
\delta_{\operatorname{S}}(\vert{k_1 \cdots k_m}\vert) &=& 
\sum_{1 \leq i<j \leq m}   \vert {k_{i+1} \cdots k_{j-1}} \vert \wedge  \vert { k_1 \cdots k_{i-1}\, (k_i {\odot}  k_j)\, k_{j+1} \cdots k_m} \vert \\
&& -  \sum_{1 \leq i<j \leq m}  \vert {(k_j {\odot} k_i)\, k_{i+1} \cdots k_{j-1}} \vert \wedge \vert { k_1 \cdots k_{i-1}  k_{j+1} \cdots k_m} \vert
\end{eqnarray*}
for any integer $m\geq 2$ and for all $k_1,\dots, k_m\in H$,
where $\vert\! -\! \vert: \check{T}(H) \to \vert \check{T}(H)\vert$ denotes the canonical projection. 
Since $\delta_{\operatorname{S}}$ shifts degrees by $(-1)$, it  induces a $(-1)$-filtered linear map 
$$
\hat \delta_{\operatorname{S}}:\vert\check{T}(\!(H)\!)\vert \longrightarrow \vert\check{T}(\!(H)\!)\vert \hat\otimes \vert \check{T}(\!(H)\!)\vert
$$ 
where $\vert\check{T}(\!(H)\!) \vert$ denotes the degree-completion of $\vert\check{T}(H)\vert$.
Let $\reallywidehat{ \vert \K{\check{\pi}}\vert }$ denote 
the completion of $\vert \K\check{\pi}\vert \simeq \K[\pi]\big/\big(\big[ \K[\pi],\K[\pi]\big] + \K 1\big)$ 
with respect to the filtration that it inherits from $\K[\pi]$.

\begin{corollary} \label{cor:delta}
For any  associator $\Phi$, the following diagram is commutative:
$$
\xymatrix{
 {\reallywidehat{ \vert \K{\check{\pi}}\vert } }  \ar[d]_-{\hat \theta_Z}^-\simeq \ar[rr]^-{\hat \delta_{\operatorname{T}}} 
 &&    {\reallywidehat{ \vert \K{\check{\pi}}\vert } } \hat\otimes  {\reallywidehat{ \vert \K{\check{\pi}}\vert } }  \ar[d]^-{\hat \theta_Z \hat\otimes \hat \theta_Z}_-\simeq & \\
 \vert\check{T}(\!(H)\!)\vert  \ar[rr]_-{\hat \delta_{\operatorname{S}}} && \vert\check{T} (\!(H)\!)\vert \hat \otimes \vert \check{T}(\!(H)\!)\vert
 }
$$
\end{corollary}

\section{Proof of the formal descriptions} \label{sec:proofs}

In this section, the ground ring is a commutative field $\K$ of characteristic zero
and $\Sigma$ is {a} disk with finitely-many punctures numbered from $1$ to $p$.
We prove the results that have been stated in Section~\ref{sec:statement}.
Before that, we introduce some diagrammatic operators 
and we illustrate their efficiency by giving a second proof of Theorem \ref{th:special} 
(when the special expansion under consideration is induced by the Kontsevich integral).

\subsection{Some diagrammatic operators}

We use the same notations as in Section~\ref{subsec:lambda-dim3}. 
In particular, recall that we have a free product decomposition
$$
\pi^v = \iota^v(\pi) * F(z^v)
$$
where $\pi = \pi_1(\Sigma,\ast)$  and  $\pi^v = \pi_1(\Sigma\setminus\{v\},\ast)$, 
which induces a projection $p^v: \pi^v \to \iota^v (\pi) \simeq \pi$.
The goal of this subsection is to give a diagrammatic description of the composition
$$
\xymatrix{
{\K[\pi^v]} \ar[r]^-{\frac{\partial  \ }{\partial z^v}}  & {\K[\pi^v]}  \ar[r]^{p^v} & {\K[\pi]}. 
}
$$

On this purpose, we  regard the unframed pure braid group on $2$ strands in $\Sigma$ 
as the group of unframed $(p+2)$-strand pure braids in the disk that become trivial if the last two strings are deleted.
Thus, by restriction of the Kontsevich integral, we obtain a multiplicative  map
$$
Z: PB_2(\Sigma) \longrightarrow \A_{p,uv}/FI
$$
where $\A_{p,uv}$ denotes the closed subspace of $\A(\uparrow_1 \cdots \uparrow_p \uparrow_u \uparrow_v)$
spanned by Jacobi diagrams whose all connected components touch $\uparrow_u$ or $\uparrow_v$,
and $FI$ is the ``Framing Independence'' relation.
Here, and unless otherwise specified in {the} sequel, {\emph{any $(p+2)$-strand pure braid in the disk arising from  $PB_2(\Sigma)$ 
is  equipped with the parenthesizing}} $(l_p(++))$ where $l_p$ is the leftmost parenthesizing $(\cdots ((++)+)\cdots +)$ of lenght $p$.

We also need some additional notations.
We denote by $\varepsilon^{u {,} v}: \A_{p,uv}/FI \to \K$ the map that ``deletes'' the two strands $\uparrow_u$ and $\uparrow_v$,
and by  $\varepsilon^u: {\A_{p,uv}/FI} \to {\A_{p,\ast}/FI}$   the map that only ``deletes'' $\uparrow_u$ and renames $\ast$ the label $v$.
Let $P_u$ be the closed subspace of ${ \A_{p,uv}/FI}$
 spanned by Jacobi diagrams whose all connected components  touch $\uparrow_u$, 
and let $\horiz{P}_u$ be the subalgebra of ${ \A_{p,uv}/FI}$ generated by $t_{uv}, t_{1u},\dots,t_{pu}$: clearly $\horiz{P}_u \subset P_u$.
We shall also need the subalgebra $\horiz{P}_{u,v}$ of ${ \A_{p,uv}/FI}$ 
generated by commutators of the form 
$$
[t_{i_1j_1},[t_{i_2j_2},\cdots [t_{i_{n-1}j_{n-1}},t_{i_n j_n}]\cdots ]]
$$
for all  $n\in \N$ and  $i_1,\dots,i_n,j_1,\dots,j_n \in \{1,\dots,p,u,v\}$,
where $i_k\neq j_k $ and  $\{i_k,j_k \} \cap \{u,v\} \neq \varnothing$ for any $k \in \{1,\dots, n\}$, 
and where at least one of the pairs $\{i_1,j_1\}$, \dots, $\{i_n,j_n\}$ is equal to $\{u,v\}$.
Finally, for any diagram $y\in \A_{p,uv}/FI$, we~set
$$
y^\times := 
\begin{array}{c}\labellist
\scriptsize\hair 2pt
\pinlabel {$\cdots$}  at 220 418
 \pinlabel {$\cdots$}  at 209 84
 \pinlabel {$u$} [r] at 435 10
 \pinlabel {$v$} [r] at 590 12
 \pinlabel {$p$} [r] at 345 12
 \pinlabel {$1$} [r] at 100 15
 \pinlabel {$y$}  at 350 256
\endlabellist
\centering
\includegraphics[scale=0.1]{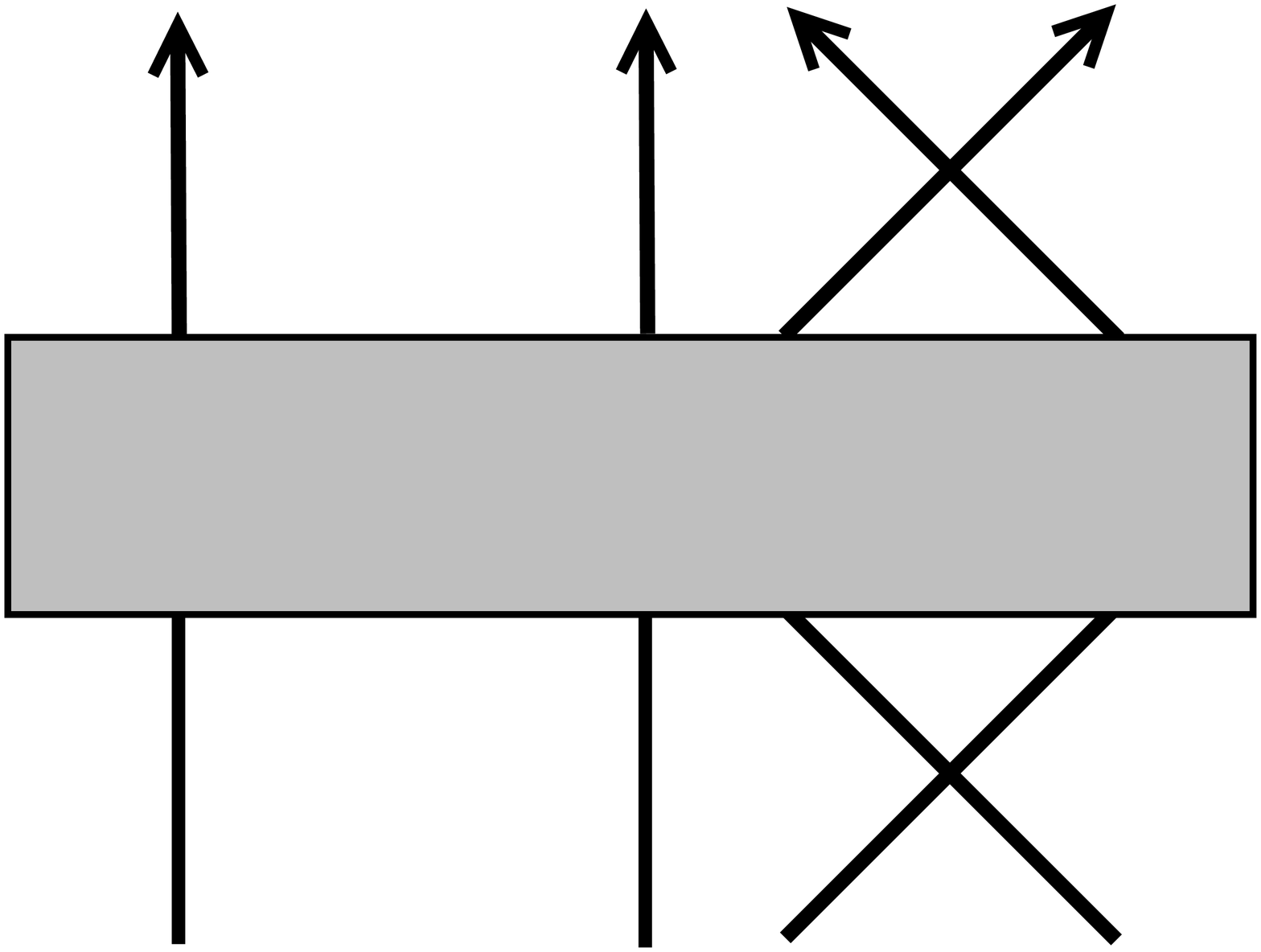} \end{array}.
$$
The above notations relative to the points $u$ and $v$ 
have obvious analogues when the roles of $u$ and $v$ are reversed.
Clearly $\horiz{P}_{u,v}=\horiz{P}_{v,u}$ and $y^\times \in \horiz{P}_{u,v}$ for all $y\in \horiz{P}_{u,v}$.

\begin{lemma} \label{lem:vvv}
There is a unique $(-1)$-filtered linear  map $D^v: \horiz{P}_u \to \horiz{\A}_{p,\ast}/FI\stackrel{\eqref{eq:another_simple_map}}{\simeq} T(\!(H)\!)$ such that the following  diagram is commutative:
\begin{equation} \label{eq:square_vvv}
\xymatrix{
\K[\pi^v] \ar[r]^-{\frac{\partial\ }{\partial z^v}} \ar[d]_-{Z} & \K[\pi^v]  \ar[r]^{p^v} & \K[\pi] \ar[d]^-{Z} \\
\horiz{P}_u \ar[rr]^-{D^v} & &  \horiz{\A}_{p,\ast}/FI
}
\end{equation}
Moreover, we have the following properties:
\begin{equation} \label{eq:vvv1}
\forall x,y\in \horiz{P}_u, \quad D^v(xy) = D^v(x)\ \varepsilon^{u,v}(y) +  \varepsilon^v(x)\, D^v(y);
\end{equation}
\begin{equation}  \label{eq:vvv2}
\horiz{P}_{u,v}\subset \horiz{P}_{u} \qquad \hbox{and} \qquad \forall y \in \horiz{P}_{u,v}, \quad D^v(y^\times) = \hat SD^v(y);
\end{equation}
\begin{equation}  \label{eq:vvv3}
D^v(t_{uv})=-1, \quad 
\forall i\in \{1,\dots,p\}, \ D^v(t_{iu}) = - z_i \phi,
\end{equation}
where $\phi\in T(\!(H)\!)$ is a constant (depending on $\Phi$) such that $\phi -\hat S(\phi) = 1/2 +s(-z)$.
\end{lemma}

\begin{proof}
The group $\pi^v$ is the fundamental group of the disk with $p+1$ punctures.
Thus, by using the same arguments as in Proposition \ref{prop:construction}, we obtain that the restriction to $\pi^v$ of the Kontsevich integral $Z: PB_2(\Sigma) \to \A_{p,uv}/FI$
takes values in $\horiz{P}_u$, and that the resulting linear  map $Z:\K[\pi^v] \to \horiz{P}_u$ induces an isomorphism of complete algebras
\begin{equation} \label{eq:Z_P_u}
Z:\widehat{\K[\pi^v]} \longrightarrow \horiz{P}_u.
\end{equation}
Here $\widehat{\K[\pi^v]} $ denotes the completion of ${\K[\pi^v]}$ with respect to its $I$-adic filtration, which is also 
the filtration inherited from the $I$-adic filtration of $\K[PB_{p+2}]$ (see Remark \ref{subsec:remarks_special}.2).
Besides, observe the following fact:   for any integer $m \geq 1$,
$\frac{\partial }{\partial z^v}$ maps the $m$-th term of the $I$-adic filtration of $\K[\pi^v]$ to its $(m-1)$-st term,
which implies that the linear map $p^v \frac{\partial \ }{\partial z^v}:  {\K[\pi^v]} \to {\K[\pi]}$ is $(-1)$-filtered.
In particular, it induces a $(-1)$-filtered linear map $p^v \frac{\partial \ }{\partial z^v}: \widehat{\K[\pi^v]} \to \widehat{\K[\pi]}$.
Then we define the map  $D^v$  by the following composition
$$
\xymatrix{
\widehat{\K[\pi^v]} \ar[rr]^-{p^v \frac{\partial \ }{\partial z^v} }&& \widehat{\K[\pi]} \ar[d]^-{Z}_-\simeq \\
\horiz{P}_u \ar@{-->}[rr]^-{ D^v} \ar[u]^-{Z^{-1}}_-\simeq & & {\horiz{\A}_{p,\ast}}/FI.
}
$$
This proves the existence of the map $D^v$ in the commutative square \eqref{eq:square_vvv}, and the unicity is clear:
since $\K[\pi^v]$ is dense in $\widehat{\K[\pi^v]}$,
the image of $\K[\pi^v]$ by $Z$ is dense in $\horiz{P}_u$.

To prove property \eqref{eq:vvv1}, let $x,y \in \horiz{P}_u$.
We  set $x':= Z^{-1}(x)  \in \widehat{\K[\pi^v]}  $ and  $y':= Z^{-1}(y) \in \widehat{\K[\pi^v]} $.
Then 
\begin{eqnarray*}
D^v(xy)  \ = \ D^v\big(Z(x')\,Z(y')\big) 
& \by{eq:square_vvv}  & Zp^v \Big(\frac{\partial x'y'}{\partial z^v} \Big) \\
&=&  Zp^v \Big(\frac{\partial x'}{\partial z^v} \hat\varepsilon(y') + x' \frac{\partial y'}{\partial z^v} \Big) \\
& \by{eq:square_vvv} & \hat \varepsilon(y')\, D^v(x) + Z\big(p^v(x')\big)\, D^v (y) \\
& = & \varepsilon^{u,v}(y)\, D^v(x) +  \varepsilon^v(x)\, D^v (y).
\end{eqnarray*}

We  prove the first part of property \eqref{eq:vvv2}.
Firstly, we  show that $Z: PB_2(\Sigma) \to \A_{p,uv}/FI$ maps $\pi^u \cap \pi^v$ to $\horiz{P}_{u,v}$.
Let $\beta \in \pi^u \cap \pi^v \subset PB_2(\Sigma)$. 
Since $\beta\in \pi^v$, we have $Z(\beta) \in \horiz{P}_u$ 
so that $\log Z(\beta)$ is a series of commutators of the form
$$
[t_{i_1u},[t_{i_2u},\cdots [t_{i_{n-1}u},t_{i_n u}]\cdots ]]
$$
where $n\in \N$  and  $i_1,\dots,i_n\in \{1,\dots,p,v\}$;
by distinguishing those commutators that involve a $t_{uv}$ from those commutators that only involve $t_{1u},\dots,t_{pu}$,
we can decompose $\log Z(\beta)$ as a sum of two terms  $(\log Z(\beta))'$ and $(\log Z(\beta))''$.
Since $\beta \in \pi^u$, we obtain 
$$
0= \varepsilon^v \log Z(\beta) = \varepsilon^v \big((\log Z(\beta))''\big)
$$
and it follows that $(\log Z(\beta))''=0$: in particular $\log Z(\beta)$ belongs to $\horiz{P}_{u,v}$, so that $Z(\beta)$ belongs to $\horiz{P}_{u,v}$.
Thus, the Kontsevich integral $Z: PB_2(\Sigma) \to \A_{p,uv}/FI$ induces a filtration-preserving homomorphism
\begin{equation} \label{eq:Z_P_uv}
Z: \widehat{\K[\pi^u \cap \pi^v]}{\longrightarrow}  \horiz{P}_{u,v}.
\end{equation}
Here  $\widehat{\K[\pi^u \cap \pi^v]} $ denotes the completion of ${\K[\pi^u \cap \pi^v]}$ 
with respect to the filtration inherited from the $I$-adic filtration of $\K[PB_{p+2}]$
(which does \emph{not} {coincide} with the $I$-adic filtration of ${\K[\pi^u \cap \pi^v]}$) 
or, equivalently,  it is the completion with respect to the filtration inherited from the $I$-adic filtration of $\K[\pi^v]$.
Next we show that \eqref{eq:Z_P_uv} is surjective,  which will imply that $\horiz{P}_{u,v} \subset \horiz{P}_{u}$.
Consider the following $(p+2)$-strand  pure braid  for any elements $i,j \in \{1,\dots,p,u,v\}$:
$$
\tau_{ij} :=
\begin{array}{c}
\labellist
\small\hair 2pt
 \pinlabel {$1$} [t] at 11 4
 \pinlabel {$i$} [t] at 264 4
 \pinlabel {$j$} [t] at 659 4
  \pinlabel {$p$} [t] at 870 4
 \pinlabel {$u$} [t] at 950 4
 \pinlabel {$v$} [t] at 1010 4
 \pinlabel {$\cdots$} at 96 44
 \pinlabel {$\cdots$}  at 466 41
  \pinlabel {$\cdots$}  at 806 41
\endlabellist
\centering
\includegraphics[scale=0.2]{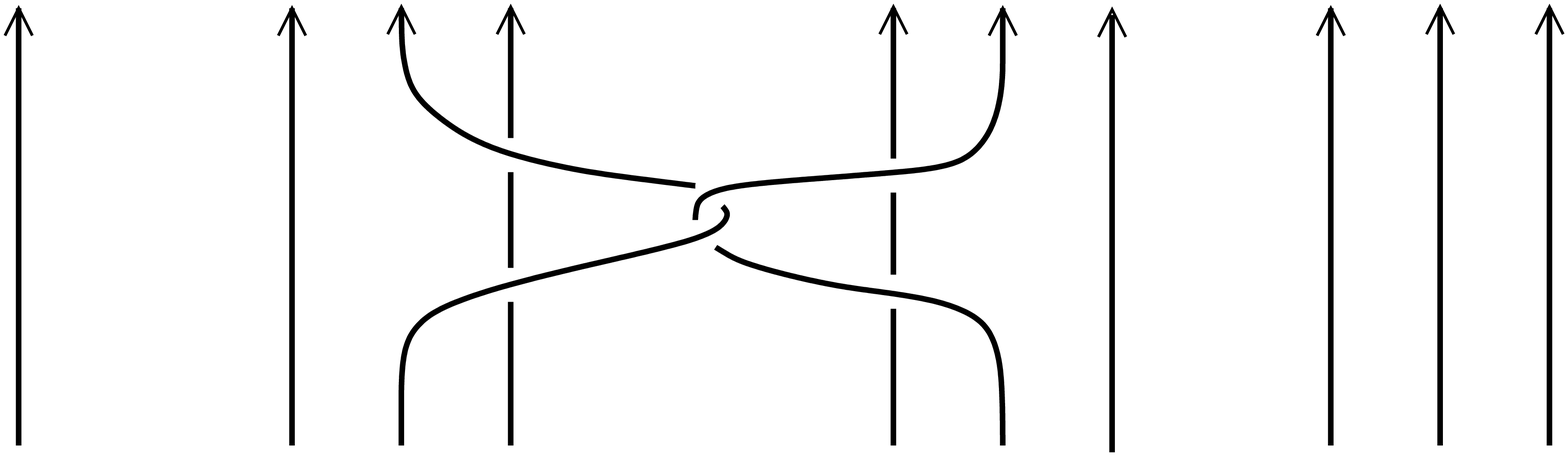}
\end{array}
$$

\vspace{0.1cm}
\noindent
and note that $Z(\tau_{ij})= 1+ t_{ij} + (\deg >1)$. Therefore, we have
$$
Z\big([\tau_{i_1j_1},[\tau_{i_2 j_2},\cdots [\tau_{i_{n-1} j_{n-1}},\tau_{i_n j_n}]\cdots ]]\big) =
1+ [t_{i_1j_1},[t_{i_2 j_2},\cdots [t_{i_{n-1} j_{n-1}},t_{i_n j_n}]\cdots ]] + (\deg >n)
$$
for any  $n\in \N$ and  $i_1,\dots,i_n,j_1,\dots,j_n \in \{1,\dots,p,u,v\}$ where $i_k\neq j_k$ and $\{i_k,j_k \} \cap \{u,v\} \neq \varnothing$ for all $k \in \{1,\dots, n\}$;
if we further assume that  at least one of the pairs $\{i_1,j_1\}$, \dots, $\{i_n,j_n\}$ is equal to $\{u,v\}$,
then the pure braid $[\tau_{i_1j_1},[\tau_{i_2 j_2},\cdots [\tau_{i_{n-1} j_{n-1}},\tau_{i_n j_n}]\cdots ]] \in PB_2(\Sigma)$ 
becomes trivial if any of the strands $u$ or $v$ is deleted, which implies that it belongs to $\pi^u \cap \pi^v$.
This fact implies that the filtered map \eqref{eq:Z_P_uv} is surjective at the graded level, 
which implies its surjectivity.

We now prove the second part of property \eqref{eq:vvv2}.
Since $\K[\pi^u \cap \pi^v]$ is dense in $\widehat{\K[\pi^u \cap \pi^v]}$ and \eqref{eq:Z_P_uv} is surjective, 
it is enough to prove the identity of \eqref{eq:vvv2} 
for $y  := Z(y') \in \horiz{P}_{u,v}$ where $y'$ is an arbitrary element of $\pi^u \cap \pi^v\subset PB_2(\Sigma)$. 
Then, using  Lemma \ref{lem:U_V}, we obtain 
$$
Zp^v \Big(\frac{\partial \sigma y' \sigma^{-1} }{\partial z^v}\Big) 
= Z\Big(\overline{p^v \frac{\partial  y'  }{\partial z^v} } \Big) = \hat S Z\Big(p^v \frac{\partial  y'  }{\partial z^v}  \Big)
\by{eq:square_vvv} \hat S D^v Z(y') = \hat SD^v (y).
$$
Besides we have
\begin{eqnarray*}
Zp^v \big(\frac{\partial \sigma y' \sigma^{-1} }{\partial z^v} \big) &\by{eq:square_vvv}& D^v Z (\sigma y' \sigma^{-1}) \\
&=& D^v  \big( {e^{t_{uv}/2}} {\!\ecrossing} Z(y')\,  {e^{-t_{uv}/2}}  {\!\ecrossing} \big) \\
&=& D^v  \big( e^{t_{uv}/2} Z(y')^\times  e^{-t_{uv}/2} \big) \\ 
&\by{eq:vvv1} & D^v ( e^{t_{uv}/2})  +  D^v \big(  Z(y')^\times  e^{-t_{uv}/2} \big) \\
& \by{eq:vvv1} &  D^v ( e^{t_{uv}/2})    + D^v \big(Z(y')^\times\big) + D^v ( e^{-t_{uv}/2}) \\ 
& \by{eq:vvv1} &  D^v  \big(Z(y')^\times\big) \ = \ D^v  ( y^\times).
\end{eqnarray*}
We deduce that $\hat SD^v (y) =  D^v  ( y^\times)$ as required.

We now prove the first identity of  \eqref{eq:vvv3}.
To compute $D^v(t_{uv})$, we consider $T:= e^{-{t_{uv}}} \in  \horiz{P}_{u}$. Then 
\begin{eqnarray*}
D^v(t_{uv}) &= &D^v\big(-\log(1+(T-1))\big) \\
& = & - D^v\Big(\sum_{r\geq 1} \frac{(-1)^{r+1}}{r} (T-1)^r\Big)
\ = \  \sum_{r\geq 1} \frac{(-1)^{r}}{r} D^v\big((T-1)^r\big);
\end{eqnarray*}
for any $r\geq 2$, we have
$$
D^v \big((T-1)^r\big)=D^v\big((T-1)^{r-1}(T-1)\big)  \by{eq:vvv1} 0 
$$
and we have
$$
D^v(T-1)= D^v(T) = D^v  Z(z^v) \by{eq:square_vvv} Z p^v( \frac{\partial z^v}{\partial z^v}) = 1. 
$$
It follows that 
\begin{equation} \label{eq:recall'}
D^v(t_{uv}) = -1. 
\end{equation}

We now prove the second identity of  \eqref{eq:vvv3}.
Let $x\in \pi$ and {consider}  $\iota^v(x)\in \pi^v\subset PB_2(\Sigma)$. We have
\begin{equation} \label{eq:change_par}
Z(\iota^v(x))= \varphi^{-1}\, Z({\iota^v(x)_\sim^\sim})\, \varphi
\end{equation}
where ${\iota^v(x)_\sim^\sim}$ is the pure braid $\iota^v(x)$ equipped with the parenthesizing $((l_p+)+)$ instead of $(l_p (++))$ {at the top and bottom,}
and where $\varphi \in {\A_{p,uv}/FI}$ is constructed from  the associator $\Phi\in \A(\uparrow_1 \uparrow_2 \uparrow_3)$
 by duplicating $(p-1)$ times the first string: 
\begin{equation} \label{eq:Phi_to_varphi}
\varphi := \Delta_{\uparrow_1 \mapsto \uparrow_1 \cdots \uparrow_p,\, \uparrow_2 \mapsto \uparrow_u,\, \uparrow_3 \mapsto \uparrow_v} (\Phi) \in {\A_{p,uv}/FI}.
\end{equation}
Since $\Phi$ is assumed to be horizontal, we have $\varphi \in \horiz{P}_{u,v}$.
It follows from \eqref{eq:change_par} that
\begin{eqnarray*}
D^v Z(\iota^v(x)) \ = \  D^v \left( \varphi^{-1}\, Z({\iota^v(x)_\sim^\sim})\, \varphi \right) 
&\by{eq:vvv1}& D^v(\varphi^{-1}) +D^v \big(  Z({\iota^v(x)_\sim^\sim})\, \varphi \big) \\
& \by{eq:vvv1} & D^v(\varphi^{-1}) +D^v  Z({\iota^v(x)_\sim^\sim})  +  Z(x)\,   D^v(\varphi) \\
& \by{eq:vvv1} & D^v  Z({\iota^v(x)_\sim^\sim})  +  \big(Z(x)-1\big)\,   D^v(\varphi).
\end{eqnarray*}
Besides, we have
$$ 
D^v Z(\iota^v(x))  \by{eq:square_vvv}  Z p^v \frac{\partial \iota^v(x)}{\partial z^v}= 0
\quad \hbox{and} \quad
 Z({\iota^v(x)_\sim^\sim})  = {Z(x)_u}\! \uparrow_v
$$
where $ {Z(x)_u}\! \uparrow_v$ means that  $Z(x) \in {{\A}_{p,\ast}}/FI$
is transformed to an element of  $\A_{p,uv}/FI$ by labelling $u$ the interval $\uparrow_\ast$
and by juxtaposing a disjoint interval $\uparrow_v$.
We deduce that
$$
\forall x\in \pi, \quad D^v \left({Z(x)_u} \!\uparrow_v\right) = \big(1- Z(x)\big)\,  D^v(\varphi).
$$
Since the image of $\K[\pi]$ by $Z$ is dense in  $\horiz{\A}_{p,\ast}/FI \simeq T(\!(H)\!)$, it follows that
$$
\forall r\in {\horiz{\A}_{p,\ast}}/FI , \quad D^v ( {r_u} \! \uparrow_v ) = (\varepsilon(r)- r) D^v(\varphi).
$$
In particular, we obtain
\begin{equation} \label{eq:repeat}
\forall i\in\{1,\dots,p\}, \quad D^v(t_{iu})  = -z_i\,  D^v(\varphi).
\end{equation}

In the sequel, we set $\phi := D^v(\varphi)$ and it remains to prove that $\phi -\hat S(\phi) = 1/2 +s(-z)$.
Denote by $\kappa\in \pi^v$ the homotopy class of $\overline{\partial(\Sigma\setminus \{v\})}$ and let $t_{\square u}:= t_{1u}+ \cdots + t_{pu}$: we have 
$$
Z p^v \big( \frac{\partial \kappa}{\partial z^v}\big)= Zp^v\big( \frac{\partial \iota^v({\overline \nu}) z^v}{\partial z^v}\big) = Z({\overline \nu})  = e^{-z}.
$$
{Using the fact that $D^v (e^{t_{uv}/2} )=-1/2$ which follows from \eqref{eq:recall'}, 
decomposing $\kappa$ in the form $\sigma( \sigma^{-1} \kappa \sigma)\sigma^{-1}$,
and observing that $ \sigma^{-1} \kappa \sigma$ is  obtained from the pure braid \eqref{eq:gamma} by doubling $p$ times its first strand,}
we obtain 
\begin{eqnarray*}
D^v Z(\kappa) &=& D^v  \left( {e^{t_{uv}/2}} \!\ecrossing\, \varphi^{-1} e^{-t_{\square v}-t_{uv}} \varphi\,\, { e^{-t_{uv}/2 } } \!\ecrossing \right) \\
&=& D^v \left(e^{t_{uv}/2}  (\varphi^{-1})^\times e^{-t_{\square u}-t_{uv}} \varphi^\times e^{-t_{uv}/2}  \right) \\
&\by{eq:vvv1}& D^v (e^{t_{uv}/2} ) + D^v\left((\varphi^{-1})^\times e^{-t_{\square u}-t_{uv}} \varphi^\times e^{-t_{uv}/2}  \right) \\
&\by{eq:vvv1}& -1/2+  D^v((\varphi^{-1})^\times) + D^v\left( e^{-t_{\square u}-t_{uv}} \varphi^\times e^{-t_{uv}/2}  \right) \\
&\by{eq:vvv1}& -1/2 -  D^v(\varphi^{\times}) + D^v(e^{-t_{\square u}-t_{uv}}) + e^{-z} D^v  \left(  \varphi^\times e^{-t_{uv}/2}  \right) \\
& \by{eq:vvv1} & -1/2 -  D^v(\varphi^{\times}) + D^v(e^{-t_{\square u}-t_{uv}}) + e^{-z} D^v(  \varphi^\times ) + e^{-z} D^v(   e^{-t_{uv}/2}) \\
&=& -1/2 -  D^v(\varphi^{\times}) + D^v(e^{-t_{\square u}-t_{uv}}) + e^{-z} D^v (  \varphi^\times ) + e^{-z}/2 \\
& \by{eq:vvv2} &  D^v(e^{-t_{\square u}-t_{uv}})  +(e^{-z} -1)/2  + (e^{-z} -1) \hat S(\phi).
\end{eqnarray*}
Using  \eqref{eq:square_vvv}, we deduce that
$$
e^{-z}/2 +1/2 =   D^v(e^{-t_{\square u}-t_{uv}})   + (e^{-z} -1) \hat S(\phi).
$$
We now compute
\begin{eqnarray*}
D^v (e^{-t_{\square u}-t_{uv}}) & = & \sum_{k\geq 0} \frac{1}{k!} D^v\big( (-t_{\square u}-t_{uv})^k\big) \\
& \by{eq:vvv1}&  \sum_{k\geq 1} \frac{1}{k!} {(-z)}^{k-1} D^v (-t_{\square u}-t_{uv}) \\
& \by{eq:recall'} &\frac{e^{-z}-1}{-z}\, (D^v(-t_{\square u})+1) \ \by{eq:repeat}  \ (1-e^{-z})\, \phi + \frac{1-e^{-z}}{z}
\end{eqnarray*}
and it follows that
$$
(e^{-z}-1)/2 +1 =   (1-e^{-z})\, \phi + \frac{1-e^{-z}}{z}  + (e^{-z} -1) \hat S(\phi)  
$$
or, {equivalently,} $1/2 + s(-z) =\phi - \hat S(\phi).$
\end{proof}

We now give a diagrammatic description of the composition
$$
\xymatrix{
{\K[\pi^u]} \ar[r]^-{\frac{\partial  \ }{\partial z^u}}  & {\K[\pi^u]}  \ar[r]^{p^u} & {\K[\pi]}
}
$$
which is similar to Lemma \ref{lem:vvv}.

\begin{lemma} \label{lem:uuu}
There is a unique $(-1)$-filtered linear  map $D^u: \horiz{P}_{v} \to \horiz{\A}_{p,\ast}/FI \stackrel{\eqref{eq:another_simple_map}}{\simeq}T(\!(H)\!)$ such that the following  diagram is commutative:
\begin{equation} \label{eq:square_uuu}
\xymatrix{
\K[\pi^u] \ar[r]^-{\frac{\partial \ }{\partial z^u}} \ar[d]_-{Z}& \K[\pi^u]  \ar[r]^{p^u} & \K[\pi] \ar[d]^-{Z} \\
\horiz{P}_{v} \ar[rr]^-{D^u} & &  \horiz{\A}_{p,\ast}/FI
}
\end{equation}
Moreover, we have the following properties: 
\begin{equation} \label{eq:uuu1}
\forall x,y \in \horiz{P}_{v}, \quad D^u(xy) = D^u(x)\ \varepsilon^{u,v}(y) +  \varepsilon^u(x)\, D^u(y);
\end{equation}
\begin{equation} \label{eq:uuu2}
\horiz{P}_{u,v} \subset \horiz{P}_{v}\qquad \hbox{and} \qquad \forall y \in  \horiz{P}_{u,v}, \quad D^u(y)= \hat S{D^v(y)};
\end{equation}
\begin{equation} \label{eq:uuu3}
D^u(t_{uv})=-1, \qquad \forall i\in \{1,\dots,p\}, \  D^u(t_{iv})=  -z_i(1/2 +\phi),
\end{equation}
where $\phi \in T(\!(H)\!)$ is the same  constant as in Lemma \ref{lem:vvv}.
\end{lemma}

\begin{proof}
The existence and unicity of the map $D^u$, as well as its property \eqref{eq:uuu1}, are proved as we did for the map $D^v$ in Lemma \ref{lem:vvv}.

We now prove property \eqref{eq:uuu2}.
The inclusion $\horiz{P}_{u,v} \subset \horiz{P}_{v}$ is proved in the same way as 
we proved that $\horiz{P}_{u,v} \subset \horiz{P}_{u}$  in  Lemma \ref{lem:vvv}.
We saw during the proof of this lemma that the Kontsevich integral  restricts to  a surjective filtration-preserving map
$$
Z: \widehat{\K[\pi^u \cap \pi^v]}{\longrightarrow}  \horiz{P}_{u,v}.
$$
Therefore, it is enough to prove the identity \eqref{eq:uuu2} 
for $y  := Z(y') \in \horiz{P}_{u,v}$ where $y'$ is an arbitrary element of $\pi^u \cap \pi^v\subset PB_2(\Sigma)$. 
Using Lemma \ref{lem:U_V}, we obtain 
\begin{eqnarray*}
D^u (y) \ = \ D^uZ(y') & \by{eq:square_uuu} & Z p^u \Big( \frac{\partial y'}{\partial z^u} \Big) \\
&=&  Z\Big( \overline{ p^v \Big( \frac{\partial y'}{\partial z^v} \Big)} \Big) \\ 
&=&  \hat S Z   p^v \Big( \frac{\partial y'}{\partial z^v} \Big)  \  \by{eq:square_vvv}   \ \hat S D^v Z(y')  
\ = \  \hat S D^v (y).
\end{eqnarray*}

We now prove \eqref{eq:uuu3}. The fact that $D^u(t_{uv})=-1$ follows from  \eqref{eq:uuu2} and \eqref{eq:vvv3}.
It remains to compute  $D^u(t_{iv})$ for any $i\in \{1,\dots,p\}$. Let $x\in \pi$ and consider its image $\iota^u(x) \in \pi^u$. 
{Since $\iota^u(x) = \sigma \iota^v (x) \sigma^{-1} \in PB_2(\Sigma)$, we obtain }
\begin{eqnarray*} 
 \notag Z(\iota^u(x)) &=& {e^{t_{uv}/2}} \!\ecrossing \varphi^{-1} Z ( {\iota^v(x)_\sim^\sim}) \varphi\,  {e^{-t_{uv}/2}}  \!\ecrossing \\
&=& e^{t_{uv}/2} (\varphi^{-1})^\times  \big(Z ( {\iota^v(x)_\sim^\sim})\big)^\times \varphi^\times  e^{-t_{uv}/2}  
\end{eqnarray*}
where ${\iota^v(x)_\sim^\sim}$ is the  pure braid $\iota^v(x)\in \pi^v \subset PB_2(\Sigma)$ 
computed with the parenthesizing $((l_p +)+)$ instead of the parenthesizing $(l_p(++))$,
and  where $\varphi\in  \horiz{P}_{u,v}$
is constructed  from the associator $\Phi\in \A(\uparrow_1 \uparrow_2 \uparrow_3)$ 
as we have  explained in the proof of Lemma~\ref{lem:vvv}.
Besides, we have 
$$ 
D^u Z(\iota^u(x))  \by{eq:square_uuu}  Z p^u \Big(\frac{\partial \iota^u(x)}{\partial z^u} \Big) = 0,
$$
and we also have
$$
 Z({\iota^v(x)_\sim^\sim})  ={Z(x)_u} \! \uparrow_v,
$$
meaning that  $Z(x) \in \A_{p,\ast}/FI$
is transformed to an element of  $\A_{p,uv}/FI$ by labelling $u$ the interval $\uparrow_\ast$
and by juxtaposing a disjoint  interval $\uparrow_v$. It follows that 
$$
\forall x \in \pi, \quad 
0 = D^u  \big(e^{t_{uv}/2} (\varphi^{-1})^\times  \big( {Z(x)_u}\! \uparrow_v \big)^\times \varphi^\times  e^{-t_{uv}/2} \big) 
$$
which implies that
$$
\forall r \in {\horiz{\A}_{p,\ast}}/FI, \quad 
0 = D^u  \big(e^{t_{uv}/2} (\varphi^{-1})^\times  \big( r_u \uparrow_v \big)^\times \varphi^\times  e^{-t_{uv}/2} \big).
$$
In particular, we obtain 
\begin{equation} \label{eq:ul2}
\forall i\in \{1,\dots,p\}, \quad 0 = D^u \big(e^{t_{uv}/2} (\varphi^{-1})^\times\,  t_{iv} \, \varphi^\times  e^{-t_{uv}/2} \big).
\end{equation}
We now develop this identity using \eqref{eq:uuu1}:
\begin{eqnarray*}
0 & = & D^u  \big(e^{t_{uv}/2} (\varphi^{-1})^\times\,  t_{iv}  \, \varphi^\times  e^{-t_{uv}/2} \big) \\
&=&  D^u \big(  t_{iv} \, \varphi^\times  e^{-t_{uv}/2} \big)\\
&=& D^u (t_{iv}) + z_i\, D^u  \big( \varphi^\times  e^{-t_{uv}/2} \big) \\
&=& D^u (t_{iv}) + z_i\, D^u  ( \varphi^\times) + z_i\, D^u ( e^{-t_{uv}/2}) \\
& \by{eq:uuu2} & D^u (t_{iv}) + z_i\, \hat SD^v  ( \varphi^\times) + z_i\, \hat SD^v ( e^{-t_{uv}/2})\\
&{\displaystyle \mathop{=}^{\eqref{eq:vvv2}}_{ \eqref{eq:vvv3}} } & D^u (t_{iv}) + z_i\, D^v ( \varphi) + \frac{1}{2} z_i \ = \  D^u (t_{iv}) + z_i \phi + \frac{1}{2} z_i
\end{eqnarray*}
{where $\phi := D^v(\varphi)$ as in the proof of Lemma~\ref{lem:vvv}.}
We  conclude that $D^u (t_{iv}) = -z_i(1/2+ \phi)$.
\end{proof}

\subsection{Another proof of Theorem \ref{th:special}.} 

In this subsection, we prove Theorem \ref{th:special} in the case where the special expansion $\theta$ under consideration is the expansion $\theta_Z$ induced by the Kontsevich integral $Z$.

Let $E$ be the  bilinear map defined by the following composition:
$$
\xymatrix{
\widehat{\K[\pi]} \times \widehat{\K[\pi]}  \ar[rr]^-{\hat \eta} && \widehat{\K[\pi]} \ar[d]^-{Z}_-\simeq \\
{\big( \horiz{\A}_{p,\ast}/FI \big)} \times {\big( \horiz{\A}_{p,\ast}/FI \big)}  \ar@{-->}[rr]_-{E} \ar[u]^-{Z^{-1} \times Z^{-1}}_-\simeq  && \horiz{\A}_{p,\ast}/FI
}
$$
Thus $E$ is a filtered Fox pairing in the Hopf algebra $\horiz{\A}_{p,\ast}/FI $, 
which we identify with the Hopf algebra $T(\!(H)\!)$ via the isomorphism \eqref{eq:another_simple_map}.
For any $x,y\in \pi$, we deduce from Theorem \ref{th:3d-eta} that
\begin{eqnarray*}
E(Z^{-1}(x),Z^{-1}(y)) \
&=& Z \eta(x,y) \\
&=& Z p^v\left( \frac{\partial\, \iota^{u}(y^{-1}) \iota^v(x) \iota^{u}(y) }{\partial z^v}\right) \\
& \by{eq:square_vvv} &  D^v Z \big( \iota^{u}(y^{-1}) \iota^v(x) \iota^{u}(  {y}  ) \big) \\ 
&=& D^v \left( Z(\iota^{u}(y^{-1})) \, Z(\iota^v(x)) \, Z(\iota^{u}(y))  \right) \\
&=& D^v \Big( {e^{t_{uv}/2}}  \!\ecrossing \varphi^{-1} \big(Z(y)^{-1}_u\! \uparrow_v \big) \varphi\, {e^{-t_{uv}/2} } \!\ecrossing 
\  \varphi^{-1} \big(Z(x)_u\! \uparrow_v\big)  \varphi \\
&& \qquad {e^{t_{uv}/2}}  \!\ecrossing \varphi^{-1} \big(Z(y)_u\! \uparrow_v \big) \varphi\, {e^{-t_{uv}/2}}  \!\ecrossing  \Big)\\
&=& D^v \Big( e^{t_{uv}/2} (\varphi^\times)^{-1} \big(\uparrow_u Z(y)^{-1}_v\big) \varphi^\times e^{-t_{uv}/2} 
\  \varphi^{-1} \big(Z(x)_u \uparrow_v\big)  \varphi \\
&& \qquad e^{t_{uv}/2} (\varphi^\times)^{-1} \big( \uparrow_u Z(y)_v \big) \varphi^\times e^{-t_{uv}/2}  \Big),
\end{eqnarray*}
where $Z(x)_u\! \uparrow_v$ denotes the element of $\A_{p,uv}/FI$ obtained from $Z(x) \in {\A}_{p,\ast}/FI$
by changing the label $\ast$  to the label $u$ and by juxtaposing  a disjoint interval  $\uparrow_v$, 
and where  the notations $Z(y)^{\pm 1}_u\! \uparrow_v ,\ \uparrow_u\! Z(y)^{\pm 1}_v$  have similar meanings.
 Since $Z(\K[\pi])$ is dense in $\horiz{\A}_{p,\ast}/FI$, we deduce that  the pairing $E$ is given by 
\begin{eqnarray*}
E(a,b) &=&   D^v \Big( e^{t_{uv}/2} (\varphi^\times)^{-1} \big(\uparrow_u\! (S(b'))_v\big) \varphi^\times e^{-t_{uv}/2} 
\  \varphi^{-1} \big(a_u\! \uparrow_v\big)  \varphi \\
 && \qquad e^{t_{uv}/2} (\varphi^\times)^{-1} \big( \uparrow_u\! (b'')_v \big) \varphi^\times e^{-t_{uv}/2}  \Big)
\end{eqnarray*}
 for any $a,b \in \horiz{\A}_{p,\ast}/FI$.
In particular, we have for any $i,j \in \{1,\dots, p\}$
\begin{eqnarray*}
E(t_{i\ast},t_{j\ast})
&=&   D^v \big( - e^{t_{uv}/2} (\varphi^\times)^{-1} t_{jv} \varphi^\times e^{-t_{uv}/2} 
\,  \varphi^{-1}\,  t_{iu}  \varphi\, e^{t_{uv}/2} (\varphi^\times)^{-1} \, 1 \, \varphi^\times e^{-t_{uv}/2}   \\
&& \qquad + e^{t_{uv}/2} (\varphi^\times)^{-1}\, 1 \, \varphi^\times e^{-t_{uv}/2} 
\,  \varphi^{-1}\,  t_{iu}  \varphi\, e^{t_{uv}/2} (\varphi^\times)^{-1} \, t_{jv} \, \varphi^\times e^{-t_{uv}/2}  \big) \\
&=& D^v \big( - e^{t_{uv}/2} (\varphi^\times)^{-1} t_{jv} \varphi^\times e^{-t_{uv}/2} \,  \varphi^{-1}\,  t_{iu}  \varphi   \\
&& \qquad +    \varphi^{-1}\,  t_{iu}  \varphi\, e^{t_{uv}/2} (\varphi^\times)^{-1} \, t_{jv} \, \varphi^\times e^{-t_{uv}/2}  \big).
\end{eqnarray*}
The definition {\eqref{eq:Phi_to_varphi}} of $\varphi$ from the horizontal associator $\Phi$ implies that
 $\ell:= \log(\varphi) \in \A_{p,uv}/FI$ is a series of Lie words in $t_{1u}, \dots, t_{pu}, t_{uv}$
showing at least one occurence of $t_{uv}$: therefore
$$
U_i := \varphi^{-1}\,  t_{iu}  \varphi  -  t_{iu}  = e^{-\ell} t_{iu} e^{\ell}  -  t_{iu} = \exp(-[\ell,-])(t_{iu}) -  t_{iu} 
$$
 belongs to $\horiz{P}_{u,v}$; by the same argument,
$$
V_j :=  e^{t_{uv}/2} (\varphi^\times)^{-1} t_{jv} \varphi^\times e^{-t_{uv}/2} - t_{jv}
 $$
 belongs to $\horiz{P}_{u,v}$. Thus we obtain
\begin{eqnarray*}
E(t_{i\ast},t_{j\ast})
&=& D^v\big( - (t_{jv}+V_j)\, (t_{iu} +U_i) + (t_{iu} +U_i)(t_{jv}+V_j) \big) \\
&=& D^v\big([t_{iu},t_{jv}] + [U_i,t_{jv}] + [t_{iu},V_j] + [U_i,V_j]\big).
\end{eqnarray*}
Since $U_i \in \horiz{P}_{u,v}$, we have $[U_i,t_{jv}] \in \horiz{P}_{u,v}$ and, similarly,
since $V_j \in \horiz{P}_{u,v}$, we have $[t_{iu},V_j]  \in \horiz{P}_{u,v}$; 
since $\horiz{P}_{u,v}$ is a subalgebra of $\A_{p,uv}/FI$, we also have $[U_i,V_j] \in \horiz{P}_{u,v}$;
the STU relation  implies that $[t_{iu},t_{jv}]= \delta_{ij} [t_{iv},t_{uv}]$ and, in particular, $[t_{iu},t_{jv}] \in \horiz{P}_{u,v}$.
We now compute the value of $D^v$ on each of those four elements of $\horiz{P}_{u,v}$.
Since $\varepsilon^v (V_j) = \varepsilon^{u,v} (V_j) = \varepsilon^v (U_i) = \varepsilon^{u,v} (U_i)=0$, we have
$$
D^v([U_i,V_j]) \by{eq:vvv1} 0 ;
$$
since $\varepsilon^{u,v}(V_j) = \varepsilon^v(V_j)=\varepsilon^{u,v}(t_{iu})=0$, we have
\begin{eqnarray*}
D^v( [t_{iu},V_j]) \   \by{eq:vvv1} \ z_i D^v(V_j) 
&  \by{eq:uuu2} & z_i \hat SD^u(V_j) \\
&=& z_i \hat S D^u \left(-t_{jv} + e^{t_{uv}/2} (\varphi^\times)^{-1} t_{jv} \varphi^\times e^{-t_{uv}/2} \right) \\
&  \by{eq:uuu1} & - z_i \hat S D^u (t_{jv})  +  z_i \hat SD^u\left( t_{jv} \varphi^\times e^{-t_{uv}/2} \right) \\
& \by{eq:uuu1} & z_i \hat S\big( z_j D^u (\varphi^\times e^{-t_{uv}/2}) \big) \\
&=& - z_i\, \hat SD^u (\varphi^\times e^{-t_{uv}/2})\,  z_j \\
& \by{eq:uuu2} & - z_i\, D^v  (\varphi^\times e^{-t_{uv}/2})\,  z_j  \\
& \by{eq:vvv1} &  - z_i \big(D^v  (\varphi^\times)+ D^v( e^{-t_{uv}/2})\big)  z_j
\ \mathop{=}_{\eqref{eq:vvv2}}^{\eqref{eq:vvv3}} \ -z_i (\hat S(\phi) +1/2) z_j;
\end{eqnarray*}
since $\varepsilon^{u,v}(U_i) = \varepsilon^u(U_i)=\varepsilon^{u,v}(t_{jv})=0$, we have
\begin{eqnarray*}
D^v([U_i,t_{jv}]) & \by{eq:uuu2} & \hat SD^u([U_i,t_{jv}]) \\
& \by{eq:uuu1}& -\hat S (z_j D^u(U_i))\\
&=& \hat SD^u(U_i)\, z_j \\
& \by{eq:uuu2} & D^v(U_i)\, z_j \\
&=& D^v(-t_{iu} + \varphi^{-1} t_{iu} \varphi)\, z_j \\
& \by{eq:vvv1} & - D^v (t_{iu})\, z_j + D^v (t_{iu}\varphi)\, z_j  \ \by{eq:vvv1} \ z_i D^v(\varphi) z_j \ = \ z_i \phi z_j;
\end{eqnarray*}
finally, we have
\begin{eqnarray*}
D^v([t_{iu},t_{jv}]) \ = \ \delta_{ij} D^v([t_{iv},t_{uv}]) 
&\by{eq:uuu2}&  \delta_{ij} \hat SD^u([t_{iv},t_{uv}]) \\
&\by{eq:uuu1}&  \delta_{ij} \hat S\left( z_i D^u(t_{uv}) \right) \ \by{eq:uuu3} \ \delta_{ij} z_i \ = \ (z_i {\odot} z_j).
\end{eqnarray*}
We deduce that 
$$
E(t_{i\ast},t_{j\ast}) = (z_i {\odot} z_j) + z_i (\phi -\hat S(\phi) -1/2) z_j = (z_i {\odot} z_j) + z_i\, s(-z)\, z_j.
$$
Since the complete algebra $T(\!(H)\!) \simeq \horiz{\A}_{p,\ast}/FI$ is generated by the $z_i =t_{i\ast}$ for all $i\in \{1,\dots,p\}$,
we conclude that the Fox pairings $E$ and $( {- \odot - } )+ \rho_{s(-z)}$ coincide.

\subsection{Proof of Theorem \ref{th:vec_mu}}

Let $N: \horiz{\A}_{p,\ast} \to \horiz{\A}_{p,\ast} /FI $ be the  map defined by the following composition:
\begin{equation} \label{eq:N_def}
\xymatrix{
{\widehat{\K[ \overrightarrow{\pi} ] }} \ar[rr]^-{\hat{\vec{\mu}}} && {\widehat{\K[{\pi}]}}   \ar[d]^-{Z}_-\simeq  \\
 \horiz{\A}_{p,\ast}  \ar@{-->}[rr]_-{N } \ar[u]^-{Z^{-1}}_-\simeq &&  \horiz{\A}_{p,\ast} /FI
}
\end{equation}
In this subsection, $ \horiz{\A}_{p,\ast}$ {is identified with} $T(\!(H)\!) \hat\otimes \K[[C]]$ via the isomorphism \eqref{eq:a_simple_map},
and  $ \horiz{\A}_{p,\ast}/FI$ is identified with $T(\!(H)\!)$ via the isomorphism \eqref{eq:another_simple_map}. 

According to \eqref{eq:tensorial_eta}, the Fox pairing  $\eta$ translates 
into the Fox pairing ${(\!-\!{\odot}\!-\!) + \rho_{s(-z)}}$ through the special expansion $\theta_Z$.
Therefore,  the map $N$ is a quasi-derivation ruled by ${(\!-\!{\odot}\!-\!) + \rho_{s(-z)}}$.
We shall compute the values of $N$ on the generators $t_{\ast \ast}=C$ and $t_{i\ast}=z_i$ (for all $i\in\{1,\dots,p\}$) 
of the complete algebra $ \horiz{\A}_{p,\ast} \simeq T(\!(H)\!) \hat\otimes \K[[C]]$.

First we compute $N(t_{\ast \ast})$. 
Since $N$  is a quasi-derivation and $t_{\ast \ast}$ is mapped to zero under  the projection $p: \horiz{\A}_{p,\ast}  \to \horiz{\A}_{p,\ast}/FI$, we have
$$
N(-t_{\ast \ast}/2)  =  
N(e^{- t_{\ast \ast} /2}) =  N Z(\digamma^{-1}) \by{eq:N_def} Z \vec{\mu}(\digamma^{-1}) =  Z(1)= 1.
$$
We deduce that 
\begin{equation} \label{eq:C_case}
N(t_{\ast \ast}) = -2 = {\xi}(C) = {\xi}(C) + q_{-1/4+\phi\, ,\, -1/4-\hat S(\phi)}(C).
\end{equation}

The computation of $N(t_{i\ast})$ for an arbitrary element $i\in \{1,\dots,p\}$  needs some preliminaries.
For any  $\vec{x} \in \overrightarrow{\pi}$, we obtain using Theorem \ref{th:3d-mu} that
\begin{eqnarray*}
Z\vec{\mu}(\vec{x})  & =& 
Z p^v\left(\frac{\partial\ \iota^u(x^{-1})\, c(\vec{x})}{\partial z^v}\right)\\
& \by{eq:square_vvv} & D^v Z\big(\iota^u(x^{-1})\, c(\vec{x})\big) \\
& = &  D^v \left( Z(\iota^u(x^{-1})) \ Z( c( \vec{x})) \right) \\ 
& = &  D^v \left( {e^{t_{uv}/2} } \!\ecrossing \varphi^{-1} Z( {\iota^v(x^{-1})_\sim^\sim} ) \varphi\,  {e^{-t_{uv}/2} } \!\ecrossing  Z( c( \vec{x}) ) \right)\\
& = &  D^v \left( e^{t_{uv}/2} (\varphi^{-1})^{\times} \big(Z( {\iota^v(x^{-1})_\sim^\sim} )\big)^\times \varphi^\times  e^{-t_{uv}/2} \ Z( c( \vec{x}) ) \right) \\
&=&   D^v \left( e^{t_{uv}/2} (\varphi^{-1})^{\times} \big((X^{-1})_u\! \uparrow_v \big)^\times \varphi^\times  e^{-t_{uv}/2} 
\ \Delta_{\uparrow_\ast \mapsto \uparrow_u \uparrow_v}(\vec{X}) \right)
\end{eqnarray*}
where we have denoted  $X:=Z(x)\in \A_{p,\ast}/FI$ and $\vec{X}:= Z( \vec{x})\in  \A_{p,\ast}$,
and where $(X^{-1})_u\! \uparrow_v$ is the element of $\A_{p,uv}/FI$ obtained from $X^{-1}$
by labelling $u$ the interval $\uparrow_\ast$ and by  juxtaposing a disjoint interval $\uparrow_v$.  Thus we have obtained
\begin{equation} \label{eq:special_case}
\forall \vec{x} \in \overrightarrow{\pi}, \quad
N Z (\vec{x}) 
=  D^v \left( e^{t_{uv}/2} (\varphi^{-1})^{\times}\,  \big(\uparrow_u\! (X^{-1})_v  \big)\, \varphi^\times  e^{-t_{uv}/2}   
\ \Delta_{\uparrow_\ast \mapsto \uparrow_u \uparrow_v}(\vec{X}) \right).
\end{equation}
Since the image of $\K[\overrightarrow{\pi}]$ by $Z$ is dense in $\horiz{\A}_{p,\ast}$,  we deduce that
\begin{equation}  \label{eq:formula_N}
  N(\vec r) = D^v \left( e^{t_{uv}/2} (\varphi^{-1})^{\times}\,  \big(\uparrow_u (S(\vec{r}\,'))_v  \big)\, \varphi^\times  e^{-t_{uv}/2}   
\ \Delta_{\uparrow_\ast \mapsto \uparrow_u \uparrow_v}(\vec{r}\,'') \right) 
\end{equation}
for any $\vec{r} \in \horiz{\A}_{p,\ast}$,
where $r$ is the projection of $\vec r$ in  $ \horiz{\A}_{p,\ast}/FI$, 
$\Delta(\vec r) = \vec{r}\,' \otimes \vec{r}\,''$ denotes the coproduct  in $\horiz{\A}_{p,\ast}$, 
the projection of $\vec{r}\,' $ in $\horiz{\A}_{p,\ast}/FI$ is still denoted by $\vec{r}\, '$,
and $S$ denotes the antipode of ${\horiz{\A}_{p,\ast}/FI}$.
In particular, we obtain
\begin{eqnarray*}
N(t_{i\ast}) &=& D^v \Big( e^{t_{uv}/2} (\varphi^{-1})^{\times}\!  \big(\uparrow_u\! (S(t_{i\ast}))_v  \big)\, \varphi^\times  e^{-t_{uv}/2}   
\ \Delta_{\uparrow_\ast \mapsto \uparrow_u \uparrow_v}(1)  \\
&& \qquad +  e^{t_{uv}/2} (\varphi^{-1})^{\times}\,  \big(\uparrow_u\! (S(1))_v  \big)\, \varphi^\times  e^{-t_{uv}/2}   
\ \Delta_{\uparrow_\ast \mapsto \uparrow_u \uparrow_v}(t_{i\ast}) \Big) \\
&=& D^v \Big( -e^{t_{uv}/2} (\varphi^{-1})^{\times}\,  t_{iv}\, \varphi^\times  e^{-t_{uv}/2}    
 + \Delta_{\uparrow_\ast \mapsto \uparrow_u \uparrow_v}(t_{i\ast}) \Big)  \\
 &=& D^v \Big( -e^{t_{uv}/2} (\varphi^{-1})^{\times}\,  t_{iv}\, \varphi^\times  e^{-t_{uv}/2}    + t_{iu} + t_{iv} \Big) \\
 &=& D^v \Big( \underbrace{-e^{t_{uv}/2} (\varphi^{-1})^{\times}\,  t_{iv} \, \varphi^\times  e^{-t_{uv}/2}  
 +  t_{iv}}_{\in \horiz{P}_{u,v}} \Big) +D^v(t_{iu}) \\
 & \by{eq:uuu2} & 
 \hat SD^u \Big( {-e^{t_{uv}/2} (\varphi^{-1})^{\times}\,  t_{iv}\, \varphi^\times  e^{-t_{uv}/2}  +  t_{iv}} \Big) +D^v(t_{iu}) \\
 & \by{eq:ul2} & \hat SD^u (t_{iv}) +  D^v (t_{iu}) \\
 & {\displaystyle \mathop{=}^{\eqref{eq:uuu3}}_{\eqref{eq:vvv3}}} &  \hat S(-z_i(1/2+\phi)) - z_i \phi
  \ = \ (1/2+\hat S(\phi))z_i - z_i \phi.
\end{eqnarray*}
We deduce that
\begin{equation} \label{eq:k_case}
\forall i \in \{1,\dots,p\}, \ N(t_{i\ast}) = q_{\phi, -1/2-\hat S(\phi)}(z_i) =  {\xi}(z_i) + q_{-1/4+\phi\, ,\, -1/4-\hat S(\phi)}(z_i).
\end{equation}

Since $N$ and ${\xi} + q_{-1/4+\phi\, ,\, -1/4-\hat S(\phi)}$ are both quasi-derivations ruled by the Fox pairing $(\!-\!{\odot}\!-\!) + \rho_{s(-z)}$,
and since the complete  algebra $\horiz{\A}_{p,\ast} \simeq T(\!(H)\!) \hat\otimes \K[[C]]$ is generated by $C$ and~$H$,
we conclude thanks to  \eqref{eq:C_case} and  \eqref{eq:k_case}  that 
$$
N = {\xi} + q_{-1/4+\phi\, ,\, -1/4-\hat S(\phi)}.
$$

To conclude the proof  of Theorem \ref{th:vec_mu}, 
we still need to study in more detail the element $\phi = D^v (\varphi) \in T(\!(H)\!)$. Based on this purpose, we set 
$$
\ell := \log(\varphi) \in \mathcal{A}_{p,uv}/FI.
$$
Note that $\ell$ is obtained from $\log(\Phi) \in \A(\uparrow_1 \uparrow_2 \uparrow_3)$ 
in the same way as $\varphi$ is obtained from $\Phi$,
i.e$.$ by duplicating $(p-1)$ times the first string (see  the proof of Lemma \ref{lem:vvv}).
{If the interval $\uparrow_3$ is deleted, then  $\Phi$ becomes $1 \in  \A(\uparrow_1 \uparrow_2 )$
so that  $\log(\Phi)$ becomes $0  \in  \A(\uparrow_1 \uparrow_2 )$,} 
which implies that $\varepsilon^v(\ell) = \varepsilon^{u,v}(\ell)=0$. Therefore
\begin{eqnarray*}
\phi \ = \ D^v (\exp(\ell)) 
&=& \sum_{k\geq 1} \frac{1}{k!} D^v(\ell^k) \\
& \by{eq:vvv1}& \sum_{k\geq 1} \frac{1}{k!} \left( \varepsilon^v(\ell^{k-1})\,  D^v(\ell) + D^v(\ell^{k-1})\, \varepsilon^{u,v}(\ell) \right) \ = \ D^v(\ell).
\end{eqnarray*}
{We set $a:= t_{12}$  and $b:=t_{23}$. Recall that $\log(\Phi)$ is a Lie series in $a,b$ and that $\Phi$ satisfies a ``pentagon'' equation and two ``hexagon'' equations (see \cite{Oh}, for instance):
the pentagon implies  that the linear part of $\log(\Phi)$ is trivial, while the hexagon force its quadratic part  to be equal to $-[a,b]/24$.
Therefore, by isolating the $b$-linear part of $\log(\Phi)$, we obtain}
\begin{equation} \label{eq:log(Phi)}
\log(\Phi) = - \frac{1}{24} [a,b] + \sum_{i\geq 2} q_i\, (\operatorname{ad}_{a})^i(b) + \sum_{n\geq 3}  \sum_{j\in J_n}w_j(a,b) \  \in \A(\uparrow_1 \uparrow_2 \uparrow_3)
\end{equation}
where $\operatorname{ad}_{a}:=[a,-]$,
$q_2,q_3,\dots$ are elements of $\K$, $J_3,J_4,\dots$ are finite sets and, for all $n\geq 3$ and $j\in J_n$, 
$w_j(a,b)$ is a Lie word in $a,b$ of length $n$ containing at least two copies of $b$. 
(Note that the scalars  $q_2,q_3,\dots$ in this decomposition are unique.)
Setting now $a:= t_{1u}+\cdots + t_{pu}$ and $b:=t_{uv}$, we deduce that
\begin{equation} \label{eq:formula_ell}
\ell =   - \frac{1}{24} [a,b] +  \sum_{i\geq 2}  q_i\, (\operatorname{ad}_{a})^i(b) + \sum_{n\geq 3}  \sum_{j\in J_n}  w_j(a,b) 
\ \in \mathcal{A}_{p,uv} /FI.
\end{equation}

\begin{lemma}
For any {integer} $n\geq 3$ and any  $j\in J_n$, we have $D^v(w_j(a,b))=0$.
\end{lemma}

\begin{proof}
The Lie word $w_j(a,b)$ can be written in the form $[b,m]$ or  in the form $[a,m]$
  where $m$ is a certain Lie word in $a,b$ of length at least $2$ and, in the second  case,
we also require that $m$ contains at least two copies of $b$. In the first case, we have 
$$
D^v([b,m])  \by{eq:vvv1}  \underbrace{\varepsilon^v(b)}_{=0}\, D^v(m)  + D^v(b)\, \varepsilon^{u,v}(m)
-  \varepsilon^v(m)\, D^v(b) - D^v(m)\, \underbrace{\varepsilon^{u,v}(b)}_{=0}  
$$
and, since $m$ is Lie word in $a,b$ of length at least $2$, we also have $\varepsilon^v(m)= \varepsilon^{u,v}(m) =0$ so that $D^v([b,m])  = 0$.
In the second case, we have
\begin{eqnarray*}
D^v([a,m])  &\by{eq:vvv1}  &\varepsilon^v(a)\, D^v(m) + D^v(a)\, \underbrace{\varepsilon^{u,v}(m)}_{=0}
 - \underbrace{\varepsilon^v(m)}_{=0} D^v(a)- D^v (m)\, \underbrace{\varepsilon^{u,v}(a)}_{=0} \\
 &=& z  D^v(m)
\end{eqnarray*}
and, by an  induction on the length of $m$, we conclude that $D^v([a,m]) =0$ as well.
\end{proof}

\begin{lemma}
For any {integer} $i\geq 1$, we have $D^v\big((\operatorname{ad}_{a})^i(b)\big)= -z^i$.
\end{lemma}

\begin{proof}
For any {integer $i\geq 1$}, we have
\begin{eqnarray*}
 D^v \big((\operatorname{ad}_{a})^i(b)\big) &=& D^v \big([a,(\operatorname{ad}_{a})^{i-1}(b)]\big)\\
&\by{eq:vvv1} & \varepsilon^v(a)\, D^v\big((\operatorname{ad}_{a})^{i-1}(b)\big)  + D^v(a)\, \underbrace{\varepsilon^{u,v}\big((\operatorname{ad}_{a})^{i-1}(b)\big)}_{=0} \\
&& -  \underbrace{\varepsilon^v\big((\operatorname{ad}_{a})^{i-1}(b)\big)}_{=0}\, D^v(a) - D^v\big((\operatorname{ad}_{a})^{i-1}(b)\big)\, \underbrace{\varepsilon^{u,v}(a)}_{=0} \\
&=& z  D^v\big((\operatorname{ad}_{a})^{i-1}(b)\big).
\end{eqnarray*}
{In particular, for $i=1$, we obtain  $D^v (\operatorname{ad}_{a} (b))= zD^v(b)$ which is equal to $-z$ by \eqref{eq:vvv3}.}
Thus the lemma is proved by an induction on {$i\geq 1$.}
\end{proof}

By applying the above two lemmas to \eqref{eq:formula_ell}, 
we obtain that  $\phi \in T(\!(H)\!)$  is the evaluation at ${z\in H}$ of the formal power series  
\begin{equation} \label{eq:formula_phi}
\phi(X) :=  \frac{1}{24} X - \sum_{i\geq 2}  q_i X^i \in \K[[X]].
 \end{equation}
 This provides an explicit formula for $\phi$ in terms of the coefficients  $q_2,q_3,q_4,\dots$ appearing in the expression \eqref{eq:log(Phi)} of the Drinfeld associator $\Phi$.
Since $\phi -\hat S(\phi)=1/2+s(-z)$ by Lemma~\ref{lem:vvv}, the  series $\phi(X)$ satisfies 
\begin{equation} \label{eq:phi_S}
\phi(-X) - \phi(X)-1/2 = s(X).
\end{equation}

\subsection{Proof of Corollary \ref{cor:even_case}} \label{subsec:even_case}

Assume now  that the Drinfeld associator $\Phi$ is even. 
Then  the expression \eqref{eq:log(Phi)}  only shows some Lie words in $a,b$ of \emph{even} length: 
consequently, the sum in \eqref{eq:formula_phi} is a sum over all \emph{odd} integers $i\geq 3$.
We deduce that $\phi(-X)=- \phi(X)$ and the equation \eqref{eq:phi_S} implies that
$$
\phi(X) = \frac{1}{4} +\frac{1}{2} s(-X).
$$
Thus Corollary \ref{cor:even_case} follows from Theorem \ref{th:vec_mu}.

\subsection{Proof of Corollary \ref{cor:delta}}

Recall that $\vert\! -\! \vert :\K[\pi] \to \K \vert {\check{\pi}}\vert$ is the canonical projection $\K[\pi] \to \K {\check{\pi}}$
composed with $\vert\! -\! \vert:  \K {\check{\pi}} \to  \K \vert{\check{\pi}}\vert$.
Similarly, let   $\vert\! -\! \vert: T(H) \to \vert\check{T}(H)\vert$ be the 
composition of the canonical  projections $T(H)\to \check T(H)$
and $\vert\! -\! \vert: \check T(H) \to \vert \check T(H) \vert$: by completion, we obtain a map
 $\vert\! -\! \vert: T(\!(H)\!) \to \vert\check{T}(\!(H)\!)\vert$.
Consider an $x\in {\K \vert {\check{\pi}}\vert}$ and let $y\in \K[\overrightarrow{\pi}]$ be such that $\vert p(y) \vert=x$. 
Then, 
\begin{eqnarray*}
(\theta_Z \otimes \theta_Z) \delta_{\operatorname{T}}(x) \ = \ (\theta_Z \otimes \theta_Z) \delta_{\operatorname{T}}( \vert p(y) \vert)  
& \by{eq:delta's} & (\theta_Z \otimes \theta_Z) (\vert\! -\! \vert \otimes \vert\! -\! \vert)\delta_{\vec \mu}(y) \\
&=&  (\vert\! -\! \vert \otimes \vert\! -\! \vert) (\theta_Z \otimes \theta_Z) \delta_{\vec \mu}(y) \\
&=&   (\vert\! -\! \vert \otimes \vert\! -\! \vert) \delta_{{\xi}}({\vec{\theta}}_Z(y)) 
+  (\vert\! -\! \vert \otimes \vert\! -\! \vert)  \delta_q({\vec{\theta}}_Z(y)) 
\end{eqnarray*}
where the last identity follows from Theorem \ref{th:vec_mu} and we have {set} $ q:= q_{-1/4+{\phi(z)}\, ,\, -1/4-{\phi(-z)}}$.
We claim that 
\begin{equation}\label{eq:claim1}
 {(\vert\! -\! \vert \hat\otimes \vert\! -\! \vert)}\, \delta_q =0: T(\!(H)\!) \hat\otimes \K[[C]] 
 \longrightarrow {\vert \check{T}(\!(H)\!) \vert} \hat\otimes {\vert \check{T}(\!(H)\!) \vert}
\end{equation}
and 
\begin{equation} \label{eq:claim2}
 (\vert\! -\! \vert \hat\otimes \vert\! -\! \vert)\, \delta_{{\xi}} = \hat \delta_{\operatorname{S}}\, \vert\! -\! \vert\, p: 
 T(\!(H)\!) \hat\otimes \K[[C]] \longrightarrow \vert\check{T}(\!(H)\!)\vert \hat\otimes \vert \check{T}(\!(H)\!)\vert {.}
\end{equation}
By the previous computation, these two claims imply that
$$
\forall x\in {\K \vert {\check{\pi}}\vert}, \quad
(\theta_Z \otimes \theta_Z) \delta_{\operatorname{T}}(x) = \hat \delta_{\operatorname{S}}(\vert p\, {\vec{\theta}}_Z(y)\vert) = \hat \delta_{\operatorname{S}}( \theta_Z(x)).
$$
Since  $\K \vert \check \pi\vert \simeq \vert  \K \check \pi\vert$ is dense in $\reallywidehat{ \vert \K{\check{\pi}}\vert }$, 
this will conclude the proof of Corollary \ref{cor:delta}.

We {first} prove  claim \eqref{eq:claim1}. Using   Lemma \ref{lem:delta_e1_e2}, we obtain 
$$
\delta_q(x \hat \otimes C^n) = 
\left\{\begin{array}{ll}
{0} & {\hbox{if } n>0}\\
x \big(\hat S( e') \hat \otimes e'' + e'' \hat \otimes  \hat S(e') \big) -  \big( \hat S(e') \hat \otimes   e''       + e'' \hat \otimes  \hat S(e') \big) x & \hbox{if } n=0
\end{array}\right.
$$ 
{for any $x\in T(\!(H)\!)$  and  $n\in \N$,} where $e:= {s(-z)= -1/2 + z/12-z^3/720 + \cdots }$. For any odd integer $r\geq 1$, we have 
\begin{eqnarray*}
(S \otimes \id+ \id  \otimes S)  \Delta({z^r})&=& \sum_{i=0}^r \binom{r}{i}
(S \otimes \id   + \id  \otimes S) ({z^i} \otimes {z^{r-i}})\\
&=& \sum_{i=0}^r \binom{r}{i} \left((-1)^i {z^i} \otimes {z^{r-i}}+ (-1)^{r-i} {z^i} \otimes {z^{r-i}}\right) \ = \  0
\end{eqnarray*}
which proves that $\hat S( e') \hat\otimes e'' + e'' \hat\otimes \hat S(e') {= -\frac{1}{2} (1 \otimes 1)}$. 
It follows that ${(\vert\! -\! \vert \hat\otimes \vert\! -\! \vert)}\, \delta_q=0$.

Finally, we prove  claim \eqref{eq:claim2}. Using \eqref{eq:d_q} and \eqref{eq:delta_q}, we obtain
\begin{eqnarray*}
\forall x\in T(\!(H)\!), \ \forall n\in \N, \quad
\delta_{{\xi}} (x \hat\otimes C^n) =  x' \hat S\big( ({\xi}(x'' \hat\otimes C^n))'\big)  \wedge ({\xi}(x'' \hat\otimes C^n))''.
\end{eqnarray*}
This quantity is clearly trivial if $n>1$, and it is equal to $-2 \wedge x$ if $n=1$. 
We are going to show that it is equal to $\hat \delta_{\operatorname{S}}(x)$ if $n=0$.
This will imply that
$$
\forall x\in T(\!(H)\!), \ \forall n\in \N, \quad
(\vert\! -\! \vert \hat\otimes \vert\! -\! \vert)\delta_{{\xi}} (x\hat \otimes C^n) 
= (\vert\! -\! \vert {\hat\otimes} \vert\! -\! \vert) \hat \delta_{ \operatorname{S}}(p(x\hat \otimes C^n))= \hat \delta_{\operatorname{S}}(\vert p(x\hat \otimes C^n) \vert)
$$
which will prove that $(\vert\! -\! \vert \hat \otimes \vert\! -\! \vert)\delta_{{\xi}} = \hat \delta_{\operatorname{S}} \vert\! -\! \vert p$
and will conclude the proof of Corollary~\ref{cor:delta}.
To compute $\delta_{{\xi}} (x\hat \otimes 1)$,
we can assume without loss of generality that $x= x_1 \cdots x_m$ where $x_1,\dots,x_m \in H$.
We will use the following notations:
for any subset $J$ of $[1,m]:= \{1,\dots, m\}$, we denote by $x_J$ the ordered product of the $x_j$'s for 	all $j\in J$,
we denote $J_*:= J \setminus\{\max(J)\}$ and, for any $j\in J_*$, we set  $s(j) := {\min(J \setminus (J \cap [1,j]))}$,
$L(j):= J \cap [1,j[$ and $U(j):= J \cap ]s(j),m]$. Then
\begin{eqnarray*}
\delta_{{\xi}} (x\otimes 1) &=& x' S\big( ({\xi}(x'' \otimes 1))'\big)  \wedge ({\xi}(x'' \otimes 1))'' \\
&=& \sum_{\substack{I,J \subset[1,m] \\ I\sqcup J = [1,m]}} x_I S\big( ({\xi}(x_{J} \otimes 1))'\big)  \wedge ({\xi}(x_{J} \otimes 1))''\\
  &=& \sum_{\substack{I,J \subset[1,m] \\ I\sqcup J = [1,m]}} \sum_{j \in J_* } 
x_I S\big( \big(x_{L(j)}\, (x_j {\odot} x_{s(j)})\, x_{U(j)} \big)' \big)  \wedge \big(x_{L(j)}\, (x_j {\odot} x_{s(j)})\, x_{U(j)}\big)'' \\
&=&  \sum_{\substack{I,J \subset[1,m] \\ I\sqcup J = [1,m]}} \sum_{j \in J_* } \sum_{\substack{P,Q \subset J\setminus\{j, s(j)\}\\ P \sqcup Q=  J\setminus\{j, s(j)\} }} 
  \big(x_I S\big( x_{P \cap L(j)}\, (x_j {\odot} x_{s(j)})\, x_{P\cap U(j)} \big) \big) \wedge (x_{Q \cap L(j)}\, x_{Q\cap U(j)}) \\
  && + \sum_{\substack{I,J \subset[1,m] \\ I\sqcup J = [1,m]}} \sum_{j \in J_* } \sum_{\substack{P,Q \subset J\setminus\{j, s(j)\}\\ P \sqcup Q=  J\setminus\{j, s(j)\} }} 
 \big( x_I S( x_{P \cap L(j)}\, x_{P\cap U(j)} ) \big) \wedge \big(x_{Q \cap L(j)}\, (x_j {\odot} x_{s(j)})\, x_{Q\cap U(j)}\big).
\end{eqnarray*}
These triple sums can be rearranged as follows: 
\begin{eqnarray*}
&& \delta_{{\xi}} (x\otimes 1) \\
&=& \sum_{1 \leq j<l \leq m}  \sum_{\substack{I,P,Q \subset [1,m] \setminus  \{j,l\} \\ I \sqcup P \sqcup Q = [1,m] \setminus  \{j,l\} \\ I \supset [j+1,l-1] }}
 \big( x_I S\big( x_{P \cap [1,j[}\, (x_j {\odot} x_l)\, x_{P \cap ]l,m]} \big)  \big) \wedge \big( x_{Q \cap [1,j[}\, x_{Q \cap ]l,m]} \big)\\
  &&+ \sum_{1 \leq j<l \leq m}  \sum_{\substack{I,P,Q \subset [1,m] \setminus  \{j,l\} \\ I \sqcup P \sqcup Q = [1,m] \setminus  \{j,l\} \\ I \supset [j+1,l-1] }}
  \big(x_I S\big( x_{P \cap [1,j[}\, x_{P \cap ]l,m]} \big) \big) \wedge \big( x_{Q \cap [1,j[}\, (x_j {\odot} x_l)\, x_{Q \cap ]l,m]} \big)\\
&=& \sum_{1 \leq j<l \leq m}    \sum_{\substack{I',P',Q' \subset [1,j-1]\\ I' \sqcup P' \sqcup Q' = [1,j-1] }}
\sum_{\substack{I'',P'',Q'' \subset [l+1,m]\\ I'' \sqcup P'' \sqcup Q'' = [l+1,m] }}
  \big(x_{I'} x_{[j+1,l-1]} x_{I''} S\big( x_{P'}\, (x_j {\odot} x_l)\, x_{P''} \big) \big) \wedge( x_{Q'}\, x_{Q''} )\\
  &&+ \sum_{1 \leq j<l \leq m}   \sum_{\substack{I',P',Q' \subset [1,j-1]\\ I' \sqcup P' \sqcup Q' = [1,j-1] }}
\sum_{\substack{I'',P'',Q'' \subset [l+1,m]\\ I'' \sqcup P'' \sqcup Q'' = [l+1,m] }}
 \big( x_{I'} x_{[j+1,l-1]} x_{I''} S\big( x_{P'}\, x_{P''} \big)\big)  \wedge \big( x_{Q'}\, (x_j{\odot} x_l)\, x_{Q''} \big)\\  
  &=&- \sum_{1 \leq j<l \leq m}    \sum_{\substack{I',P',Q' \subset [1,j-1]\\ I' \sqcup P' \sqcup Q' = [1,j-1] }}
\sum_{\substack{I'',P'',Q'' \subset [l+1,m]\\ I'' \sqcup P'' \sqcup Q'' = [l+1,m] }}
  \big(x_{I'} x_{[j+1,l-1]} x_{I''} S(x_{P''}) (x_j {\odot} x_l) S(x_{P'})  \big)  \wedge( x_{Q'}\, x_{Q''} )\\
  &&+ \sum_{1 \leq j<l \leq m}   \sum_{\substack{I',P',Q' \subset [1,j-1]\\ I' \sqcup P' \sqcup Q' = [1,j-1] }}
\sum_{\substack{I'',P'',Q'' \subset [l+1,m]\\ I'' \sqcup P'' \sqcup Q'' = [l+1,m] }}
  \big( x_{I'} x_{[j+1,l-1]} x_{I''} S( x_{P''}) S (x_{P'}) \big)  \wedge\big( x_{Q'}\, (x_j {\odot} x_l)\, x_{Q''} \big)\\
  &=&- \sum_{1 \leq j<l \leq m}    \sum_{\substack{I',P',Q' \subset [1,j-1]\\ I' \sqcup P' \sqcup Q' = [1,j-1] }}
\sum_{\substack{ Q''\subset  [l+1,m] }}
  \big(x_{I'} x_{[j+1,l-1]} \varepsilon(x_{[l+1,m]\setminus Q''}) (x_j {\odot} x_l)\, S(x_{P'}) \big)  \wedge( x_{Q'}\, x_{Q''} )\\
  &&+ \sum_{1 \leq j<l \leq m}   \sum_{\substack{I',P',Q' \subset [1,j-1]\\ I' \sqcup P' \sqcup Q' = [1,j-1] }}
\sum_{\substack{ Q''\subset  [l+1,m] }}
  \big( x_{I'} x_{[j+1,l-1]}  \varepsilon(x_{[l+1,m]\setminus Q''}) S (x_{P'}) \big)  \wedge \big( x_{Q'}\, (x_j {\odot} x_l)\, x_{Q''} \big)\\
   &=&- \sum_{1 \leq j<l \leq m}    \sum_{\substack{I',P',Q' \subset [1,j-1]\\ I' \sqcup P' \sqcup Q' = [1,j-1] }}
  \big(x_{I'} x_{[j+1,l-1]}  (x_j {\odot} x_l)\, S(x_{P'}) \big)  \wedge( x_{Q'}\, x_{[l+1,m]} )\\
  &&+ \sum_{1 \leq j<l \leq m}   \sum_{\substack{I',P',Q' \subset [1,j-1]\\ I' \sqcup P' \sqcup Q' = [1,j-1] }}
  \big( x_{I'} x_{[j+1,l-1]}  S (x_{P'}) \big)  \wedge\big( x_{Q'}\, (x_j {\odot} x_l)\, x_{[l+1,m]} \big)\\
     &=&- \sum_{1 \leq j<l \leq m}    \sum_{\substack{I',P',Q' \subset [1,j-1]\\ I' \sqcup P' \sqcup Q' = [1,j-1] }}
  \big(x_{[j+1,l-1]} (x_j {\odot} x_l)\, S(x_{P'})x_{I'}  \big)  \wedge( x_{Q'}\, x_{[l+1,m]} )\\
  &&+ \sum_{1 \leq j<l \leq m}   \sum_{\substack{I',P',Q' \subset [1,j-1]\\ I' \sqcup P' \sqcup Q' = [1,j-1] }}
  \big( x_{[j+1,l-1]}  S (x_{P'}) x_{I'}  \big)  \wedge\big( x_{Q'}\, (x_j {\odot} x_l)\, x_{[l+1,m]} \big)\\
   &=&- \sum_{1 \leq j<l \leq m}    \sum_{\substack{ Q' \subset  [1,j-1] }}
  \big(x_{[j+1,l-1]} (x_j {\odot} x_l)\, \varepsilon(x_{[1,j-1] \setminus Q'})  \big)  \wedge( x_{Q'}\, x_{[l+1,m]} )\\
  &&+ \sum_{1 \leq j<l \leq m}   \sum_{\substack{ Q' \subset [1,j-1] }}
  \big( x_{[j+1,l-1]}  \varepsilon(x_{[1,j-1] \setminus Q'})  \big)  \wedge\big( x_{Q'}\, (x_j {\odot} x_l)\, x_{[l+1,m]} \big)\\
   &=&- \sum_{1 \leq j<l \leq m}  
  \big(x_{[j+1,l-1]} (x_j {\odot} x_l)\,   \big)  \wedge( x_{[1,j-1]}\, x_{[l+1,m]} )\\
  &&+ \sum_{1 \leq j<l \leq m}  
   x_{[j+1,l-1]}     \wedge   \big( x_{[1,j-1]}\, (x_j {\odot} x_l)\, x_{[l+1,m]} \big)
\ = \ \delta_{\operatorname{S}} (x_1 \cdots x_m) \ = \ \delta_{\operatorname{S}} (x).
\end{eqnarray*}

\subsection{Remarks.} \label{subsec:Gamma}
{1)} Let $\Phi$ be an arbitrary associator
whose logarithm is written in the form~\eqref{eq:log(Phi)}.
It follows from  \eqref{eq:formula_phi} that
\begin{eqnarray*}
\phi(X) - \phi(-X)  &=&   \frac{1}{12} X -  \sum_{j\geq 1}  2q_{2j+1}\, X^{2j+1}.
\end{eqnarray*}
Besides, we have
$$
\phi(X) - \phi(-X)  \by{eq:phi_S}  -  s(X) - 1/2  \by{eq:s(X)}  \sum_{k\geq 1} \frac{B_{2k}}{(2k)!} X^{2k-1}.
$$
{Thus we obtain the following identity, relating some coefficients of $\Phi$ to Bernoulli numbers:}
\begin{equation}    \label{eq:Bernoulli}
\forall j\geq 1, \quad q_{2j+1} = - \frac{B_{2j+2}}{2\, (2j+2)!} \ \in \Q \subset \K.
\end{equation}

{2) The identity  \eqref{eq:Bernoulli}  also follows from a result of Enriquez \cite[Corollary 0.4]{En_Gamma},
which generalizes a result of Drinfeld for the KZ associator \cite[Equation (2.15)]{Dr2} 
and is also contained in an unpublished work of Deligne and Terasoma.
As a matter of fact,  the  series $\phi(X)\in \K[[X]]$ is essentially  Enriquez's $\Gamma$-function.

To be more specific, let $\K\langle\! \langle A,B \rangle \! \rangle$ be the complete associative algebra freely generated by  $A$ and $B$,
and let $(-)^{\operatorname{op}}$ be the unique anti-homomorphism of algebras defined by $A^{\operatorname{op}} := A $ and $B^{\operatorname{op}} := B$.
Then, for any ``associator'' $\Psi \in \K\langle\! \langle A,B \rangle \! \rangle$ of ``parameter'' $\lambda\in \K\setminus\{0\}$ in the sense of~\cite{En_Gamma},
$$
\Phi := \Psi^{\operatorname{op}} \Big(\frac{t_{12}}{\lambda},\frac{t_{23}}{\lambda} \Big) \in  \A(\uparrow_1 \uparrow_2 \uparrow_3)
$$
is an ``associator'' with our conventions. Recall that the ``{$\Gamma$-function}'' of $\Psi$ is a series  
$$
\Gamma_\Psi(u) = \exp\Big( - \sum_{n\geq 2} \zeta_\Psi(n) \frac{u^n}{n}\Big) \in \K[[u]]
$$
where $(\zeta_\Psi(n))_{n\geq 2}$  is a sequence of elements of $\K$ satisfying $\zeta_\Psi(n) = -\lambda^n B_n/(2\, n!)$ for $n$ even.
By comparing the expression \eqref{eq:log(Phi)} of $\Phi$ to the definitions of \cite{En_Gamma}, it can be verified that }
$$
\forall n\geq 2, \quad   {\zeta_\Psi(n+1)} = (-\lambda)^{n+1} q_n.
$$

\appendix

\section{Formal description of Turaev's intersection pairing}

\label{sec:hip_special}

This appendix, which was a part of  \cite{MT_draft}, is aimed at proving Theorem \ref{th:special}.
In this appendix, the ground ring is a commutative field $\K$ of characteristic zero.

\subsection{Symplectic expansions} \label{subsec:symplectic}

We recall from \cite{Ma} the notion of ``symplectic expansion.''
Let $\Sigma^+$ be a compact connected oriented surface of genus $p\geq 1$ with one boundary component,
and set $\pi^+:= \pi_1(\Sigma^+, \ast)$ where $\ast \in \partial \Sigma^+$.
Let $H^+:= H_1(\Sigma^+;\K)$, denote by $T(H^+)$ the  tensor algebra over $H^+$ and denote by $T(\!(H^+)\!)$ the degree-completion of $T(H^+)$.

A \emph{symplectic expansion} of $\pi^+$ is a  map $\theta^+: \pi^+ \to T(\!(H^+)\!)$ with the following properties:
\begin{enumerate}
\item[(i)] for all $x,y\in \pi^+$, $\theta^+(xy) = \theta^+(x)\, \theta^+(y)$;
\item[(ii)] for all $x\in \pi^+$,  $\theta^+(x)$ is group-like;
\item[(iii)] for all $x\in {\pi^+}$,   $\theta^+(x)= 1 + [x] + (\deg \geq 2)$  
where $[x]\in {H^+} \simeq {(\pi^+/[\pi^+,\pi^+])}\otimes_\Z \K$;
\item[(iv)]   $\theta^+(\nu)=\exp(-\omega)$.
\end{enumerate}
In condition (iv), $\nu \in \pi^+$ is the homotopy class of the oriented curve ${\partial \Sigma^+}$,
while   $\omega \in \Lambda^2 H^+$ is the degree $2$ element of $T(H^+)$
corresponding to the homology intersection form $\omega\in \Lambda^2 \Hom(H^+,\K)$ 
through the isomorphism $H^+ \to \Hom(H^+,\K), x \mapsto \omega(x,-)$.
Because $\pi$ is a free group, it is not difficult to construct a symplectic expansion of $\pi$ 
 by proceeding by successive finite-degree approximations. See  \cite[Lemma 2.16]{Ma}.

Let $ \widehat{\K[\pi^+]}$ denote the completion of $\K[\pi^+]$ with respect to the $I$-adic filtration.
The conditions (i), (ii) and (iii) imply that $\theta^+$ induces an isomorphism of complete Hopf algebras 
$$
\hat \theta^+: \widehat{\K[\pi^+]} \stackrel{\simeq}{\longrightarrow} T(\!(H^+)\!)
$$
which, at the level of graded Hopf algebras, gives the canonical isomorphism \eqref{eq:canonical_alg_iso}.  

\subsection{Formal description of Turaev's intersection pairing in the symplectic case}

We use the same notations as in Section~\ref{subsec:symplectic},
and recall from \cite{MT_twists} how symplectic expansions 
provide an algebraic description of the homotopy intersection pairing $\eta^+$ of the surface $\Sigma^+$.

Consider the operation $\stackrel{\omega}{\leadsto}$ in $T(H^+)$ defined by
$(x\stackrel{\omega}{\leadsto} 1) = (1 \stackrel{\omega}{\leadsto} x) :=0$ for any $x\in T(H^+)$ and  by
$$
\left(h_1 \cdots  h_m \stackrel{\omega}{\leadsto} k_1 \cdots  k_n\right)
:= \omega(h_m, k_1)\  h_1 \cdots  h_{m-1}  k_2  \cdots  k_n
$$
for any integers $m,n \geq 1$ and for any $h_1,\dots,h_m,k_1,\dots,k_n \in H^+$.
This operation extends to a filtered Fox pairing
$$
(- \stackrel{\omega}{\leadsto} -): T(\!(H^+)\!) \times T(\!(H^+)\!) \longrightarrow T(\!(H^+)\!)
$$
in the complete Hopf algebra $T(\!(H^+)\!)$.
Moreover, the formal power series $s(X)\in \Q[[X]]$ defined by \eqref{eq:s(X)} can be evaluated at $\omega$
to get an element $s(\omega) \in T(\!(H^+)\!)$, so that 
we can also consider the inner Fox pairing $\rho_{s(\omega)}$ in $T(\!(H^+)\!)$. 
Then, according to \cite[Theorem 10.4]{MT_twists},
the diagram 
\begin{equation}   \label{eq:tensorial_eta_symplectic}
\xymatrix{
\widehat{\K[\pi^+]} \times \widehat{\K[\pi^+]}  \ar[d]_-{\hat \theta^+ \times \hat \theta^+}^-\simeq \ar[rr]^-{\hat \eta^+} 
&& \widehat{\K[\pi^+]} \ar[d]^-{\hat \theta^+}_-\simeq \\
T(\!(H^+)\!) \times T(\!(H^+)\!)  \ar[rr]_-{(-\, \stackrel{\omega}{\leadsto}\, - )+ \rho_{s(\omega)}} && T(\!(H^+)\!)
}
\end{equation}
is commutative for any symplectic expansion $\theta^+$ of $\pi^+$.
This implies a prior result of Kawazumi and Kuno,
which provides  a formal description of the Goldman bracket  $\langle - , - \rangle_{\operatorname{G}}$ for the surface~$\Sigma^+$ \cite{KK_twists} .

\subsection{From special expansions to symplectic expansions} \label{subsec:special_to_symplectic}

Let $\Sigma$ be a disk with finitely many punctures numbered from $1$ to $p$.
We set $\pi:=\pi_1(\Sigma,\ast)$ where $\ast \in \partial \Sigma$ and, 
for any $i\in\{1,\dots,p\}$, let $\check \zeta_i$ be the conjugacy class in $\pi$
that is defined by a small counter-clockwise loop around the $i$-th puncture.
Let $H:=H_1(\Sigma;\K)$ and set $z_i:= [\check \zeta_i]\in H$ for all $i\in \{1,\dots,p\}$.
Recall the notion of ``special expansion'' $\theta: \pi \to T(\!(H)\!)$ given  in Section~\ref{subsec:unframed_special}. 

\begin{lemma} \label{lem:special_to_symplectic}
Let $\Sigma^\circ\subset \Sigma$ be a disk with $p$ holes onto which $\Sigma$ deformation retracts,
and let~$\Sigma^+$ be the compact connected oriented surface of genus $p$ with one boundary component
that is obtained from $\Sigma^\circ$ by gluing a one-hole torus to each boundary component not containing $\ast$.
Let $\pi^+ :=\pi_1(\Sigma^+,\ast)$ and $H^+:= H_1(\Sigma^+;\K)$. Then, for any special expansion $\theta$ of $\pi$,
there exists a symplectic expansion $\theta^+$ of $\pi^+$ such that the following diagram is commutative:
\begin{equation}\label{eq:theta_theta^+}
\xymatrix{
\pi \ar[d]_-{\iota} \ar[r]^-{\theta} & T(\!(H)\!)  \ar[d]^-{I}\\
\pi^+ \ar[r]_-{\theta^+} & T(\!(H^+)\!)
}
\end{equation}
Here $\iota$ is the group homomorphism  induced by the inclusion $ \Sigma \simeq \Sigma^\circ \subset \Sigma^+$,
and $I$ is the complete algebra homomorphism mapping $-z_i\in H$ 
to the homology intersection form  of the $i$-th one-hole torus, which we regard as  a degree $2$ element of $T(\!(H^+)\!)$.
\end{lemma}

\begin{proof}
Consider the  basis $(\zeta_1,\dots, \zeta_p)$ of $\pi$ {given} by the following closed oriented curves in $\Sigma^\circ$:
$$
\labellist
\scriptsize\hair 2pt
 \pinlabel {$\zeta_1$} [b] at 128 350
 \pinlabel {$\zeta_p$} [b] at 443 352
  \pinlabel {$\nu$}  at 400  5
    \pinlabel {$\ast$}  [t] at 280  5
 \pinlabel {$\cdots$}  at 285 286
  \pinlabel {$\circlearrowleft$}  at 285 480
\endlabellist
\centering
\includegraphics[scale=0.15]{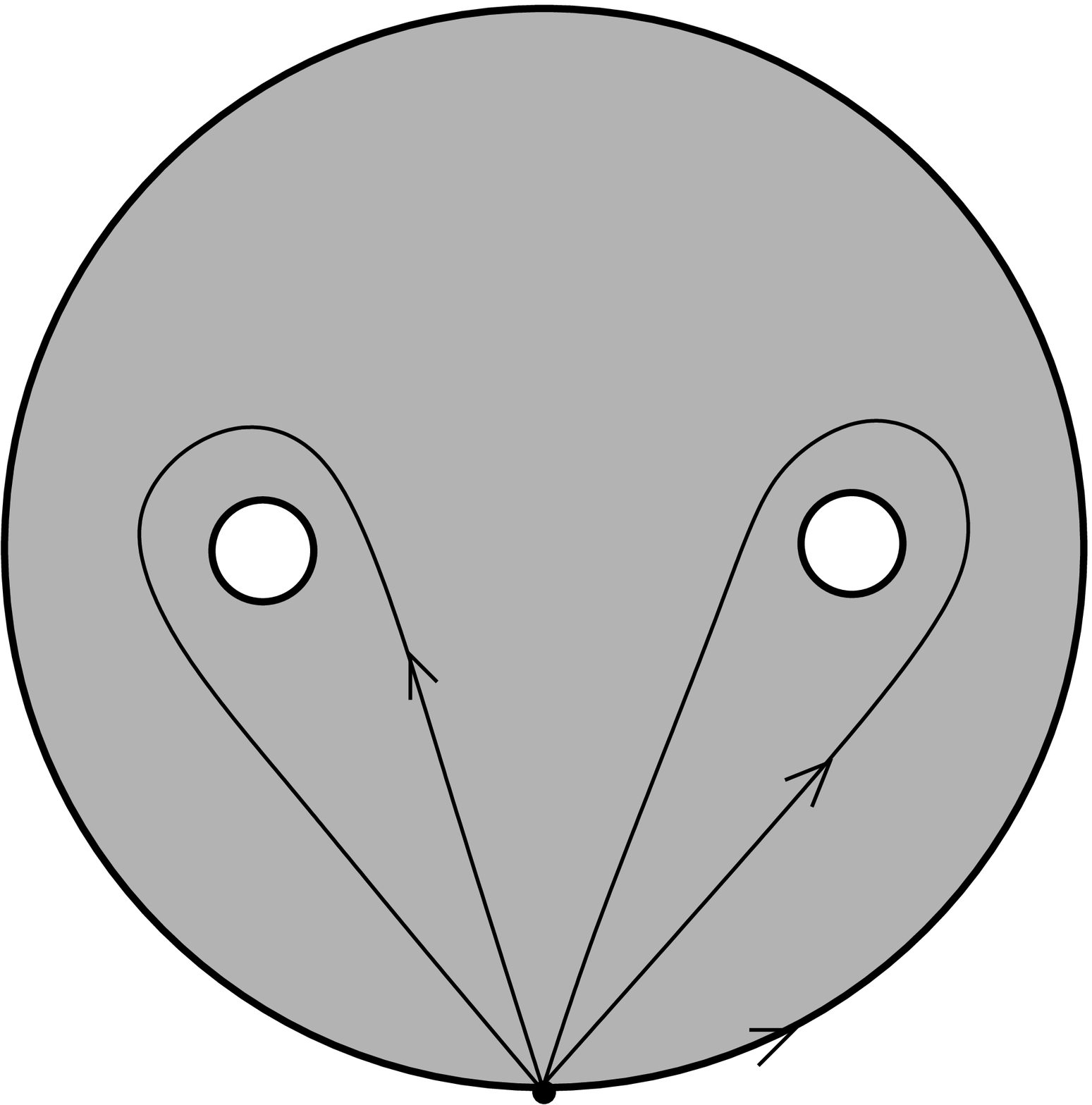}
$$
Note that $\zeta_i$ is a representative of $\check \zeta_i$ for all $i\in \{1,\dots,p\}$. 
Consider also the basis $(\alpha,\beta)$ of the fundamental group of the one-hole torus represented by the following closed oriented curves:
$$
\labellist
\scriptsize\hair 2pt
  \pinlabel {$\circlearrowleft$}  at 285 480
  \pinlabel {$\beta$} [b] at 409 270
 \pinlabel {$\alpha$} [b] at 441 398
\endlabellist
\centering
\includegraphics[scale=0.15]{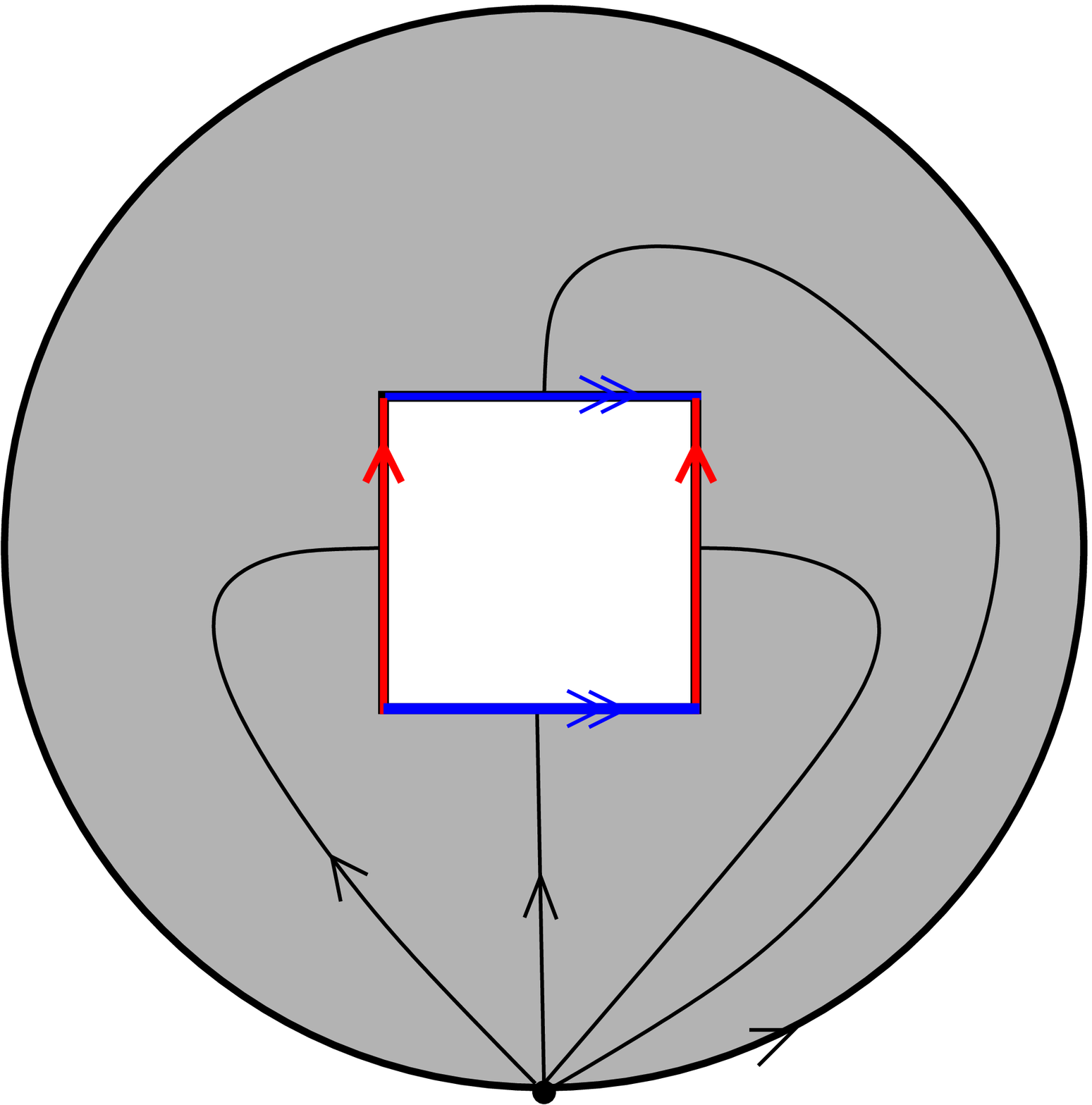}
$$
By repeating this picture $p$ times in $\Sigma^+$,
we obtain a basis $(\alpha_1,\beta_1,\dots, \alpha_p,\beta_p)$ of $\pi^+$ 
such that $\iota(\zeta_i)=[(\alpha_i)^{-1},\beta_i]$ for any $i\in \{1,\dots,p\}$. 
We set $a_i:=[\alpha_i] \in H^+$ and $b_i:=[\beta_i]\in H^+$.

{We choose a symplectic expansion $\theta': F(\alpha,\beta) \to \K\langle \! \langle a,b \rangle \! \rangle$ of 
the one-hole torus, where $a$ and $b$ denote the homology classes of $\alpha$ and $\beta$ respectively.
Thus, we have
\begin{equation} \label{eq:glike_1}
\theta'(\alpha) = \exp( a + c) \quad \hbox{and}  \quad \theta'(\beta) = \exp( b + d)
\end{equation}
where $c, d$ are series of Lie words in $a,b$ of length greater than $1$, and 
\begin{equation} \label{eq:symplectic_1}
\theta'([\alpha^{-1},\beta]) = \exp(-[a,b]).
\end{equation}
By condition (ii) of a special expansion, there exists for each $i\in \{1,\dots,p\}$ a primitive element $u_i \in T(\!(H)\!)$ such that 
\begin{equation} \label{eq:conjugacy}
\theta( \zeta_i) = \exp(u_i) \exp(z_i) \exp(-u_i).
\end{equation}
Then there is a unique multiplicative map $\theta^+: \pi^+ \to  T(\!(H^+)\!)$ such that
$$
\left\{\begin{array}{l}
\theta^+(\alpha_i) = \exp(I(u_i))\, \theta'(\alpha)\big\vert_{a \mapsto a_i, b \mapsto b_i}\exp(-I(u_i)) \\
\theta^+(\beta_i) = \exp(I(u_i))\, \theta'(\beta)\big\vert_{a \mapsto a_i, b \mapsto b_i}\exp(-I(u_i))
\end{array}\right.
$$
for all $i\in \{1,\dots,p\}$. Note that
\begin{eqnarray*}
\theta^+ \iota(\zeta_i) &=& \theta^+\big( [(\alpha_i)^{-1},\beta_i]\big) \\
&=& \theta^+(\alpha_i^{-1})\, \theta^+(\beta_i)\, \theta^+(\alpha_i)\, \theta^+(\beta_i^{-1})  \\
&=& \exp(I(u_i))\,   \theta'([\alpha^{-1},\beta])\big\vert_{a \mapsto a_i, b \mapsto b_i} \, \exp(-I(u_i)) \\
& \by{eq:symplectic_1} &  \exp(I(u_i))\,    \exp(-[a_i,b_i]) \, \exp(-I(u_i)) \\
&=&   \exp(I(u_i))\,    \exp(I(z_i)) \, \exp(-I(u_i)) \ \by{eq:conjugacy} \ I \theta(\zeta_i)
\end{eqnarray*}
which shows the commutativity of the diagram \eqref{eq:theta_theta^+}.
That $\theta^+$ satisfies the conditions (ii) and (iii) of a symplectic expansion follows easily from \eqref{eq:glike_1}.
The condition (iv) of a symplectic expansion  follows from the condition (iii) of a special expansion:
$$
\theta^+(\nu)    =   \theta^+(\iota(\nu))   \by{eq:theta_theta^+}  I \theta(\nu)   =  I \exp(z)  =  \exp I(z)  =  \exp (-\omega).
$$

\up}
\end{proof}

\subsection{Proof of Theorem \ref{th:special}}
 
We use the same notations as in Section~\ref{subsec:special_to_symplectic}.
Let $\theta:{ \pi \to T(\!(H)\!)}$ be a special expansion:
we consider the surface $\Sigma^+$
and the symplectic expansion $\theta^+: {\pi^+ \to  T(\!(H^+)\!)}$
provided by Lemma \ref{lem:special_to_symplectic}.

Denote by  $\eta^+$ the homotopy intersection pairing of $\Sigma^+$,
and let $\hat \eta^+$ be the Fox pairing in  $\widehat{\K[\pi^+]}$ obtained  from $\eta^+$ by completion.
Clearly, the diagram
\begin{equation}\label{eq:eta_eta^+}
\xymatrix{
\widehat{\K[\pi]} \times \widehat{\K[\pi]} \ar[rrrr]^-{\hat \eta}\ar[d]_-{\hat\iota \times \hat\iota} & & &&
\widehat{\K[\pi]}  \ar[d]^-{ \hat \iota} \\
\widehat{\K[\pi^+]} \times \widehat{\K[\pi^+]}\ar[rrrr]_-{\hat \eta^+} & & && \widehat{\K[\pi^+]}\\
}
\end{equation}
is commutative. Moreover, for any $i,j\in \{1,\dots,p\}$,
\begin{eqnarray*}
I(z_i) \stackrel{\omega}{\leadsto} I(z_j)&=&  [a_i,b_i]  \stackrel{\omega}{\leadsto} [a_j,b_j] \\
&=& \delta_{ij} \big( (a_ib_i)\stackrel{\omega}{\leadsto} (a_ib_i) + (b_ia_i) \stackrel{\omega}{\leadsto} (b_ia_i)\big) \\
&=& \delta_{ij} (-a_ib_i + b_i a_i) \ = \ \delta_{ij} I(z_i) \ = \ I (z_i {\odot} z_j) 
\end{eqnarray*}
which implies that $(-\stackrel{\omega}{\leadsto}-) \circ (I \times I) = I \circ (-{\odot}-)$.
Using the fact that $I(-z)=\omega$,  we deduce that the following diagram is commutative:
\begin{equation}\label{eq:z_omega}
\xymatrix{
T(\!(H)\!) \times T(\!(H)\!) \ar[rrrr]^-{(-\, {\odot}\, - )+ \rho_{s(-z)}}\ar[d]_-{I \times I} & & &&
T(\!(H)\!)  \ar[d]^-{ I} \\
T(\!(H^+)\!) \times T(\!(H^+)\!)\ar[rrrr]_-{(-\, \stackrel{\omega}{\leadsto}\, - )+ \rho_{s(\omega)}} & & && T(\!(H^+)\!)\\
}
\end{equation}
Now, for any $x,y\in \widehat{\K[\pi]}$, we have
\begin{eqnarray*}
I \hat\theta \hat \eta(x,y)&\stackrel{\eqref{eq:theta_theta^+}}{=}& \hat \theta^+ \hat \iota   \hat \eta(x,y)\\
&\stackrel{\eqref{eq:eta_eta^+}}{=}& \hat \theta^+   \hat \eta^+\big(\hat \iota (x),\hat \iota (y)\big)\\
&\stackrel{\eqref{eq:tensorial_eta_symplectic}}{=}&  
\big(\hat \theta^+ \hat \iota (x)\big) \stackrel{\omega}{\leadsto} \big(\hat \theta^+ \hat \iota (y)\big)  
+ \rho_{s(\omega)} \big(\hat \theta^+ \hat \iota (x),\hat \theta^+ \hat \iota (y)\big)\\
&\stackrel{\eqref{eq:theta_theta^+}}{=}& \big(I \hat \theta (x)\big) \stackrel{\omega}{\leadsto} \big(I \hat \theta (y)\big)  
+ \rho_{s(\omega)} \big(I \hat \theta (x),I \hat \theta (y)\big) \\
&\stackrel{\eqref{eq:z_omega}}{=}& I\left(\hat \theta (x)  {\odot}  \hat \theta (y)
+ \rho_{s(-z)} \big( \hat \theta (x), \hat \theta (y)\big)\right).
\end{eqnarray*}
Since the map  $I: T(\!(H)\!) \to T(\!(H^+)\!)$ is injective, 
we conclude that the diagram \eqref{eq:tensorial_eta} is commutative.

%
%

%
%
%

\begin{thebibliography}{CJKLS}


\bibitem[AET]{AET}
A. Alekseev, B. Enriquez, C. Torossian,
\emph{Drinfeld associators, braid groups and explicit solutions of the Kashiwara--Vergne equations.}
Publ. Math. Inst. Hautes \'Eƒtudes Sci. 112 (2010),  143--189. 

{
\bibitem[AKKN]{AKKN}
A. Alekseev, N. Kawazumi, Y. Kuno, F. Naef, 
\emph{The Goldman--Turaev Lie bialgebra in genus zero and the Kashiwara--Vergne problem.} In preparation.
}


\bibitem[AT]{AT}
A. Alekseev, C. Torossian,
\emph{The Kashiwara--Vergne conjecture and Drinfeld's associators.}
Ann. of Math. (2) 175 (2012), no$.$~2, 415--463.

\bibitem[BN1]{BN1}
D. Bar-Natan, 
\emph{On the Vassiliev knot invariants}.
Topology 34 (1995), no$.$~2, 423--472.

\bibitem[BN2]{BN2}
D. Bar-Natan, 
\emph{Vassiliev homotopy string link invariants.} 
J. Knot Theory Ramifications 4 (1995), no$.$~1, 13--32.

\bibitem[BN3]{BN3}
D. Bar-Natan, 
\emph{Vassiliev and quantum invariants of braids.} 
The interface of knots and physics (San Francisco, CA, 1995), 129--144, Proc. Sympos. Appl. Math., 51, Amer. Math. Soc., Providence, RI, 1996.

\bibitem[BN4]{BN4}
D. Bar-Natan, 
\emph{Non-associative tangles.}  
Geometric topology (Athens, GA, 1993), 139--183,
AMS/IP Stud. Adv. Math., 2.1, Amer. Math. Soc., Providence, RI, 1997.

\bibitem[BLB]{BLB}
R. Bocklandt, L. Le Bruyn,
\emph{Necklace Lie algebras and noncommutative symplectic geometry.}
Math.~Z. 240 (2002), no$.$~1, 141--167.

\bibitem[BP]{BP}
Y. Burman, M. Polyak,
\emph{Whitney's formulas for curves on surfaces.} 
Geom. Dedicata 151 (2011), 97--106.

\bibitem[CHM]{CHM}
D. Cheptea, K. Habiro, G. Massuyeau,
\emph{A functorial LMO invariant for Lagrangian cobordisms.}
Geom. Topol.  12 (2008),  no$.$~2, 1091--1170.

\bibitem[Dr1]{Dr1}
V. Drinfeld,
\emph{Quasi-Hopf algebras}. 
Leningrad Math. J. 1  (1990),  no$.$~6, 1419--1457.

\bibitem[Dr2]{Dr2}
V. Drinfeld,
\emph{On quasitriangular quasi-Hopf algebras and a group closely connected
with $\operatorname{Gal}(\overline{\mathbb Q}/\mathbb{Q})$.}
Leningrad Math. J. 2 (1991), no$.$~4, 829--860.

\bibitem[En1]{En_Gamma}
B. Enriquez,
\emph{On the Drinfeld generators of $\mathfrak{grt}_1(\mathbf{k})$ and  $\Gamma$-functions for associators.}
 Math. Res. Lett.~3  (2006),  no$.$~2-3, 231--243.

\bibitem[En2]{En_elliptic}
B. Enriquez, 
\emph{Elliptic associators.} Selecta Math. (N.S.) 20 (2014), no$.$~2, 491--584. 

\bibitem[Gi]{Gi}
V. Ginzburg,
\emph{Non-commutative symplectic geometry, quiver varieties, and operads.}
Math. Res. Lett.~8 (2001), no$.$~3, 377--400.

\bibitem[Go]{Go}
W. Goldman,
\emph{Invariant functions on Lie groups and Hamiltonian flows of surface group representations.}
Invent. Math. 85 (1986), no$.$~2, 263--302.

\bibitem[HgM]{HgM}
N. Habegger, G. Masbaum,
\emph{The Kontsevich integral and Milnor's invariants.}
Topology 39 (2000), no$.$~6, 1253--1289. 

\bibitem[Ha]{Ha}
K. Habiro,
\emph{Claspers and finite type invariants of links.}
 Geom. Topol.  4  (2000), 1--83.

\bibitem[Hu]{Hu}
P. Humbert,
\emph{Int\'egrale de Kontsevich elliptique et enchev\^etrements en genre sup\'erieur.}
University of Strasbourg, tel-00762209 (December 2012).

\bibitem[KT]{KT}
C. Kassel, V. Turaev,
\emph{Chord diagram invariants of tangles and graphs.}
Duke Math. J. 92 (1998), no$.$~3, 497--552. 

\bibitem[Kt]{Katz}
R. Katz,
\emph{Elliptic associators and the LMO functor.}
{J. Knot Theory Ramifications 25 (2016), no$.$ 1, 1650002, 71 pp.}

\bibitem[Kw1]{Ka_announcement}
N. Kawazumi,
\emph{A regular homotopy version of the Goldman--Turaev Lie bialgebra, the Enomoto--Satoh traces and the divergence cocycle in the Kashiwara--Vergne problem}.
arXiv:1406.0056 (June 2014).

\bibitem[Kw2]{Ka_genus_0}
N. Kawazumi,
\emph{A tensorial description of the Turaev cobracket on genus 0 compact surfaces}.  
arXiv:1506.03174 (June 2015). 

\bibitem[KK1]{KK_twists}
N. Kawazumi, Y. Kuno,
\emph{The logarithms of Dehn twists.}
Quantum Topol. 5 (2014), no$.$~3, 347--423.

\bibitem[KK2]{KK_intersection}
N. Kawazumi, Y. Kuno,
\emph{Intersections of curves on surfaces and their applications to mapping class groups.}
{Ann. Inst. Fourier  65  (2015),  no. 6, 2711--2762.}

\bibitem[KK3]{KK_survey}
N. Kawazumi, Y. Kuno,
\emph{The Goldman--Turaev Lie bialgebra and the Johnson homomorphisms.} 
{Handbook of Teichm\"uller theory, Volume V,  98--165, EMS Publishing House, Z\"urich, 2015.}

\bibitem[Kh]{Kh}
T. Kohno,
\emph{S\'erie de Poincar\'e--Koszul associ\'ee aux groupes de tresses pures.}
Invent. Math. 82 (1985), no. 1, 57--75. 

\bibitem[Kn]{Kn}
M. Kontsevich,
\emph{Formal (non)commutative symplectic geometry.} 
The Gel'fand Mathematical Seminars, 1990--1992, 173--187, Birkh\" auser Boston, Boston, MA, 1993.

\bibitem[Ku]{Ku}
Y. Kuno. Private communication to the author (September 2015).

\bibitem[LM1]{LM1}
T.T.Q. Le, J. Murakami, 
\emph{Representation of the category of tangles by Kontsevich's iterated integral}. Comm. Math. Phys. 168 (1995), no$.$~3, 535--562.

\bibitem[LM2]{LM2}
T.T.Q. Le, J. Murakami, 
\emph{The universal Vassiliev--Kontsevich invariant for framed oriented links}.
Compositio Math. 102 (1996), no$.$~1, 41--64.

\bibitem[Ma]{Ma}
G. Massuyeau, 
\emph{Infinitesimal Morita homomorphisms and the tree-level of the LMO invariant.}
Bull. Soc. Math. France 140  (2012),  no$.$~1, 101--161.

\bibitem[MT1]{MT_twists}
G. Massuyeau, V. Turaev,
\emph{Fox pairings and generalized Dehn twists.}
 Ann. Inst. Fourier 63 (2013), no$.$~6, 2403-- 2456.

\bibitem[MT2]{MT_dim_2}
G. Massuyeau, V. Turaev,
\emph{Quasi-Poisson structures on representation spaces of surfaces.}
Int. Math. Res. Not. IMRN 2014, no$.$~1, 1--64.

\bibitem[MT3]{MT_draft}
G. Massuyeau, V. Turaev,
\emph{Tensorial description of double brackets on surface groups and related operations.}
Unpublished draft, 2012.

\bibitem[Mo]{Mo}
S. Morita, 
\emph{Abelian quotients of subgroups of the mapping class group of surfaces.} 
Duke Math. J. 70 (1993), no$.$ 3, 699--726.

\bibitem[No]{No}
Y. Nozaki,
\emph{An extension of the LMO functor}.
arXiv:1505.02545 (May 2015), {to appear in Geom. Dedicata.}

\bibitem[Oh]{Oh}
T. Ohtsuki, 
\emph{Quantum invariants. A study of knots, 3-manifolds, and their sets.} 
Series on Knots and Everything, 29. World Scientific Publishing Co., Inc., River Edge, NJ, 2002.

\bibitem[Pa]{Pa}
C. D. Papakyriakopoulos, 
\emph{Planar regular coverings of orientable closed surfaces.} 
Knots, groups, and $3$-manifolds (Papers dedicated to the memory of R. H. Fox), 
261--292. Ann. of Math. Studies, No. 84, Princeton Univ. Press, Princeton, N.J., 1975.

\bibitem[Pe]{Pe}
B. Perron,
\emph{A homotopic intersection theory on surfaces: applications to mapping class group and braids.}
Enseign. Math. (2) 52 (2006), no$.$~1-2, 159--186. 

\bibitem[Sc]{Sc_necklace}
T. Schedler, 
\emph{A Hopf algebra quantizing a necklace Lie algebra canonically associated to a quiver.}
Int. Math. Res. Not. 2005, no$.$~12, 725--760. 

\bibitem[Tu1]{Tu_loops}
V.  Turaev,
\emph{Intersections of loops in two-dimensional manifolds.}
(Russian) Mat. Sb. 106(148) (1978), no$.$~4,  566--588.
English translation: Math. USSR--Sb. 35 (1979), 229--250.

\bibitem[Tu2]{Tu_skein}
V. Turaev, 
\emph{Skein quantization of Poisson algebras of loops on surfaces.}
 Ann. Sci. \'Ecole Norm. Sup. (4)  24  (1991),  no$.$~6, 635--704.
 
\bibitem[VdB]{VdB}
M. Van den Bergh,
\emph{Double Poisson algebras.}
Trans. Amer. Math. Soc. 360 (2008), no$.$~11, 5711--5769.


\end{thebibliography}
\end{document}